\documentclass[a4paper]{article}

\usepackage[utf8]{inputenc}
\usepackage[T1]{fontenc}
\usepackage{textcomp}
\usepackage[british]{babel}

\usepackage{amsfonts,amssymb,amsthm,amsmath}

\usepackage[a4paper]{geometry}
\usepackage{xcolor}

\usepackage{graphicx}

\definecolor{bleu_sombre}{rgb}{0,0,0.4}
\definecolor{vert_sombre}{rgb}{0,0.3,0}
\definecolor{rouge_sombre}{rgb}{0.3,0,0}
\definecolor{mygray}{rgb}{0.925,0.925,0.925}
\definecolor{mymauve}{rgb}{0.58,0,0.82}
\definecolor{mygreen}{rgb}{0,0.6,0}

\usepackage[hyphens]{url}

\usepackage[plainpages=false,colorlinks,linkcolor=bleu_sombre,citecolor=rouge_sombre,urlcolor=vert_sombre,breaklinks,bookmarks=true,bookmarksnumbered=true,pdfpagelabels]{hyperref}

\usepackage[capitalize,noabbrev]{cleveref}

\newtheorem{theorem}{Theorem}
\newtheorem{lemma}[theorem]{Lemma}
\newtheorem{definition}[theorem]{Definition}
\newtheorem{remark}[theorem]{Remark}

\newtheorem{corollary}[theorem]{Corollary}

\numberwithin{theorem}{section}
\numberwithin{equation}{section}
\numberwithin{table}{section}
\numberwithin{figure}{section}

\usepackage{siunitx}

\usepackage{empheq}

\usepackage{booktabs}
\usepackage{multirow}
\usepackage{diagbox}
\usepackage{colortbl}

\usepackage{csquotes}

\usepackage{caption}
\usepackage{subcaption}

\usepackage{placeins}

\newcommand{\nc}{\newcommand}

\makeatletter
\newcommand\RedeclareMathOperator{%
  \@ifstar{\def\rmo@s{m}\rmo@redeclare}{\def\rmo@s{o}\rmo@redeclare}%
}
\newcommand\rmo@redeclare[2]{%
  \begingroup \escapechar\m@ne\xdef\@gtempa{{\string#1}}\endgroup
  \expandafter\@ifundefined\@gtempa
     {\@latex@error{\noexpand#1undefined}\@ehc}%
     \relax
  \expandafter\rmo@declmathop\rmo@s{#1}{#2}}
\newcommand\rmo@declmathop[3]{%
  \DeclareRobustCommand{#2}{\qopname\newmcodes@#1{#3}}%
}
\@onlypreamble\RedeclareMathOperator
\makeatother

\nc{\bfx}{\mathbf{x}} 
\nc{\bfy}{\mathbf{y}} 
\nc{\bfz}{\mathbf{z}} 
\nc{\bfu}{\mathbf{u}} 
\nc{\bfv}{\mathbf{v}} 
\nc{\bfw}{\mathbf{w}} 
\nc{\bft}{\mathbf{t}} 
\nc{\bfb}{\mathbf{b}} 
\nc{\bfn}{\mathbf{n}} 
\nc{\bfr}{\mathbf{r}} 
\nc{\bfc}{\mathbf{c}} 

\nc{\bfA}{\mathbf{A}} 
\nc{\bfB}{\mathbf{B}} 
\nc{\bfC}{\mathbf{C}} 
\nc{\bfR}{\mathbf{R}} 
\nc{\bfD}{\mathbf{D}} 
\nc{\bfI}{\mathbf{I}} 
\nc{\bfM}{\mathbf{M}} 
\nc{\bfK}{\mathbf{K}} 
\nc{\bfP}{\mathbf{P}} 
\nc{\bfU}{\mathbf{U}} 
\nc{\bfZ}{\mathbf{Z}} 
\nc{\bfV}{\mathbf{V}} 
\nc{\bfE}{\mathbf{E}} %
\nc{\bfH}{\mathbf{H}} %
\nc{\bfX}{\mathbf{X}} %
\nc{\bfId}{\mathbf{I}_{\mathrm{d}}} 

\nc{\bbN}{\mathbb{N}} 
\nc{\bbR}{\mathbb{R}} 
\nc{\bbP}{\mathbb{P}} 
\nc{\bbC}{\mathbb{C}} 
\nc{\wH}{\widetilde{H}}
\nc{\calS}{\mathcal{S}} %
\nc{\calD}{\mathcal{D}} %
\nc{\calN}{\mathcal{N}} %
\nc{\calT}{\mathcal{T}} %
\nc{\calR}{\mathcal{R}} %
\nc{\calP}{\mathcal{P}} %
\nc{\calV}{\mathcal{V}} %
\nc{\calW}{\mathcal{W}} %
\nc{\calJ}{\mathcal{J}} %
\nc{\calE}{\mathcal{E}} %
\nc{\calA}{\mathcal{A}} %
\nc{\calU}{\mathcal{U}} %
\nc{\calK}{\mathcal{K}} %
\nc{\calL}{\mathcal{L}} %

\nc{\rouge}{\color{red}}
\nc{\bleu}{\color{blue}}
\nc{\cyan}{\color{cyan}}
\nc{\noir}{\color{black}\rm}

\nc\dif{\mathop{}\!\mathrm{d}}   
\DeclareMathOperator{\supp}{supp}   
\DeclareMathOperator{\diam}{diam} 

\newcommand{\vertiii}[1]{{\left\vert\kern-0.25ex\left\vert\kern-0.25ex\left\vert #1 
    \right\vert\kern-0.25ex\right\vert\kern-0.25ex\right\vert}} 
\RedeclareMathOperator{\Re}{Re} 
\RedeclareMathOperator{\Im}{Im} 

\nc{\CC}{{C\nolinebreak[4]\hspace{-.05em}\raisebox{.4ex}{\tiny\bf ++}}\ }

\usepackage[show]{ed}

\newtheorem{assumption}[theorem]{Assumption}

\newcommand{\R}{\mathbb{R}}

\newcommand{\cF}{{\cal F}}

\newcommand{\cR}{{\cal R}}
\newcommand{\cS}{{\cal S}}

\newcommand{\cK}{{\cal K}}
\newcommand{\cD}{{\cal D}}

\newcommand{\bx}{x}
\newcommand{\by}{y}
\newcommand{\ba}{\hat{a}}

\newcommand{\bze}{0}

\newcommand{\re}{{\rm e}}
\newcommand{\ri}{{\rm i}}
\newcommand{\rd}{{\rm d}}

\newcommand{\e}{\epsilon}

\newcommand{\beq}{\begin{equation}}
\newcommand{\eeq}{\end{equation}}
\newcommand{\beqs}{\begin{equation*}}
\newcommand{\eeqs}{\end{equation*}}
\newcommand{\bit}{\begin{itemize}}
\newcommand{\eit}{\end{itemize}}
\newcommand{\ben}{\begin{enumerate}}
\newcommand{\een}{\end{enumerate}}
\newcommand{\bal}{\begin{align}}
\newcommand{\eal}{\end{align}}
\newcommand{\bals}{\begin{align*}}
\newcommand{\eals}{\end{align*}}
\newcommand{\bse}{\begin{subequations}}
\newcommand{\ese}{\end{subequations}}
\newcommand{\bpr}{\begin{proposition}}
\newcommand{\epr}{\end{proposition}}
\newcommand{\bre}{\begin{remark}}
\newcommand{\ere}{\end{remark}}
\newcommand{\bpf}{\begin{proof}}
\newcommand{\epf}{\end{proof}}
\newcommand{\ble}{\begin{lemma}}
\newcommand{\ele}{\end{lemma}}
\newcommand{\bco}{\begin{corollary}}
\newcommand{\eco}{\end{corollary}}
\newcommand{\bex}{\begin{example}}
\newcommand{\eex}{\end{example}}
\newcommand{\bth}{\begin{theorem}}
\newcommand{\enth}{\end{theorem}}

\newcommand{\Rea}{\mathbb{R}}
\newcommand{\Com}{\mathbb{C}}

\newcommand{\Oi}{{\Omega^-}}

\newcommand{\Oe}{{\Omega^+}}

\newcommand{\eps}{\varepsilon}

\newcommand{\pdiff}[2]{\frac{\partial #1}{\partial #2}}

\newcommand{\half}{\frac{1}{2}}

\newcommand{\LtG}{{L^2(\bound)}}

\newcommand{\LtGt}{{\LtG\rightarrow \LtG}}

\newcommand{\HhG}{{H^{1/2}(\Gamma)}}
\newcommand{\HmhG}{{H^{-1/2}(\Gamma)}}

\newcommand{\HoG}{H^1(\Gamma)}

\newcommand{\tendi}{\rightarrow \infty}
\newcommand{\tendo}{\rightarrow 0}

\def\XXint#1#2#3{{\setbox0=\hbox{$#1{#2#3}{\int}$}
     \vcenter{\hbox{$#2#3$}}\kern-.5\wd0}}

\newcommand*{\N}[1]{\left\|#1\right\|}

\allowdisplaybreaks[4]

\newcommand{\tfa}{\text{ for all }}
\newcommand{\tfor}{\text{ for }}

\newcommand{\tin}{\text{ in }}
\newcommand{\ton}{\text{ on }}
\newcommand{\tas}{\text{ as }}
\newcommand{\tand}{\text{ and }}
\newcommand{\tst}{\text{ such that }}

\newcommand{\NtD}{P^+_{\rm NtD}}
\newcommand{\ItDS}{P^{-,\eta, R}_{\rm ItD}}
\newcommand{\ItDSp}{P^{-,\eta, R'}_{\rm ItD}}
\newcommand{\ItD}{P^{-,\eta}_{\rm ItD}}
\newcommand{\ItDext}{P^{+,\beta}_{\rm ItD}}

\newcommand{\bound}{\Gamma}

\usepackage{soul}

\definecolor{jwcol}{RGB}{27, 137, 18}  

\definecolor{dalcol}{rgb}{0.8,0,0}

\definecolor{jeffColor}{RGB}{102, 0, 204}

\definecolor{escol}{rgb}{0,0,0.8}
\definecolor{estcol}{rgb}{0,0.5,0}
\definecolor{esnewcol}{rgb}{0,0.5,0}

\newcommand{\mythmname}[1]{\emph{(#1.)}}

\newcommand{\QMC}{\epsilon}

\newcommand{\Gammaext}{\widetilde{\Gamma}}
\newcommand{\noi}{\noindent}
\newcommand{\cond}{\operatorname{cond}}
\newcommand{\hsc}{\hbar}

\newcommand{\Op}{{\rm Op}}

\newcommand{\Reg}{R}
\newcommand{\ItDR}{P^{-,\eta, \Reg}_{\rm ItD}}
\newcommand{\Bregimp}{\widetilde{B}_{k,\eta,\Reg}}
\newcommand{\Breg}{ B_{k,\eta,\Reg}}
\newcommand{\Bregp}{ B'_{k,\eta,\Reg}}
\newcommand{\BregSik}{ B_{k,\eta,S_{\ri k}}}
\newcommand{\BregSO}{ B_{k,\eta,S_{0}}}
\newcommand{\DL}{K}

\title{High-frequency estimates on boundary integral operators for the Helmholtz exterior Neumann problem.}
\author{
    J. Galkowski\thanks{Department of Mathematics, University College London, 25 Gordon Street, London, WC1H 0AY, UK, \tt J.Galkowski@ucl.ac.uk}
    \and
    P.~Marchand\thanks{Department of Mathematical Sciences, University of Bath, Bath, BA2 7AY, UK, \tt pfcm20@bath.ac.uk}
    \and
    E.~A.~Spence\thanks{Department of Mathematical Sciences, University of Bath, Bath, BA2 7AY, UK, \tt E.A.Spence@bath.ac.uk}}
\date{
    \today
}

\begin{document}

\maketitle

\begin{abstract}

We study a commonly-used second-kind boundary-integral equation for solving the Helmholtz exterior Neumann problem at high frequency, where, writing $\Gamma$ for the boundary of the obstacle, the relevant integral operators map $L^2(\Gamma)$ to itself.
We prove new frequency-explicit bounds on the norms of both the integral operator and its inverse. The bounds on the norm are valid for piecewise-smooth $\Gamma$ and are sharp up to factors of $\log k$ (where $k$ is the wavenumber), and the bounds on the norm of the inverse are valid for smooth $\Gamma$ and are observed to be sharp at least when $\Gamma$ is smooth with strictly-positive curvature.
Together, these results give bounds on the condition number of the operator on $L^2(\Gamma)$; this is the first time $L^2(\Gamma)$ condition-number bounds have been proved for this operator for obstacles other than balls. 

\end{abstract}

\section{Introduction}

\subsection{Motivation, and informal discussion of the main results and their novelty}

The frequency-dependence of the norms of both Helmholtz boundary-integral operators and their inverses has been studied since the work of Kress and Spassov \cite{KrSp:83, Kr:85} and Amini \cite{Am:90}, who studied the case when the obstacle is a ball.

Over the last 15 years there has been renewed interest in this dependence at high-frequency 
\cite{BuSa:06, BaSa:07, DoGrSm:07, ChMo:08, ChGrLaLi:09, ChGr:09, BeSp:11,
BeChGrLaLi:11, SpChGrSm:11, Me:12, BePhSp:13, Sp:14, ChHe:15, GaSm:15, HaTa:15, SpKaSm:15, BaSpWu:16, GaSp:19, ChSpGiSm:20}, 
motivated mainly by its importance in the analysis of associated boundary-element methods 
\cite{ChLa:07,
GaHa:11, 
LoMe:11,
ChGrLaSp:12, 
HeLaMe:13,
ChHeLaTw:15,
GrLoMeSp:15,
HeLaCh:15,
GaMuSp:19,
GiChLaMo:21}.
Almost all of the analysis of boundary-integral operators for the high-frequency Helmholtz equation has been for the exterior Dirichlet problem.
Indeed, there is only one paper proving frequency-explicit bounds on boundary-integral operators used to solve the high-frequency Helmholtz exterior Neumann problem \cite{BoTu:13}; in this informal discussion, we denote these operators by $B$. 
We discuss the results of  \cite{BoTu:13} in detail later, but note here that they prove (i) sharp bounds on $\|B\|$ when the obstacle is a ball and non-sharp bounds for general smooth obstacles, and (ii) a bound on $\|B^{-1}\|$ only when the obstacle is a ball.

In this paper, we prove bounds on $\|B\|$ and $\|B^{-1}\|$ (see Theorems \ref{thm:upper_bound_norm} and \ref{thm:upper_bound_inverse}). 
The bounds on $\|B\|$ are valid for piecewise smooth domains and are sharp up to factors of $\log k$, while those on $\|B^{-1}\|$ are valid for smooth domains, and are observed to be sharp (via numerical experiments) at least for strictly-convex obstacles. 
These bounds are the Neumann analogues of the Dirichlet results obtained in 
\cite{ChMo:08, ChGrLaLi:09, BeChGrLaLi:11, BaSpWu:16, LaSpWu:20}.

In obtaining these bounds, we crucially use the high-frequency decompositions of the single-layer, double-layer, and hypersingular operators from \cite{Ga:19}, the PDE results of \cite{Bu:98, Vo:00, BaSpWu:16,LaSpWu:20, GaMaSp:21, GaLaSp:21}, and results about semiclassical pseudodifferential operators (see, e.g., \cite{Zw:12}, \cite[Appendix E]{DyZw:19}).

An immediate application of these bounds is in extending the Dirichlet analysis in \cite{MaGaSpSp:21} of iterative methods applied to the 
linear systems arising from the boundary-element method 
to the Neumann case (see the discussion in \S\ref{sec:MGSS}). Furthermore, the results in \S\ref{sec:4.1} about the high-frequency components of the 
single- and double-layer operators are used in the high-frequency analysis of the boundary-element method in \cite{GaSp:22}.

\subsection{The Helmholtz exterior Neumann problem}\label{sec:Neumann}

Let $\Oi \subset \Rea^d$, $d\geq 2$ be a bounded open set such that its open complement $\Oe :=\Rea^d \setminus \overline{\Oi }$ is connected. 
Let $\Gamma:= \partial \Oi $; the majority of the results in this paper hold when $\Gamma$ is $C^\infty$ (so that we can easily use the calculus of pseudodifferential operators), but some results hold when $\Gamma$ is piecewise smooth in the sense of Definition \ref{def:piecewisesmooth} below.
Let $n$ be the outward-pointing unit normal vector to $\Oi $, and let $\gamma^\pm$ and $\partial^{\pm}_n$ denote the Dirichlet and Neumann traces on $\Gamma$ from $\Omega^{\pm}$.

We consider the exterior Neumann scattering problem. For simplicity, we consider boundary data coming from an incident plane wave $u^I(\bx):= \exp(\ri k \bx\cdot \ba)$ for $\ba\in\Rea^d$ with $|\ba|_2=1$, but we note that the same boundary-integral operators used to solve this problem can be used to solve the exterior Neumann problem given arbitrary data in $\HmhG$.
That is, we consider the sound-hard plane-wave scattering problem defined by:~given $k>0$ and the incident plane wave $u^I$,
find the total field $u \in H^1_{\rm loc}(\Omega_+)$ satisfying 
\begin{align}\label{eq:Helmholtz}
\Delta u + k^2 u =0 \quad\tin\quad \Oe ,\qquad 
\partial_n^+ u = 0 \quad\ton\quad\Gamma,
\end{align}
and
\beq\label{eq:src}
\dfrac{\partial u^S }{\partial r} -\ri ku^S = o \left(\frac{1}{r^{(d-1)/2}}\right)  \text{ as }r:=|x|\rightarrow \infty, \text{ uniformly in $x/r$},
\eeq
where \(u^S := u - u^I\) is the scattered field. We study this problem when the wavenumber $k$ is large. 

\subsection{Boundary-integral operators}
The standard single-layer, adjoint-double-layer, double-layer, and hypersingular operators are defined for $k\in \mathbb{C}$, $\phi\in \LtG$, $\psi\in H^1(\Gamma)$, and $x\in \Gamma$ by 
\begin{align}\label{eq:SD'}
&S_k \phi(\bx) := \int_\bound \Phi_k(\bx,\by) \phi(\by)\,\rd s(\by), \qquad
\DL_k' \phi(\bx) := \int_\bound \frac{\partial \Phi_k(\bx,\by)}{\partial n(\bx)}  \phi(\by)\,\rd s(\by),\\
&\DL_k \phi(\bx) := \int_\bound \frac{\partial \Phi_k(\bx,\by)}{\partial n(\by)}  \psi(\by)\,\rd s(\by), 
\quad 
H_k \psi(\bx) := \pdiff{}{n(\bx)} \int_\bound \frac{\partial \Phi_k(\bx,\by)}{\partial n(\bx)}  \psi(\by)\,\rd s(\by),
 \label{eq:DH}
\end{align}
where $\Phi_k(\bx,\by)$ is the standard Helmholtz fundamental solution satisfying the radiation condition \eqref{eq:src}; see \eqref{eq:fund} below.
(We use the notation $K_k$, $K_k'$ for the double-layer and its adjoint, instead of $D_k$, $D_k'$, to avoid a notational clash with the differential operator $D:= -\ri \partial$ used in \S\ref{sec:SCA} onwards.) 

This paper studies the integral operators
\beq\label{eq:BIEs}
    \Breg := \ri\eta \left(\dfrac{1}{2}I-\DL_k\right) +  \Reg H_k
    \quad\tand\quad
    \Breg' := \ri\eta \left(\dfrac{1}{2}I-\DL_k'\right) +  H_k\Reg
       \eeq
       where $\eta\in \Com\setminus\{0\}$, and the operator $\Reg$ satisfies the following assumption. This assumption uses the notation of semiclassical pseudodifferential operators on $\Gamma$ recapped in \S\ref{sec:SCA}. 
     
\begin{assumption}\label{ass:Reg}
$\Reg\in k^{-1} \Psi_{k^{-1}}^{-1}(\Gamma)$ is elliptic and its semiclassical principal symbol, $\sigma_{k^{-1}}(\Reg)$, is real.
\end{assumption}

The prototypical example of an operator satisfying Assumption \ref{ass:Reg} is $S_{\ri k}$, i.e. the single-layer operator at wavenumber $\ri k$. Assumption \ref{ass:Reg} and standard mapping properties of $K_k, K_k',$ and $H_k$ (see \eqref{eq:mapping} below) imply that $\Breg, \Breg': \LtGt$.
Indeed, since $H_k: L^2(\Gamma)\rightarrow H^{-1}(\Gamma)$, the fact that $R$ is a regulariser and maps $H^{-1}(\Gamma)\rightarrow\LtG$ is crucial; see \S\ref{sec:rationale} below for a recap of the history of this idea.

We use the $'$ notation on $\Breg'$ because, if $R$ is self-adjoint in the real-valued $L^2(\Gamma)$ inner product, then $\Breg$ and $\Breg'$ are self-adjoint in this inner product; see Lemma \ref{lem:quasiadjoint} below. 

\paragraph{The relationship of $\Breg$ and $\Breg'$ to the Helmholtz exterior Neumann problem.}

If $u$ is the solution of \eqref{eq:Helmholtz}-\eqref{eq:src}, then 
\begin{align}\label{eq:direct}
    \Breg \gamma^+  u  = \ri \eta \gamma^+ u^I  - \Reg \partial_{n}^+ u^I.
\end{align}
Indeed, expressing $u$ via Green's integral representation (see \eqref{eq:Green}) and taking Dirichlet and Neumann traces (using the third and fourth jump relations in \eqref{eq:jumprelations}) yields the two integral equations 
\beq\label{eq:standard}
\left(\dfrac{1}{2}I-\DL_k\right) \gamma^+ u = \gamma^+ u^I \quad\tand \quad  H_k\gamma^+ u = - \partial_n^+ u^I;
\eeq
acting on the second equation with $\Reg$ and then adding this to $\ri \eta$ times the first, we obtain \eqref{eq:direct}.

Furthermore, if $\phi$ satisfies
\begin{align}\label{eq:indirect}
\Breg' \phi  = -\partial_n^+ u^I,
\end{align}
then, by the jump relations \eqref{eq:jumprelations}, $u=u^I+ (\cK_k  \Reg-\ri \eta \cS_k)\phi$ is a solution of \eqref{eq:Helmholtz}-\eqref{eq:src} (where the double- and single-layer potentials, $\cK_k$ and $\cS_k$, are defined by \eqref{eq:SLPDLP}).
 
Since the unknown in \eqref{eq:direct} is the unknown part of the Cauchy data of $u$ satisfying  \eqref{eq:Helmholtz}-\eqref{eq:src}, the boundary-integral equation (BIE) \eqref{eq:direct} is called a \emph{direct} BIE. On the other hand, since the unknown in \eqref{eq:indirect} has less-immediate physical relevance, the BIE \eqref{eq:indirect} is known as an \emph{indirect} BIE.

The main results of the paper  -- bounds on $\Breg$, $\Breg'$, and their inverses -- are stated in the next section (\S\ref{sec:main_results}); an outline of the rest of the paper is then given in \S\ref{sec:outline}. We highlight here that the results of this paper can be extended to cover the analogous BIOs for the exterior impedance problem, with this outlined in \S\ref{sec:ext_imp}.

\section{Statement of the main results}
\label{sec:main_results}

Our first main result gives bounds on the norms of $\Breg$.and $\Breg'$.
\begin{theorem}[Bounds on $\|\Breg\|_{\LtGt}$ and $\|\Bregp\|_{\LtGt}$]\label{thm:upper_bound_norm} 

\

(i) If $\Reg$ satisfies Assumption \ref{ass:Reg} and $\Gamma$ is $C^\infty$ and curved (in the sense of Definition \ref{def:curved}) then given $k_0>0$ there exists $C>0$ such that, for all $k\geq k_0$,  
\beq\label{eq:ball}
\N{\Breg}_{\LtGt}+
\big\|\Breg'\big\|_{\LtGt}
 \leq C \big( 1 + |\eta| \big).
\eeq

(ii) If $\Reg$ satisfies Assumption \ref{ass:Reg} and $\Gamma$ is $C^\infty$
then given $k_0>0$ there exists $C>0$ such that, for all $k\geq k_0$,  
\beq\label{eq:smoothnormbound}
\N{\Breg}_{\LtGt}+
\big\|\Breg'\big\|_{\LtGt}
\leq C \Big( |\eta|\big(1 + k^{1/4}\log (k+2)\big) +\log (k+2) \Big).
\eeq

(iii) If $\Reg= S_{\ri k}$ and $\Gamma$ is piecewise smooth (in the sense of Definition \ref{def:piecewisesmooth}), then given $k_0>0$ there exists $C>0$ such that, for all $k\geq k_0$,  
\beq\label{eq:piecewisebound}
\N{\Breg}_{\LtGt}+
\big\|\Breg'\big\|_{\LtGt}
\leq C \Big( |\eta|\big(1 + k^{1/4}\log (k+2)\big) +\big(\log (k+2)\big)^{3/2} \Big).
\eeq

(iv) If $\Reg=S_{\ri k}$ and $\Gamma$ is piecewise curved (in the sense of Definition \ref{def:piecewisecurved}), then given $k_0>0$ there exists $C>0$ such that, for all $k\geq k_0$,    
\beqs
\N{\Breg}_{\LtGt}+
\big\|\Breg'\big\|_{\LtGt}
\leq C \Big( |\eta|\big(1 + k^{1/6}\log (k+2)\big) +\big(\log (k+2)\big)^{3/2} \Big).
\eeqs
\end{theorem}

We next give conditions under which $\Breg$ and $\Breg'$ are invertible on $\LtG$.

\begin{theorem}[Invertibility of $\Breg$ and $\Breg'$ on $\LtG$]\label{thm:invert}

\

(i) If $\Gamma$ is $C^\infty$, $\Reg$ satisfies Assumption \ref{ass:Reg}, and $\eta\in \Rea\setminus\{0\}$, then
there exists a $k_0>0$ such that, for all $k\geq k_0$,
 $\Breg$ and $\Breg'$ are injective and Fredholm on $\LtG$, and hence invertible.

(ii) Suppose that $\Gamma$ is $C^1$, $\eta\in \Rea\setminus\{0\}$, and either $\Reg = S_{\ri k}$ or $\Reg=S_0$, 
where in the latter case in 2-d the constant $a$ in the Laplace fundamental solution \eqref{eq:LaplaceFund} is 
taken larger than the \emph{capacity} of $\Gamma$ (see, e.g., \cite[Page 263]{Mc:00} for the definition of capacity).
Then,  for all $k>0$, $\Breg$ and $\Breg'$ are injective and are equal to a multiple of the identity plus a compact operator on $\LtG$, and hence invertible.
\end{theorem}

In addition, we prove bounds on the inverses of $\Breg$ and $\Breg'$.

\begin{theorem}[Upper bounds on $\|(\Breg)^{-1}\|_{\LtGt}$ and $\|(\Breg')^{-1}\|_{\LtGt}$]\label{thm:upper_bound_inverse}
Assume that $\eta\in \Rea\setminus\{0\}$ is independent of $k$ and that $\Reg$ satisfies Assumption \ref{ass:Reg}.

(i) If $\Oi$ is $C^\infty$ and $\Gamma$ is curved (and hence $\Oi$ is nontrapping in the sense of Definition \ref{def:nontrapping}), then there exists $k_0>0$ and $C>0$ such that, for all $k\geq k_0$,    
\beq\label{eq:ballinverse}
\N{(\Breg)^{-1}}_{\LtGt}+
\big\|(\Breg')^{-1}\big\|_{\LtGt}
\leq C k^{1/3}.
\eeq

(ii) If $\Oi$ is $C^\infty$ and nontrapping, then there exists $k_0>0$ and $C>0$ such that, for all $k\geq k_0$,    
\beq\label{eq:smoothinversebound}
\N{(\Breg)^{-1}}_{\LtGt}+
\big\|(\Breg')^{-1}\big\|_{\LtGt}
\leq C k^{2/3}.
\eeq

(iii) If $\Oi$ is $C^\infty$ then there exists $k_0>0$ such that given $\delta>0$ there exists a set $J\subset [k_0,\infty)$ with $|J|\leq \delta$ such that, given $\eps>0$, there exists $C= C(k_0,\delta,\eps)>0$ such that, for all $k \in [k_0,\infty)\setminus J,$
\beq\label{eq:inverse_bound_formost}
\N{(\Breg)^{-1}}_{\LtGt}+
\big\|(\Breg')^{-1}\big\|_{\LtGt}
\leq C k^{5d/2+1 + \eps}. 
\eeq

(iv) If $\Oi$ is $C^\infty$ then there exists $k_0>0,$ $\alpha>0$, and $C>0$ such that, for all $k\geq k_0$,    
\beqs
\N{(\Breg)^{-1}}_{\LtGt}+
\big\|(\Breg')^{-1}\big\|_{\LtGt}
\leq C \exp(\alpha k).
\eeqs
\end{theorem}

\bre[Choice of $\eta$]\label{rem:eta1}
Theorem \ref{thm:upper_bound_inverse} is proved under the assumption that $\eta$ is independent of $k$. 
This choice was advocated for in \cite{BrElTu:12, BoTu:13}, with these papers stating that this choice 
leads to a ``small number''/``nearly optimal numbers'' of iterations of the generalised minimum residual method (GMRES) compared to other choices of $\eta$; see \cite[Equation 23]{BrElTu:12}, \cite[\S5]{BoTu:13}.
\S\ref{sec:eta} contains numerical results showing that, at least for some geometries, \emph{both} the condition number of $\Breg$ 
\emph{and} the number of GMRES iterations are smaller for some $k$-dependent choices of $\eta$ than they are when $\eta$ is independent of $k$.
 \ere

Part (iv) of Theorem \ref{thm:upper_bound_inverse} shows that $\|(\Breg)^{-1}\|_{\LtGt}$ can grow at most exponentially in $k$, although Part (iii) shows that for most frequencies $\|(\Breg)^{-1}\|_{\LtGt}$ is polynomially bounded in $k$. We now show that exponential growth occurs through a discrete set of $k$s.

\begin{definition}[Quasimodes]\label{def:quasimodes}
    A family of Neumann quasimodes of quality $\QMC(k)$
    is a sequence $\{(u_j,k_j)\}_{j=1}^\infty\subset H^2_{\rm loc}(\Oe)\times \mathbb{R}$ 
    with $\partial_n^+ u =0$ on $\Gamma$ 
    such that the frequencies $k_j\tendi$ as $j \tendi$ and there exists a compact subset $\mathcal{K}\subset \Oe$ such that, for all $j$, $\supp\, u_j \subset \mathcal{K}$,
    \beqs
    \N{(\Delta +k_j^2) u_j}_{L^2(\Oe)} \leq \QMC(k_j), \quad\tand\quad\N{u_j}_{L^2(\Oe)}=1.
    \eeqs
\end{definition}

\begin{theorem}[Lower bounds on $\|(\Breg)^{-1}\|_{\LtGt}$]\label{thm:lower_bound_inverse}
Assume that $\Gamma$ is piecewise smooth, $\Reg$ is bounded on $\LtG$, and $B'_{k,\eta,\Reg}$ and $B_{k,\eta,\Reg}$ are bounded and invertible on $\LtG$.
If there exists a family of Neumann quasimodes with quality $\epsilon(k)$, then there exists $C>0$ (independent of $j$) such that 
\begin{align*}
&   \min\Big\{
    \big\|   (B'_{k_j,\eta,R})^{-1}\big\|_{\LtGt}\,,\,   
    \N{    (B_{k_j,\eta,R})^{-1}}_{\LtGt}\big\}\\
&\hspace{4cm}
\geq C \left(\frac{1}{\QMC(k_j)} - \frac{1}{k_j}\right)
k_j^{1/2}
\Big(\lVert \Reg\rVert_{L^2(\Gamma)\rightarrow L^2(\Gamma)} k_j + |\eta| \Big)^{-1} .
\end{align*}
\end{theorem}

We emphasise that the lower bound of Theorem \ref{thm:lower_bound_inverse} does not require that $R$ satisfy Assumption \ref{ass:Reg}, and so holds for more general $\Reg$ (such as $R=S_0$).

The following result gives situations where quasimodes with small quality exist;
Part (i) is \cite[Theorem 1]{St:00}, and Part (ii) is  \cite[Theorem 3.1]{NgGr:13}.
Recall that the resonances of the exterior Neumann problem are the poles of the meromorphic continuation of the solution operator from $\Im k\geq 0$ to $\Im k<0$; see, e.g., \cite[Theorem 4.4. and Definition 4.6]{DyZw:19}). 
We use the notation that $a = O(k^{-\infty})$ as $k\tendi$ if, given $N>0$, there exists $C_N$ and $k_0$ such that $|a|\leq C_N k^{-N}$ for all $k\geq k_0$, i.e.~$a$ decreases superalgebraically in $k$.

\begin{theorem}[Existence of quasimodes with $\QMC(k)=O(k^{-\infty})$]\label{thm:ellipse}

\

(i) If there exists a sequence of resonances $\{\lambda_\ell\}_{\ell=1}^\infty$ of the exterior Neumann problem with
\beqs
0\leq -\Im \lambda_\ell = \mathcal{O}\big(|\lambda_\ell|^{-\infty}\big)  \quad\tand \quad \Re \lambda_\ell \tendi \quad\tas\quad \ell \tendi,
\eeqs
then there exist families of Neumann quasimodes with $\QMC(k)=\mathcal{O}(k^{-\infty})$.

(ii) Let $d=2$. Given $a_1>a_2>0$, let 
\beq\label{eq:ellipse}
E:= \left\{(x_1,x_2) \, : \, \left(\frac{x_1}{a_1}\right)^2+\left(\frac{x_2}{a_2}\right)^2<1\right\}.
\eeq
Assume that $\Gamma$ coincides with the boundary of $E$ in the neighbourhoods of the points \((0,\pm a_2)\), 
and that $\Oe $ contains the convex hull of these neighbourhoods.
Then there exist families of Neumann quasimodes with 
\beqs
\QMC(k)=C_1 \exp( - C_2 k) \quad\tfa k>0.
\eeqs
where $C_1, C_2>0$ are both independent of $k$.
\end{theorem}

\subsection{Discussion of the main results}\label{sec:discussion}

\subsubsection{The rationale behind using $\Breg$ and $\Breg'$ to solve the exterior Neumann problem.}\label{sec:rationale}

Recall that taking the Dirichlet and Neumann traces of Green's integral representation results in the two equations \eqref{eq:standard}. 
Each of the integral operators in these two equations is not invertible for all $k>0$. This fact prompted the introduction of 
``combined-field'' or ``combined-potential'' BIEs in the 1960s and 1970s, with \cite{BuMi:71} 
using the BIE 
\begin{align}\label{eq:BW}
    B_{k,\eta} u  = \ri \eta \gamma^+u^I  -\partial_n^+ u^I , 
\quad\text{ where }\quad 
    B_{k,\eta} := \ri\eta \left(\dfrac{1}{2}I-\DL_k\right) +  H_k,
\end{align}
and \cite{BrWe:65,Le:65,Pa:65} introducing analogous BIEs for the exterior Dirichlet problem.
The analogous Neumann indirect formulation comes from posing the ansatz $u^S= (\cK_k-\ri \eta \cS_k)\phi$, after which the jump relations \eqref{eq:jumprelations} imply that 
\beq\label{eq:BWindirect}
B'_{k,\eta} \phi= -\partial_n^+ u^I,
\quad\text{ where }\quad 
    B'_{k,\eta} := \ri\eta \left(\dfrac{1}{2}I-\DL_k'\right) +  H_k.
\eeq
For $k>0$ and $\Re \eta\neq 0$, $B_{k,\eta}$ and $B'_{k,\eta}$ are bounded and invertible operators from $H^{s+1/2}(\Gamma)$ to $H^{s-1/2}(\Gamma)$ for all $|s|\leq 1/2$; see \cite[Theorem 2.27]{ChGrLaSp:12}.

The presence of $H_k$ in \eqref{eq:BW} and \eqref{eq:BWindirect} means that both $B_{k,\eta}$ and $B'_{k,\eta}$ are not bounded from $\LtG\rightarrow \LtG$, and this means that the condition numbers of their $h$-version Galerkin discretisations blow up 
as $h\tendo$ for fixed $k$ \cite[\S4.5]{SaSc:11}. 
This motivates using the BIEs \eqref{eq:direct} and \eqref{eq:indirect} where $R$ is chosen as an order $-1$ operator so that the composition $RH_k: \LtGt$. (Once $R$ is introduced, the constant $\ri \eta$ at the front of $\Breg$ and $\Breg'$ is redundant, but we keep it so that $\Breg$ and $\Breg'$ reduce to the classic operators $B_{k,\eta}$ and $B'_{k,\eta}$ when $R=I$.)

A popular choice is $R= S_0$ (see, e.g., \cite{AmHa:90, StWe:98, AnOvTu:11}) or $R= (S_0)^2$ (see, e.g., \cite[\S3.2]{CoKr:98}, \cite[Proof of Theorem 9.1]{Mi:96}). These choices are motivated by the Calder\'on relations
\beq\label{eq:Calderon}
S_k H_k = -\frac{1}{4}I + \DL_k^2
\quad\tand\quad
H_k S_k= -\frac{1}{4}I + (\DL_k')^2,
\eeq
for all $k\geq 0$;
see, e.g., \cite[Equation 2.56]{ChGrLaSp:12}. Indeed, if $R= S_0$ and $\Gamma$ is $C^1$, then $\Breg$ and $\Breg'$ equal a multiple of the identity plus a compact operator on $\LtG$, since 
$\DL_k$ and $\DL_k'$ are compact when $\Gamma$ is $C^1$ by \cite[Theorem 1.2]{FaJoRi:78}, and 
$(S_{k}- S_0)H_k $ and $H_{k}(S_k-S_0)$ are compact (this follows from the mapping properties \eqref{eq:mapping} and the bounds on $\Phi_k-\Phi_0$
 in, e.g., \cite[Equation 2.25]{ChGrLaSp:12}). The idea of composing the hypersingular operator with the single-layer operator 
goes back to \cite{Bu:76} (see the discussion in, e.g., \cite{AmHa:90}),  
and falls under the class of methods known as ``operator preconditioning''; see \cite{StWe:98, Hi:06}. 

Following the use of $\Reg=S_0$, the choice $\Reg=S_{\ri k}$ was proposed in \cite{BrElTu:12}, and then advocated for in
\cite{BoTu:13,ViGrGi:14}, with \cite{BoTu:13} also using the principal symbol of $S_{\ri k}$.
Part of the contribution of the present paper is the rigorous justification of this choice. Indeed, a result of \cite{Ga:19} (extended in Theorem \ref{thm:Hk} below) shows that the norm of $H_k$ grows with $k$. If  $R$ is an order $-1$ operator that is independent of $k$, then $R H_k :\LtGt$, but with a norm that grows with $k$. A better choice is therefore an operator of order $-1$ whose norm decreases with $k$, leading to  the general class of $R$ described in Assumption \ref{ass:Reg}, to which $S_{\ri k}$ belongs.

Finally, we note that if $\Reg$ equals $\ri \eta$ times the exterior Neumann-to-Dirichlet map $\NtD$, then $\Breg=\ri \eta$ (this can be proved by taking the Neumann trace of Green's integral representation and using the definition of $\NtD$). This observation is then the basis of the construction of suitable operators $\Reg$ (more complicated than $S_0$ or $S_{\ri k}$) in \cite{LeMi:04,AnDa:07,AnDa:05,Antoine2005a, DaDaLa:13, AnDa:21}.

\subsubsection{Comparison with the results of \cite{BoTu:13}}

The paper \cite{BoTu:13} considers the operator $\Bregp$ with $R$ equal to either $S_{\ri k}$ or its principal symbol. By Lemma \ref{lem:quasiadjoint}, the results in  \cite{BoTu:13} also hold for $\Breg$ with these choices of $R$. The majority of the bounds in \cite{BoTu:13} are proved for $\Oi$ a 2- or 3-d ball, using 
the fact that the eigenvalues of the boundary-integral operators can be expressed in terms of Bessel and Hankel functions, and then bounding the appropriate combinations of these functions uniformly in both argument and order.

The results \cite[Theorems 3.2 and 3.4]{BoTu:13} prove the bound \eqref{eq:ball} when $\Oi$ is a 2- or 3-d ball.
The result \cite[Theorem 3.12]{BoTu:13} proves that, if $\Gamma$ is $C^\infty$, then $\|\Breg \|_{\LtGt}\lesssim (1+|\eta|)k^{1/2+\eps}$ for any $\eps>0$, which is less sharp in its $k$-dependence than \eqref{eq:smoothnormbound}.
The results \cite[Theorems 3.6 and 3.9]{BoTu:13} show that there exist $k_0, C_1, C_2>0$ such that if $\Oi$ is a 2- or 3-d ball, $k\geq k_0$, and $\eta \geq C_1 k^{1/3}$, then 
\beqs
\Re \big\langle \Bregp \phi,\phi\big\rangle_\Gamma \geq C_2 \N{\phi}^2_{\LtG} \quad\tfa \phi\in\LtG;
\eeqs
i.e., that $\Bregp$ is coercive on $\LtG$ when $\Oi$ is a ball. By the Lax--Milgram theorem, this implies that $\|(\Bregp)^{-1}\|_{\LtGt} \leq (C_2)^{-1}$, under the same assumptions on $\Oi, k$, and $\eta$. The calculations in \cite{BoTu:13} suggest actually that (for sufficiently-large $k$) $\Bregp$ is coercive with constant $|\eta| k^{-1/3}$; see \cite[Remark 3.7]{BoTu:13}. If this were the case, then $\|(\Bregp)^{-1}\|_{\LtGt}\leq C k^{1/3}/|\eta|$ for $\Oi$ the ball, which would be consistent with the $k$-dependence in \eqref{eq:ballinverse} (recall that this latter bound is proved assuming that $\eta\in \Rea\setminus\{0\}$ is independent of $k$).

\subsubsection{Comparison of conditioning of $\Breg$ with that for its Dirichlet analogue}

If $\Oi$ is smooth and curved and $\eta$ is independent of $k$, 
then the $\LtG\rightarrow \LtG$ condition number of $\Breg$,
\beq\label{eq:conditionnumber}
\cond(\Breg):= \N{\Breg}_{\LtGt} \big\|(\Breg)^{-1}\big\|_{\LtGt},
\eeq
satisfies $\cond (\Breg)\sim k^{1/3}$. This is the same $k$-dependence as the condition number of the direct and indirect 
boundary-integral operators used to solve the exterior Dirichlet problem for this geometry. Indeed, these operators are, respectively,
\beq\label{eq:Aketa}
A'_{k,\eta}:= \frac{1}{2}I + \DL_k' - \ri \eta S_k \quad\tand\quad A_{k,\eta} :=  \frac{1}{2}I + \DL_k - \ri \eta S_k.
\eeq
When $|\eta|\sim k$ (which one can actually prove is the optimal choice for general $\Oi$), $\cond(A'_{k,\eta}) = \cond(A_{k,\eta}) \sim k^{1/3}$, with the bound on $\|(A'_{k,\eta})^{-1}\|_{\LtGt}$ coming from \cite[Theorem 4.3]{ChMo:08} or \cite[Theorem 1.13]{BaSpWu:16} and the bound on the norm coming from \cite[Theorem 1.2]{GaSm:15} and \cite[Theorem A.1]{HaTa:15}.

When $\Oi$ is $C^\infty$ and nontrapping and $\eta\sim 1$,  $\cond(\Breg) \lesssim k^{11/12}\log k$ by \eqref{eq:smoothnormbound} and \eqref{eq:smoothinversebound}. In contrast, when $\Oi$ is $C^\infty$ and nontrapping and $|\eta|\sim k$, $\cond(A'_{k,\eta}) = \cond(A_{k,\eta}) \lesssim k^{1/2}\log k$ (with the bound on the norm again coming from \cite{GaSm:15, HaTa:15} and the bound on the inverse coming from \cite[Theorem 1.13]{BaSpWu:16}).
For summaries of the results on the conditioning of $A_{k,\eta}'$ and $A_{k,\eta}$ and their sharpness, see \cite[\S5.4]{ChGrLaSp:12}, \cite[Section 7]{BaSpWu:16}, \cite[Theorem 6.4]{ChSpGiSm:20}.

\subsubsection{Why is $\Breg$ harder to analyse than $B_{k,\eta}$?}\label{sec:harder}

The summary is that analysing $\Breg$ is harder than analysing $B_{k,\eta}$ because $(B_{k,\eta})^{-1}$ can be expressed in terms of the interior Impedance-to-Dirichlet map, about which much is known, but $(\Breg)^{-1}$ can only be expressed in terms of a non-standard Impedance-to-Dirichlet map involving $R$ (see \eqref{eq:ItDSdef} below), about which very little was known until the recent results of \cite{GaLaSp:21}. 

As well as being used to solve the exterior Neumann problem, the integral operator $B_{k,\eta}$ defined by \eqref{eq:BW} can be also used to solve the interior impedance problem
\beq\label{eq:iip}
\Delta u + k^2 u =0 \quad\tin \Oi  \quad\tand \quad \partial_n^- u - \ri \eta \gamma^- u = g \quad\ton \Gamma.
\eeq
Indeed, seeking a solution of \eqref{eq:iip} of the form $u= \cK_k\phi$, the third and fourth jump relations in \eqref{eq:jumprelations} implies that $B_{k,\eta}\phi=g$. 
This relationship between the operator $\Breg$, the exterior Neumann problem, and the interior impedance problem is demonstrated further by the decomposition
\begin{align}\label{eq:fav_formula}
    (B_{k,\eta})^{-1}= \NtD  - (I-\ri\eta \NtD )\ItD
\end{align}
\cite[Equation 2.94]{ChGrLaSp:12}. Here $\NtD:\HmhG\rightarrow\HhG$ 
is the Neumann-to-Dirichlet map for the Helmholtz equation posed in $\Oe $ with the Sommerfeld radiation condition \eqref{eq:src}, and $\ItD: \HmhG\rightarrow \HhG$ is the Impedance-to-Dirichlet map for the problem \eqref{eq:iip} (i.e., the map $g\mapsto \gamma^-u$).
Recall that both $\NtD$ and $\ItD$ have unique extensions to bounded operators $H^{s-1/2}(\Gamma)\rightarrow H^{s+1/2}(\Gamma)$ for $|s|\leq 1/2$ 
(see \cite[Section 2.7]{ChGrLaSp:12} and Lemma \ref{lem:NtD1} below) and thus \eqref{eq:fav_formula} is valid on this range of Sobolev spaces.

The analogue of \eqref{eq:fav_formula} for $\Breg$ is 
    \begin{align}\label{eq:fav_formula_reg_intro}
        (\Breg)^{-1}= \NtD \Reg^{-1} - (I-\ri\eta \NtD \Reg^{-1})\ItDR
    \end{align}
where the map \(\ItDR\) takes
$g \mapsto \gamma^- u$, where $u$ is the solution of 
\beq\label{eq:ItDSdef}
\Delta u + k^2 u =0 \quad\tin \Oi  \quad\tand \quad \Reg\partial_n^- u - \ri \eta \gamma^- u = g \quad\ton \Gamma.
\eeq
The formula \eqref{eq:fav_formula_reg_intro} was proved in \cite[Lemma 6.1]{BaSpWu:16}; since it is central to the present paper we 
nevertheless state this result as Lemma \ref{lem:fav_formula_reg} below and 
give a short proof, different to that in \cite{BaSpWu:16}.
In \S\ref{sec:ItDS} we prove the necessary results about the problem \eqref{eq:ItDSdef} to prove Theorem \ref{thm:upper_bound_inverse}, using results about semiclassical pseudodifferential operators and recent results about the frequency-explicit wellposedness of \eqref{eq:ItDSdef} from \cite[Section 4]{GaLaSp:21}. 

\subsubsection{Extending the results of \cite{MaGaSpSp:21} to $\Breg$.}\label{sec:MGSS}
The paper \cite{MaGaSpSp:21} proves a $k$-explicit bound on the number of iterations when GMRES is applied to the standard second-kind integral equation for the exterior Dirichlet problem when $\Oi$ is trapping, and the proof uses the Dirichlet analogues of (a) the bounds in Parts (iii) and (iv) of Theorem \ref{thm:upper_bound_inverse}, and (b)
 the bounds in Theorem \ref{thm:upper_bound_norm}.
Therefore, with the bounds of Theorems \ref{thm:upper_bound_norm} and \ref{thm:upper_bound_inverse} in hand, the main result of \cite{MaGaSpSp:21} (i.e., \cite[Theorem 1.6]{MaGaSpSp:21}) also holds for $\Breg$; see \cite[Remark 2.7]{MaGaSpSp:21}.

\subsection{Outline of the paper}\label{sec:outline}
\S\ref{sec:SCA} recaps existing results about layer potentials, boundary-integral operators, and semiclassical pseudodifferential operators. 
\S\ref{sec:4} proves new results about boundary-integral operators.
\S\ref{sec:NtD} proves new bounds on the exterior Neumann-to-Dirichlet map $\NtD$.
\S\ref{sec:ItDS} proves new bounds on the interior impedance-to-Dirichlet map $\ItDS$.
\S\ref{sec:mainproofs} proves the main results in \S\ref{sec:main_results}.
\S\ref{sec:num} contains numerical experiments illustrating the main results.
\S\ref{sec:eta} contains a heuristic discussion and numerical experiments investigating the dependence on the coupling parameter $\eta$.
Appendix \S\ref{sec:recap} recaps the definitions of layer potentials, their jump relations, and Green's integral representation.
Appendix \S\ref{sec:geo} defines precisely the geometric definitions used in the statements of the main results.
Appendix \S\ref{sec:ext_imp} outlines how the main results can be extended from the exterior Neumann problem to the exterior impedance problem.

\paragraph{Notation:}
In many of the proofs, $C>0$ is a constant whose values may change from line to line. We sometimes use the notation that $a\lesssim b$ if there exists $C>0$, independent of $k$, such that $a\leq Cb$. We say that $a\sim b$ if $a\lesssim b$ and $b\lesssim a$.

\section{Recap of existing results about layer potentials, boundary-integral operators, and semiclassical pseudodifferential operators}
\label{sec:SCA}

\subsection{Definition of weighted Sobolev spaces}\label{sec:weighted_norms}

We first define weighted Sobolev spaces on $\Rea^d$, and then use these to define analogous weighted Sobolev spaces on $\Gamma$.
Let
\beqs
(\cF u)(\zeta):= \int_{\Rea^d} \exp(-\ri \zeta\cdot x)\,u(x)\,\rd x,
\eeqs
and, for $s\in \Rea$ and $k>0$, let 
\beq\label{eq:Hsk}
H_k^s(\Rea^d):= \Big\{ u\in \mathcal{S}^*(\Rea^d) \,\tst\, \big(1+ k^{-2}|\zeta|^2\big)^{s/2} (\cF u)(\zeta) \in L^2(\Rea^d) \Big\},
\eeq
where $\mathcal{S}(\Rea^d)$ is the Schwartz space (see, e.g., \cite[Page 72]{Mc:00}) and $\mathcal{S}^*(\Rea^d)$ its dual.
Define the norm
\beq\label{eq:Hsknorm}
\N{u}_{H^s_k(\Rea^d)}^2:= \int_{\Rea^d} \big(1 + k^{-2}|\zeta|^2\big)^{s} |(\cF u)(\zeta)|^2 \, \rd \zeta.
\eeq
and observe that, for $s>0$,
\beq\label{eq:weightednorm}
 \N{u}_{H^{-s}_k(\Rea^d)} \leq \N{u}_{L^2(\Rea^d)}\leq \N{u}_{H^s_k(\Rea^d)}.
\eeq

If $\Gamma$ is $C^{m-1,1}$, the weighted spaces $H^s_k(\Gamma)$ for $|s|\leq m$ can be defined by charts; see, e.g., \cite[Pages 98 and 99]{Mc:00} for the unweighted case and \cite[\S5.6.4]{Ne:01} or \cite[Definition E.20]{DyZw:19} for the weighted case (but note that \cite[\S5.6.4]{Ne:01} uses the weight $(k^2 + |\zeta|^2\big)^{s}$ in \eqref{eq:Hsknorm} instead of our $(1 + k^{-2}|\zeta|^2\big)^{s}$).

The facts we need about these spaces in the rest of the paper are the following.

(i) Since $H^{-s}_k(\Rea^d)$ is an isometric realisation of the dual space of $H^s_k(\Rea^d)$ \cite[Page 76]{Mc:00}, $H^{-s}_k(\Gamma)$ is a realisation of the dual space of $H^s_k(\Gamma)$ \cite[Page 98]{Mc:00}.

(ii) 
\beq\label{eq:1knorm}
    \N{ w }_{H^1_k(\Gamma)}^2 \sim k^{-2}\N{ \nabla_{\Gamma} w }_{\LtG}^2 + \N{ w }_{L^2(\Gamma)}^2,
\eeq
where $\nabla_\Gamma$ is the surface gradient operator, defined in terms of a parametrisation of the boundary by, e.g., \cite[Equation A.14]{ChGrLaSp:12}. 

(iii) If $\Gamma$ is Lipschitz, then given $k_0>0$ and $1/2<s<3/2$ there exists $C>0$ such that for all $k\geq k_0$  
the Dirichlet trace operators $\gamma^\pm$ satisfy
\begin{equation}
\label{e:basicTrace}
\|\gamma^\pm\|_{H_k^s(\Omega^{\pm})\to H_k^{s-\frac{1}{2}}(\Gamma)}
\leq Ck^{\frac{1}{2}};
\end{equation}
this is proved in the unweighted case in \cite[Theorem 3.38]{Mc:00}, and the proof for the weighted case follows similarly; see, e.g., \cite[Theorem 5.6.4]{Ne:01}.
When $\gamma^+ u = \gamma^-u$ we write $\gamma u = \gamma^\pm u$; recall that the adjoint of this two-sided trace operator is defined by 
\beq\label{eq:traceadjoint}
\big\langle \gamma^*\phi,u\big\rangle_{\Rea^d} =\big\langle \phi, \gamma u\big\rangle_\Gamma
\eeq
for $\phi\in H^{1/2-s}(\Gamma)$, $1/2<s<3/2$, and $u\in C_{\rm comp}^\infty(\mathbb{R}^d)$ (see, e.g., \cite[Equation 6.14]{Mc:00}),
and then \eqref{e:basicTrace} implies that
\beq
\label{e:basicTrace2}
\|\gamma^*\|_{H_k^{\frac{1}{2}-s}(\Gamma)\to H_k^{-s}(\mathbb{R}^d)}
\leq Ck^{\frac{1}{2}}.
\eeq

\subsection{Recap of results about layer potentials and integral operators}

\begin{theorem}\mythmname{Bounds on the $\LtGt$ norms of $\DL_k, \DL_k'$ 
\cite[Appendix A]{HaTa:15}, \cite[Chapter 4]{Ga:19}}\label{th:K_k_bounds}
    Let \(\Oi\) a bounded Lipschitz open set such that the open complement \(\Oe :=\bbR^d\setminus \overline{\Oi}\) is connected

    \begin{enumerate}
        \item If $\Oi$ is convex and $\Gamma$ is $C^\infty$ and curved (in the sense of Definition \ref{def:curved}),  then given \(k_0>0\) there exists $C>0$ such that 
        \begin{align*}
            \N{ \DL_k'}_{L^2(\Gamma)\rightarrow L^2(\Gamma)}+            \N{ \DL_k}_{L^2(\Gamma)\rightarrow L^2(\Gamma)} \leq C\quad\tfa k\geq k_0.
        \end{align*}
\item        If \(\Gamma\) is piecewise curved (in the sense of Definition \ref{def:piecewisecurved}), then given \(k_0>0\) there exists $C>0$ such that
        \begin{align*}
                    \N{ \DL_k'}_{L^2(\Gamma)\rightarrow L^2(\Gamma)}+    
            \N{ \DL_k}_{L^2(\Gamma)\rightarrow L^2(\Gamma)} \leq C k^{1/6} \log (k+2)\quad\tfa k\geq k_0.
        \end{align*}
        \item If \(\Gamma \) is piecewise smooth (in the sense of Definition \ref{def:piecewisesmooth}), 
        then given \(k_0>0\) there exists $C>0$ such that
             \begin{align*}
                         \N{ \DL_k'}_{L^2(\Gamma)\rightarrow L^2(\Gamma)}+    
            \N{ \DL_k}_{L^2(\Gamma)\rightarrow L^2(\Gamma)} \leq C k^{1/4} \log (k+2)\quad\tfa k\geq k_0.
        \end{align*}
    \end{enumerate}
\end{theorem}

We make two remarks: (i) the bounds in Points 2 and 3 are sharp up to the factor of $\log (k+2)$, and the bound in Point 1 is sharp; see \cite[\S3]{GaSp:19}, \cite[\S A.3]{HaTa:15}, (ii) 
by \cite{GaSp:19}, bounds with the same $k$-dependence hold on the $L^2(\Gamma)\rightarrow H^1_k(\Gamma)$ norms under the additional assumption that $\Gamma$ is $C^{2,\alpha}$ for some $\alpha>0$, which is necessary for $\DL_k$ and $\DL_k$ to be bounded operators from $\LtG\rightarrow \HoG$ \cite[Theorem 4.2]{Ki:89}, \cite[Theorem 3.6]{CoKr:98}.

\begin{theorem}[{Bound on $\cD_k$ \cite[Theorem 1.2]{HaTa:15}}]
\label{thm:HT}
If $\Gamma$ is piecewise smooth, then, given $\chi \in C^\infty_{\rm comp}(\Rea^d)$ and $k_0>0$ there exists $C>0$ such that
\beqs
\N{\chi \cD_k}_{\LtG\rightarrow L^2(\Oe)}\leq C
\quad\tand\quad
\N{\cD_k}_{\LtG\rightarrow L^2(\Oi)}\leq C
\eeqs
for all $k\geq k_0$.
\end{theorem}

We also recall well-known bounds on the free resolvent i.e., integration against the fundamental solution $\Phi_{k} (x,y)$ defined by \eqref{eq:fund}. Let
\begin{align}\label{eq:free_resolvent} 
\calR_k f (x):= \int_{\Oe } \Phi_{k} (x,y)f(y) \dif y.
\end{align}

\begin{theorem}[{Bound on $\cR_k$}]
\label{thm:Newton}
Given $\chi_1, \chi_2 \in C^\infty_{\rm comp}(\Rea^d)$ and $k_0>0$ there exists $C>0$ such that
\beqs
\frac{1}{k}\N{\chi_1 \cR_k \chi_2}_{L^2(\Oe)\rightarrow H^2(\Oe)}
+ 
\N{\chi_1 \cR_k \chi_2}_{L^2(\Oe)\rightarrow H^1(\Oe)}
+
k\N{\chi_1 \cR_k \chi_2}_{L^2(\Oe)\rightarrow L^2(\Oe)}
\leq C
\eeqs
for all $k\geq k_0$.
\end{theorem}

\bpf[References for the proof]
See, e.g., \cite[Theorem 3.1]{DyZw:19} for odd $d$ and \cite[Theorem 14.3.7]{Ho:83a} for arbitrary dimension (note that \cite[Theorem 14.3.7]{Ho:83a} is for fixed $k$, but a rescaling of the independent variable yields the result for arbitrary $k$).
\epf

Finally, we recall that $S_{\ri k}:\HmhG\rightarrow \HhG$ is coercive by, e.g., \cite[Theorem 5.6.5]{Ne:01}; this result is proved using Green's first identity and the first two jump relations in \eqref{eq:jumprelations}.
Note that we use different weighted norms than \cite[Theorem 5.6.5]{Ne:01}, so that \cite[Theorem 5.6.5]{Ne:01} has the coercivity constant independent of $k$, but we have it proportional to $1/k$.

\begin{theorem}[{Coercivity of $S_{\ri k}$ on $\HmhG$ \cite[Theorem 5.6.5]{Ne:01}}]\label{thm:Sikcoer}
If $\Gamma$ is Lipschitz then given $k_0>0$  there exists $C>0$ such that, for all $k\geq k_0$,
\beqs
\big\langle S_{\ri k}\phi,\phi\big\rangle_\Gamma \geq \frac{C}{k} \N{\phi}^2_{H^{-1/2}_k(\Gamma)} \quad\tfa \phi \in\HmhG.
\eeqs
\end{theorem}

\subsection{Recap of results about semiclassical pseudodifferential operators}\label{sec:scpseudos}

\subsubsection{The semiclassical parameter and weighted Sobolev spaces}

Semiclassical pseudodifferential operators are pseudodifferential operators with a large/small parameter, where behaviour with respect to this parameter is explicitly tracked in the associated calculus. In our case the small parameter is 
$\hsc:=k^{-1}$; normally this parameter is denoted by $h$, but we use $\hsc$ to avoid a notational clash with the meshwidth of the $h$-version of the boundary element method.
The notation $\hsc$ is motivated by the fact that the semiclassical parameter is often related to Planck's constant, which is written as $2\pi\hsc$ see, e.g., \cite[S1.2]{Zw:12}, \cite[Page 82]{DyZw:19}.

We define the weighted spaces $H^s_\hsc(\Rea^d)$ by \eqref{eq:Hsk} with $\hsc =k^{-1}$. 
These spaces can also be defined by the semiclassical Fourier transform and its inverse
\beqs
(\mathcal F_{\hsc}\phi)(\xi) := \int_{\mathbb R^d} \exp\big( -\ri x \cdot \xi/\hsc\big)
\phi(x) \, \rd x, \qquad
(\mathcal F^{-1}_{\hsc}\psi)(x) := (2\pi \hsc)^{-d} \int_{\mathbb R^d} \exp\big( \ri x \cdot \xi/\hsc\big)
 \psi(\xi)\, \rd \xi;
\eeqs
see \cite[\S3.3]{Zw:12}.
Indeed, since $(\cF u)(\xi/\hsc)= \cF_\hsc(\xi)$, \eqref{eq:Hsknorm} implies that
\beq\label{eq:Hhnorm}
\Vert u \Vert_{H_k^m(\Rea^d)} ^2 = \hsc^{-d} \int_{\Rea^d} \langle \xi \rangle^{2m}
 |\mathcal F_\hsc u(\xi)|^2 \, \rd \xi, 
\eeq
where $\langle \xi \rangle := (1+|\xi|^2)^{1/2}$.
We define $\|\cdot\|_{H^s_\hsc(\Rea^d)}$ to be the right-hand side of \eqref{eq:Hhnorm};
 this definition
 means that 
 $\|\cdot\|_{H^s_\hsc(\Rea^d)} = \|\cdot\|_{H^s_k(\Rea^d)}$; we use this clashing notation to avoid writing $H^s_{k^{-1}}(\Rea^d)$ and $\|\cdot\|_{H^s_{k^{-1}}(\Rea^d)}$.
The weighted spaces $H^s_\hsc(\Gamma)$ are then equal to $H^s_k(\Gamma)$ defined in 
\S\ref{sec:weighted_norms}.

In \S\ref{sec:331}-\S\ref{sec:335} we review basic facts about semiclassical pseudodifferential operators, with our default references being \cite{Zw:12} and \cite[Appendix E]{DyZw:19}. 
Homogeneous -- as opposed to semiclassical --
versions of these results can be found in, e.g.,
\cite[Chapter 7]{Ta:96}, \cite[Chapter 7]{SaVa:02}, \cite[Chapter
6]{HsWe:08}.

\subsubsection{Phase space, symbols, quantisation, and semiclassical pseudodifferential operators.}\label{sec:331}

For simplicity of exposition, we begin by discussing semiclassical pseudodifferential operators on $\Rea^d$, and then 
outline in \S\ref{sec:ReatoGamma} below how to extend the results from $\Rea^d$ to $\Gamma$.

The set of all possible positions $x$ and momenta (i.e.~Fourier variables) $\xi$ is denoted by $T^*\Rea^d$; this is known informally as ``phase space''. Strictly, $T^*\Rea^d :=\Rea^d \times (\Rea^d)^*$, i.e. the cotangent bundle to $\mathbb{R}^d$, but 
for our purposes, we can consider $T^*\Rea^d$ as $\{(x,\xi) : \bx\in \Rea^d, \xi\in\Rea^d\}$.

A symbol is a function on $T^*\Rea^d$ that is also allowed to depend on $\hsc$, and can thus be considered as an $\hsc$-dependent family of functions.
Such a family $a=(a_\hsc)_{0<\hsc\leq\hsc_0}$, with $a_\hsc \in C^\infty({T^*\mathbb R^d})$, 
is a \emph{symbol
of order $m$}, written as $a\in S^m(T^*\Rea^d)$,
if for any multiindices $\alpha, \beta$
\beq\label{eq:Sm}
| \partial_x^\alpha \partial^\beta_\xi a_\hsc(x,\xi) | \leq C_{\alpha, \beta} 
\langle \xi\rangle^{m -|\beta|}
\quad\tfa (x,\xi) \in T^* \Rea^d \text{ and for all } 0<\hsc\leq \hsc_0,
\eeq
(recall that $\langle\xi\rangle:= (1+ |\xi|^2)^{1/2}$) and 
$C_{\alpha, \beta}$ does not depend on $\hsc$; see \cite[p.~207]{Zw:12}, \cite[\S E.1.2]{DyZw:19}. 

For $a \in S^m$, we define the \emph{semiclassical quantisation} of $a$, denoted by $a(x,\hsc D):\mathcal{S}(\mathbb{R}^d)\to \mathcal{S}(\mathbb{R}^d)$, by  
\beq \label{eq:quant}
a(x,\hsc D) v(x) := (2\pi \hsc)^{-d} \int_{\Rea^d} \int_{\Rea^d} 
\exp\big(\ri (x-y)\cdot\xi/\hsc\big)\,
a(x,\xi) v(y) \,\rd y  \rd \xi
\eeq
where $D:= -\ri \partial$; see, e.g., \cite[\S4.1]{Zw:12} \cite[Page 543]{DyZw:19}.
We also write $a(x,\hsc D)= \Op_\hsc(a)$. The integral in \eqref{eq:quant} need not converge, and can be understood \emph{either} as an oscillatory integral in the sense of \cite[\S3.6]{Zw:12}, \cite[\S7.8]{Ho:83}, \emph{or} as an iterated integral, with the $y$ integration performed first; see \cite[Page 543]{DyZw:19}.

Conversely, if $A$ can be written in the form above, i.\,e.\ $A =a(x,\hsc D)$ with $a\in S^m$, we say that $A$ is a \emph{semiclassical pseudo-differential operator of order $m$} and
we write $A \in \Psi_{\hsc}^m$. We use the notation $a \in \hsc^l S^m$  if $\hsc^{-l} a \in S^m$; similarly 
$A \in \hsc^l \Psi_\hsc^m$ if 
$\hsc^{-l}A \in \Psi_\hsc^m$. We define $\Psi^{-\infty}_\hsc = \cap_m \Psi^{-m}_\hsc$.

\begin{theorem}\mythmname{Composition and mapping properties of
semiclassical pseudo-differential operators \cite[Theorem 8.10]{Zw:12}, \cite[Propositions E.17, E.19, and E.24]{DyZw:19}}\label{thm:basicP} If $A\in \Psi_{\hsc}^{m_1}$ and $B  \in \Psi_{\hsc}^{m_2}$, then
\begin{itemize}
\item[(i)]  $AB \in \Psi_{\hsc}^{m_1+m_2}$.
\item[(ii)]  For any $s \in \mathbb R$, $A$ is bounded uniformly in $\hsc$ as an operator from $H_\hsc^s$ to $H_\hsc^{s-m_1}$. 
\end{itemize}
\end{theorem}

A key fact we use below is that if $\psi\in C_{\rm comp}^\infty(\mathbb{R})$ then, given $s\in\mathbb{R}$, $N>0$ and $\hsc_0>0$ there exists $C>0$ such that for all $\hsc\leq \hsc_0$, 
\beq\label{eq:frequencycufoff}
\N{\psi(|\hsc D|)}_{H_\hsc^{s}(\Rea^d)\to H_\hsc^{s+N}(\Rea^d)} \leq C;
\eeq
this can easily be proved using the semiclassical Fourier transform, since $\psi(|\hsc D|)$ is a Fourier multiplier (i.e., $\psi(|\hsc D|)$ is defined by \eqref{eq:quant} with $a(x,\xi)= \psi(|\xi|)$, which is independent of $x$).

\subsubsection{The principal symbol map $\sigma_{\hsc}$.}
Let the quotient space $ S^m/\hsc S^{m-1}$ be defined by identifying elements 
of  $S^m$ that differ only by an element of $\hsc S^{m-1}$. 
For any $m$, there is a linear, surjective map
$$
\sigma^m_{\hsc}:\Psi_\hsc ^m \to S^m/\hsc S^{m-1},
$$
called the \emph{principal symbol map}, 
such that, for $a\in S^m$,
\beq\label{eq:symbolone}
\sigma_\hsc^m\big(\Op_\hsc(a)\big) = a \quad\text{ mod } \hsc S^{m-1};
\eeq
see \cite[Page 213]{Zw:12}, \cite[Proposition E.14]{DyZw:19} (observe that \eqref{eq:symbolone} implies that 
$\operatorname{ker}(\sigma^m_{\hsc}) = \hsc\Psi_\hsc ^{m-1}$).

When applying the map $\sigma^m_{\hsc}$ to 
elements of $\Psi^m_\hsc$, we denote it by $\sigma_{\hsc}$ (i.e.~we omit the $m$ dependence) and we use $\sigma_{\hsc}(A)$ to denote one of the representatives
in $S^m$ (with the results we use then independent of the choice of representative).
The key properties of the principal symbol that we use below are that 
\beq \label{eq:multsymb}
\sigma_{\hsc}(AB)=\sigma_{\hsc}(A)\sigma_{\hsc}(B)
\quad\tand \quad 
\sigma_\hsc(A^*) =\overline{\sigma_\hsc(A)};
\eeq
see \cite[Proposition E.17]{DyZw:19}.
               
\subsubsection{Extension of the above results from $\Rea^d$ to $\Gamma$}\label{sec:ReatoGamma}

While the definitions above are written for operators on $\Rea^d$, semiclassical pseudodifferential operators and all of their properties above have analogues on compact manifolds  (see e.g.~\cite[\S14.2]{Zw:12},~\cite[\S E.1.7]{DyZw:19}). Roughly speaking, the class of semiclassical pseudodifferential operators of order $m$ on a compact manifold $\Gamma$, $\Psi^m_\hsc(\Gamma)$, are operators that, in any local coordinate chart, have kernels of the form~\eqref{eq:quant} where the function $a\in S^m$ modulo a remainder operator $R$ that has the property that
\begin{equation}
\label{e:remainder}
\|R\|_{H_\hsc^{-N}\to H_{\hsc}^N}\leq C_{N} \hsc^N;
\end{equation}
we say that an operator $R$ satisfying~\eqref{e:remainder} is $O(\hsc^\infty)_{\Psi_\hsc^{-\infty}}$.

Semiclassical pseudodifferential operators on manifolds continue to have a natural principal symbol map
$$
\sigma_{\hsc}:\Psi_\hsc^m\to S^m(T^*\Gamma)/\hsc S^{m-1}(T^*\Gamma)
$$
where now $S^m(T^*\Gamma)$ is the class of functions on $T^*\Gamma$, the cotangent bundle of $\Gamma$, that satisfy the estimate~\eqref{eq:Sm}.
The properties \eqref{eq:multsymb} hold as before.

Finally, there is a noncanonical quantisation map $\Op_{\hsc}:S^m(T^*\Gamma)\to \Psi^m(\Gamma)$ (involving choices of cut-off functions and coordinate charts) that satisfies
$$
\sigma_\hsc (\Op_{\hsc}(a))=a,
$$
and for all $A\in \Psi_\hsc^m(\Gamma)$ there exists $a\in S^m(T^*\Gamma)$ such that 
$$
A=\Op_{\hsc}(a)+O(\hsc^\infty)_{\Psi_\hsc^{-\infty}}.
$$

\subsubsection{Local coordinates}\label{sec:local}

Near the boundary $\Gamma$, we use Riemannian/Fermi normal coordinates $(x_{1},x')$, in which $\Gamma$
is given by $\{x_{1}=0\}$, $\Oi = \{ x_1<0\},$ $\Oe = \{x_1>0\}$, and so $\partial_n = \partial_{x_1}$. We write $D' = -\ri \partial_{x'}$.
The conormal and cotangent variables are given by $(\xi_{1},\xi')$. We write $g_{\Gamma}(x')$ for the metric induced on $\Gamma$ from the Euclidean metric on $\mathbb{R}^d$, and $|\cdot|_g$ for the corresponding norm (thus abbreviated $g_\Gamma$ to $g$ in the subscript). 
The trace operators $\gamma^\pm$ are such that
$$
(\gamma^{\pm}u)(x')=\lim_{x_1\to 0^{\pm}}u(x_1,x'),\qquad u\in C^\infty(\mathbb{R}^d)
$$ 
and $\gamma^*$ defined by \eqref{eq:traceadjoint} satisfies $\gamma^* \phi(x)= \phi(x') \delta(x_1)\sqrt{|\det g_{\Gamma}(x')|}$.
Finally, recall that, in these local coordinates, the conormal bundle to $\Gamma$, $N^*\Gamma$, consists of $(x,\xi)$ of the form $(0,x',\xi_1,0)$.

\subsubsection{Ellipticity}\label{sec:335}

We now give a simplified version of the general semiclassical ellipticity estimate. 

\begin{theorem}[Simplified elliptic estimate]\label{thm:elliptic}
Assume that $\Gamma$ is $C^\infty$. If $B\in \Psi^\ell_\hsc(\Gamma)$ is \emph{elliptic}, i.e., there exists $\delta>0$ such that
\beqs
\inf_{(x',\xi')\in T^*\Gamma} \big|\sigma_\hsc(B)(x',\xi')\langle \xi'\rangle^{-\ell}\big|\geq \delta,
\eeqs
then there exists $\hsc_0>0$ such that, for all $0<\hsc\leq \hsc_0$, $B^{-1}\in \Psi^{-\ell}_{\hsc}(\Gamma)$.
\end{theorem}

\bpf[References for the proof]
This follows from \cite[Theorem E.33]{DyZw:19} (and the second remark afterwards) with $P=I, A=I, B_1=B$, $m=0$, and $\ell=0$.
In simplifying this general result, we use that (i) since $A=I$, the $O(\hsc^\infty)$ error term on the right-hand side of \cite[Equation E.2.9]{DyZw:19} can be absorbed on the left-hand side, (ii) since $\Gamma$ is compact $B \in \Psi^\ell_\hsc(\Gamma)$ is compactly supported.
\epf

\begin{corollary}[Upper and lower bounds on $R$]\label{cor:Rbound}
If $\Reg$ satisfies Assumption \ref{ass:Reg}, then given $k_0>0$ and $t\in \Rea$ there exists $C_1,C_2>0$ such that for all $k\geq k_0$,
\beq\label{eq:Rbound1}
\frac{C_1}{k} \leq         \N{R}_{H^t_k(\Gamma)\rightarrow H^{t+1}_k(\Gamma)}\leq \frac{C_2}{k}
\eeq
and 
\beq\label{eq:Rbound2}
\frac{k}{C_2} \leq \N{\Reg^{-1}}_{H^{t+1}_k(\Gamma)\rightarrow H^{t}_k(\Gamma)}\leq \frac{k}{C_1}.
\eeq
\end{corollary}

\bpf
Once we prove the upper bounds in \eqref{eq:Rbound1} and \eqref{eq:Rbound2},
the lower bounds then follow. Indeed, the upper bound in \eqref{eq:Rbound1} implies the lower bound in \eqref{eq:Rbound2}, and vice versa.

By assumption, $R= k^{-1} \widetilde{\Reg}$ with $\widetilde{\Reg}\in \Psi_\hsc^{-1}(\Gamma)$. Therefore, by Part (ii) of Theorem \ref{thm:basicP}, given $\hsc_0>0$ and $t\in \Rea$ there exists $C_2>0$ such that 
\beqs
\big\|\widetilde R\big\|_{H^t_\hsc(\Gamma) \rightarrow H^{t+1}_\hsc(\Gamma)}\leq C_2 \quad\tfa 0<\hsc\leq \hsc_0;
\eeqs
the upper bound in \eqref{eq:Rbound1} immediately follows since $R=k^{-1}\widetilde{\Reg}$.
By assumption, $\widetilde R$ is elliptic, and thus invertible by Theorem \ref{thm:elliptic}. Indeed, given $\hsc_0>0$ and $t\in \Rea$, there exists $C_1>0$ such that 
\beqs
\big\|\widetilde \Reg^{-1}\big\|_{H^{t+1}_\hsc(\Gamma) \rightarrow H^{t}_\hsc(\Gamma)}\leq (C_1)^{-1} \quad\tfa 0<\hsc\leq \hsc_0;
\eeqs
the upper bound in \eqref{eq:Rbound2} immediately follows.
\epf

\subsubsection{Sharp G\aa rding inequality}

\begin{theorem}\label{thm:sharpG}
If $\Gamma$ is $C^\infty$ and $A\in \Psi_\hsc^\ell(\Gamma)$ with $\Re \sigma_\hsc(A)\geq 0$ on $T^*\Gamma$, then there exists $C>0$ and $\hsc_0$ such that, for all $0<\hsc\leq \hsc_0$,
\beqs
\Re\big\langle A \phi,\phi\big\rangle_\Gamma \geq - C\hsc \N{\phi}^2_{H^{(\ell-1)/2}_\hsc(\Gamma)} \quad\tfa \phi \in H^{\ell/2}(\Gamma).
\eeqs
\end{theorem}

\bpf[References for the proof]
This follows from \cite[Proposition E.23]{DyZw:19} using the fact that every $A\in \Psi_\hsc^\ell(\Gamma)$ is compactly supported since $\Gamma$ is compact.
\epf

\subsubsection{Microlocality of pseudodifferential operators}
We next recall the fact that pseudodifferential operators act microlocally (i.e., pseudo locally in phase space). We include here the following lemma which follows from the more general statements in \cite[E.2.4-E.2.5]{DyZw:19} or \cite[Theorem 9.5]{Zw:12}.

\begin{lemma}
\label{l:disjoint}
If $K\Subset T^*\Gamma$, $a,b\in C_{\rm comp}^\infty(K)$ and there exists $c>0$ such that 
$$
d\big( \supp a, \supp b\big) >c, \quad{ then } \quad
\Op_{\hsc}(a)\Op_{\hsc}(b)=O(\hsc^\infty)_{\Psi_\hsc^{-\infty}}.
$$
\end{lemma}

\subsection{Restriction of pseudodifferential kernels to submanifolds }

We recall in this section a simplified version of~\cite[Lemma 4.22]{Ga:19} which describes the operator 
\beq\label{eq:gammaAgamma}
\gamma^{\pm} A\gamma^*
\eeq
when $A\in \Psi^m_{\hsc}(\mathbb{R}^d)$.
The motivation for considering operators of the form \eqref{eq:gammaAgamma} is the following.
Let $L$ be a vector field equal to $\partial_n$ in a neighbourhood of $\Gamma$
(where $n$ is the outward-pointing unit normal vector to $\Oi$). Then 
 with $\cR_k$ defined by \eqref{eq:free_resolvent}, 
for $\Im k \geq 0$, 
\begin{align}
\label{eq:SHdef}
&S_k=\gamma \cR_k \gamma^*, 
\hspace{3cm} H_k = \gamma^\pm L \cR_k L^* \gamma^*,\\
& K_k = \mp \half I + \gamma^\pm \cR_k L^* \gamma^*, 
\hspace{1cm} K_k' = \pm \half I + \gamma^\pm L \cR_k  \gamma^*;\label{eq:KK'def}
\end{align}
see \cite[Page 202 and Equation 7.5]{Mc:00}. That is, $S_k, K_k \pm I/2$, $K_k' \mp I/2$, and $H_k$ can all be written in the form \eqref{eq:gammaAgamma} for suitable $A$ involving $\cR_k$, $L$, and $L^*$.

In the next lemma, we use the notions of conic sets $V\subset T^*\mathbb{R}^d\setminus \{0\}$ and conic neighborhoods thereof. Here, we say that $V\subset T^*\mathbb{R}^d\setminus\{0\}$ is  \emph{conic} if for all $\lambda>0$
$$
(x,\xi)\in V \quad\text{ implies }\quad (x,\lambda\xi)\in V.
$$
For a conic set $V\subset T^*\mathbb{R}^d$, we say that $U$ is a \emph{conic neighborhood of $V$} if $U\subset T^*\mathbb{R}^d\setminus \{0\}$ is an open conic set containing the closure of $V$ (as a subset of $T^*\mathbb{R}^d\setminus \{0\}$).

\begin{lemma}\label{l:computeLayer}
Suppose that $\Gamma\subset \mathbb{R}^d$ is an embedded hypersurface. Let $A\in \Psi^m_{\hsc}(\mathbb{R}^d)$ with $A=\Op_{\hsc}(a)+O(\hsc^\infty)_{\Psi_{\hsc^{-\infty}}}$ and suppose that
\begin{equation}
\label{e:classical1a}
\left|\partial_x^\alpha\partial_{\xi}^\beta\Bigg(a(x,\xi)- \sum_{j=-1}^ma_j(x,\xi)\Bigg)\right|\leq C_{\alpha\beta}\langle\xi\rangle^{-2-|\beta|}\quad \tfor |\xi|\geq 1,
\end{equation}
with $a_j$ homogeneous of degree $j$ in $\xi$ (i.e., $a_j(x,\lambda\xi)=\lambda^j a_j(x,\xi)$ for $\lambda>0$) and there is an open conic neighbourhood, $U$, of $N^*\Gamma$ such that 
\begin{equation}
\label{e:classical2a}
a_j(x,\xi)=(-1)^ja_j(x,-\xi)\quad \tfor (x,\xi)\in U\cap T^*\mathbb{R}^d\setminus\{0\}.
\end{equation}
Then $\gamma^\pm A \gamma^*\in \hsc^{-1}\Psi^{m+1}$ and, in coordinates $(x,\xi)$ with $\Gamma=\{(0,x')\}$, 
\beq\label{eq:prinsymb}
\sigma_{\hsc}(\gamma^{\pm}A\gamma^*)=\lim_{x_1\to 0^{\pm}} \mathcal{F}_{\hsc}^{-1}\Big(\sigma_{\hsc}(A)(x,\cdot,\xi')\Big)(x_1).
\eeq
\end{lemma}

The non-semiclassical analogue of Lemma \ref{l:computeLayer} can be found in, e.g., \cite[Chapter 7, \S11]{Ta:96} and \cite[Theorem 8.4.3]{HsWe:08}. These non-semiclassical results are slightly simpler because there one is not concerned with the behavior of the symbol inside a compact set and hence one works directly with homogeneous expansions of symbols; i.e.~the assumption \eqref{e:classical1a} is immediate from the definition of a polyhomogeneous pseudodifferential operator.

A key ingredient in the proof of Lemma \ref{l:computeLayer} is the following preparatory lemma.

\begin{lemma}\label{l:computeLayerOld}
Suppose that $A\in \Psi^m_\hsc(\mathbb{R}^d)$, $\Gamma=\{x_1=0\}$, and there are $\e>0$, $a\in S^m$ such that $A=\Op_{\hsc}(a) +O(\hbar^\infty)_{\Psi_\hsc^{-\infty}}$,
\begin{equation}
\label{e:classical}
\left|\partial_x^\alpha\partial_{\xi}^\beta\Bigg(a(x,\xi)- \sum_{j=-1}^ma_j(x,\xi)\Bigg)\right|\leq C_{\alpha\beta}\langle\xi\rangle^{-2-|\beta|}\quad \tfor |\xi|\geq 1\tand |x_1|<\e
\end{equation}
with $a_j$ homogeneous of degree $j$ in $\xi$ (i.e., $a_j(x,\lambda\xi)=\lambda^j a_j(x,\xi)$ for $\lambda>0$) and satisfying
\beq\label{e:classical2}
\left|\partial_{x}^{\alpha}\partial_{\omega'}^\beta \Big(a_j(x,\xi_1,\omega')-\sum_{\ell=-1}^j \widetilde{a}_{j,\ell}(x,\omega')\xi_1^\ell\Big)\right|\leq C_{\alpha\beta}|\xi_1|^{-2-|\beta|}\quad \tfor 
 |\omega'|\leq 1, |x_1|<\e, \tand
|\xi_1|\geq 1,
\eeq
where $\widetilde{a}_{j,\ell}(x,\omega')\in C^\infty(\mathbb{R}^{d-1}_{x'}\times B(0,2)_{\omega'})$.

Then $\gamma^{\pm}A\gamma^*\in \hsc^{-1}\Psi^{m+1}_\hsc(\Gamma)$ with semiclassical principal symbol given by
\eqref{eq:prinsymb}.
\end{lemma}
\begin{proof}
First observe that, for $u\in C^\infty(\Gamma)$,
\begin{align*}
&\gamma^{\pm}A\gamma^* u(x')= (2\pi \hbar)^{-d}\lim_{x_1\to 0^{\pm}}\int_{\Rea^d}\int_{\Gamma} \re^{\frac{\ri}{\hbar}(x_1\xi_1+\langle x'-y',\xi'\rangle) }a(x_1,x',\xi_1,\xi')u(y')\sqrt{|\det g_{\Gamma}(y')|}dy'\rd \xi\\
&= (2\pi \hbar)^{-(d-1)}\int_{\Rea^{d-1}}\int_{\Gamma} \re^{\frac{\ri}{\hbar}\langle x'-y',\xi'\rangle}\left(\lim_{x_1\to 0^{\pm}}(2\pi \hbar)^{-1}\int_{\Rea} \re^{\frac{\ri}{\hbar}x_1\xi_1 }a(x_1,x',\xi_1,\xi')\rd \xi_1\right)u(y')\sqrt{|\det g_{\Gamma}(y')|}dy'\rd \xi'.
\end{align*}
Therefore, to prove the lemma, we only need to show that 
\begin{align*}
I_{\pm}(a):=(2\pi \hbar)^{-1}\lim_{x_1\to 0^{\pm}}\int_{\Rea} \re^{\frac{\ri}{\hbar}x_1\xi_1 }a(x_1,x',\xi_1,\xi')\rd \xi_1\in \hbar^{-1}S^{m+1}(T^*\Gamma).
\end{align*}

We start by decomposing $a$ into its integrable and non-integrable pieces (with respect to $\xi_1$).  Let $\varphi\in C_{\rm comp}^\infty((-2,2);[0,1])$ with $\varphi \equiv 1$ on $[-1,1]$, let $\psi:=1-\varphi$, and write 
$$
a=a_L+a_{H},\qquad a_H=\sum_{j=-1}^m a_j(x,\xi)\psi(|\xi|).
$$
By~\eqref{e:classical} and the fact that $a\in S^m(T^*\Rea^d)$, $a_L\in S^{\min(-2,m)}(T^*\mathbb{R}^d)$ (where the minimum is achieved at $m$ only when all the $a_j$s equal zero). Since 
$$
I_{\pm}(a_L)= (2\pi \hbar)^{-1}\lim_{x_1\to 0^{\pm}}\int_{\Rea} \re^{\frac{\ri}{\hbar}x_1\xi_1}a_L(x,\xi)\rd \xi_1.
$$ 
and, for $n >1$,
$$
\int_{\Rea} \langle \xi\rangle^{-n}\rd \xi_1\leq C\langle \xi'\rangle^{1-n}, 
$$
we have
$$
I_{+}(a_L)=I_-(a_L)\in \hbar^{-1}S^{\min(-1,m+1)}(T^*\Gamma).
$$
Now, using the change of variables $\xi_1 \mapsto \xi_1 \langle\xi'\rangle$, the homogeneity of $a_j(x,\xi)$, and the fact that $\phi +\psi=1$, we have
\begin{align*}
I_{\pm}(a_j\psi)&=(2\pi \hbar)^{-1}\lim_{x_1\to 0^{\pm}}\int_\Rea \re^{\frac{\ri}{\hbar}x_1\xi_1} a_j(x,\xi)\psi(|\xi|)\rd \xi_1 \\
&=(2\pi \hbar)^{-1}\langle \xi'\rangle^{j+1}\lim_{x_1\to 0^{\pm}}\int_\Rea \re^{\frac{\ri}{\hbar}x_1\langle \xi'\rangle \xi_1}\psi\big( \big|(\langle \xi'\rangle \xi_1,\xi')\big|\big)\varphi\big(|\xi_1|/2\big)a_j\left(x,\xi_1,\dfrac{\xi'}{\langle\xi'\rangle}\right)\rd \xi_1\\
&\qquad +(2\pi \hbar)^{-1}\langle \xi'\rangle^{j+1}\lim_{x_1\to 0^{\pm}}\int_\Rea \re^{\frac{\ri}{\hbar}x_1\langle \xi'\rangle \xi_1}\psi\big(|\xi_1|/2\big)a_j\left(x,\xi_1,\dfrac{\xi'}{\langle\xi'\rangle}\right)\rd \xi_1\\
&=:(2\pi \hbar)^{-1}\langle \xi'\rangle^{j+1}\big(I_{1,\pm,j}+I_{2,\pm,j}\big).
\end{align*}
Since $(2\pi \hbar)^{-1}\langle \xi'\rangle^{j+1}\in \hbar ^{-1}S^{j+1}(T^*\Gamma)$, we need only show that $I_{1,\pm,j},\,I_{2,\pm,j}\in S^0(T^*\Gamma)$.

We first study $I_{1,\pm,j}$. By the definition of $\varphi$, uniformly in $|x_1|$ small, 
$$
\psi\big( |(\langle \xi'\rangle \xi_1,\xi')|\big)\,\varphi\big(|\xi_1|/2\big)\,a_j\left(x,\xi_1,\dfrac{\xi'}{\langle\xi'\rangle}\right)\in C_{\rm comp}^\infty\big((-4,4)_{\xi_1};S^0(T^*\Gamma)\big),
$$
and thus
\beqs
I_{1,+,j}=I_{1,-,j}\in S^0(T^*\Gamma),
\eeqs
where we use that $\partial\psi(|(\langle \xi'\rangle \xi_1,\xi')|)$ is compactly supported in $\xi'$ to see that derivatives falling on this term are harmless.

Finally, we consider $I_{2,\pm,j}$. Observe that
by the chain rule and the fact that $\xi'/\langle\xi'\rangle \in S^0(T^*\Gamma)$,
 to obtain $I_{2,\pm,j}\in S^0(T^*\Gamma)$ we only need to show that
$$
\lim_{x_1\to \pm 0}\int_\Rea \re^{\frac{\ri}{\hsc}x_1\langle \xi'\rangle \xi_1}\psi\big(|\xi_1|/2\big)a_j(x,\xi_1,\omega')\rd \xi_1\in C^\infty\big( \mathbb{R}_{x'}^{d-1}\times B(0,2)_{\omega'}\big).
$$
To do this, put
$$
q_j(x,\omega',\xi_1):= \sum_{\ell=-1}^j \widetilde{a}_{j,\ell}(x,\omega')\xi_1^\ell\in  C^\infty\big(\mathbb{R}^{d}_x\times B(0,2)_{\omega'};\mathcal{S}'(\mathbb{R})\big),
$$
where we interpret $\xi_1^{-1}$ as $\rm{p.v.} \xi_1^{-1}$ (see, e.g., \cite[Page 166]{Mc:00}), and let $r_{j}(x,\omega',\xi_1)=\psi(|\xi_1|/2)a_j(x,\omega',\xi_1)-q_j(x,\omega',\xi_1)$. Observe that $r_j\in C^\infty(\mathbb{R}^{d}_{x}\times B(0,2)_{\omega'};\mathcal{S}'(\mathbb{R}))$ 
and, by \eqref{e:classical2}, 
$$
|r_j(x,\xi_1,\omega')|\leq C|\xi_1|^{-2}\quad \tfor |\xi_1|\geq 2.
$$
Therefore, since the Fourier transform of an $L^1$ function is continuous,
$$
r_{j,\Gamma}(x,\omega'):=(2\pi h)^{-1}\int_\Rea \re^{\frac{\ri}{\hsc}x_1\langle \xi'\rangle\xi_1}r_j(x,\xi_1,\omega')\rd \xi_1
$$
is continuous in $x_1$ and satisfies, 
$$
\big|\partial_{x'}^\alpha\partial_{\omega'}^\beta r_{j,\Gamma}(0,x',\omega')\big|=  \int_\Rea 
\Big|
\partial_{x'}^\alpha\partial_{\omega'}^\beta
r_j(0,x',\xi_1,\omega')\Big|\rd \xi_1\leq C_{\alpha\beta}.
$$
Therefore, we only need to consider the term $q_j$. For this, recall that
$$
\xi_1^j=\mathcal{F}_{\hsc}((\hsc D_{x_1})^j \delta_0(x_1)), \quad j\geq 0,\qquad {\rm p.v. }\,\xi_1^{-1}= \mathcal{F}_{\hsc}\left(\dfrac{\ri}{2h}\operatorname{sgn}(x_1)\right)
$$
where $\operatorname{sgn}(x):= 1$ for $x>0$ and $:= -1$ for $x<0$.
Let
\begin{align*}
q_{j,\Gamma}(x,\xi')&:=(2\pi h)^{-1}\int_\Rea \re^{\frac{\ri}{\hsc}x_1\langle \xi'\rangle\xi_1}\left(\sum_{\ell=-1}^j\widetilde{a}_{j,\ell}(x,\omega')\xi_1^\ell \right)\rd \xi_1\\
&=\sum_{\ell=0}^j \widetilde{a}_{j,\ell}(x,\omega')(\hsc D_{x_1})^\ell \delta(x_1\langle\xi'\rangle)+ \dfrac{\ri}{2h}\widetilde{a}_{j,-1}(x,\omega')\operatorname{sgn}(x_1\langle\xi'\rangle).\end{align*}
Therefore,
$$
\lim_{x_1\to 0^{\pm}} q_{j,\Gamma}(x_1,x',\xi')=\pm\frac{\ri}{2h}\widetilde{a}_{j,-1,}(0,x',\omega')\in C^\infty\big(\mathbb{R}^{d-1}_{x'}\times B(0,2)_{\omega'}\big),
$$
and the proof is complete.
\end{proof}

\begin{proof}[Proof of Lemma \ref{l:computeLayer}]
By a partition of unity and pseudolocality of pseudodifferential operators, we can assume that $\supp a$ is contained in a small open subset, $V$, of $\mathbb{R}^d$. 
Let $V'\subset \mathbb{R}^d$ and $\Phi:V'\to V$ be a diffeomorphism such that 
$$
\{ \Phi(0,y')\,:\, (0,y')\in V\}=V\cap \Gamma.
$$
To prove the lemma, we observe that, by~\cite[Theorem 9.9]{Zw:12}, \cite[Theorem 18.1.17]{Ho:85},
$$
\Phi^*\Op_{\hsc}(a)(\Phi^{-1})^* =\Op_{\hsc}(b)+O(\hsc^\infty)_{\Psi^{-\infty}_\hsc},
$$
where  $b\in S^m$ satisfies
\beq\label{eq:Hormander}
b(y,\eta)-\sum_{|\alpha|\leq N-1}\frac{1}{\alpha!} \partial_\xi ^\alpha a\big(\Phi(y),[(\partial \Phi)^{-1}]^t(\Phi(y))\eta\big)(\hsc D_z)^\alpha \re^{\frac{\ri}{\hsc}\langle \rho_{\Phi(y)}(z),\eta\rangle}\Big|_{z=\Phi(y)}\in \hsc^{\lceil \frac{N}{2}\rceil}S^{m-\lceil \frac{N}{2}\rceil},
\eeq
and
$$
\rho_x(z)=\Phi^{-1}(z)-\Phi^{-1}(x)-\partial\Phi^{-1}(x)(y-x).
$$

Now, if $\varphi\in C_c^\infty(\mathbb{R})$ with $\varphi\equiv 1$ on $[-1,1]$, then, by \eqref{e:classical1a},  $a-\sum_{j=-1}^m(1-\varphi)a_j\in S^{\min\{m,-2\}}$. Therefore, by writing
\beqs
a= \Big(a-\sum_{j=-1}^m(1-\varphi)a_j\Big) + \sum_{j=-1}^m(1-\varphi)a_j
\eeqs
and then changing variables and using \eqref{eq:Hormander},
$$
\Big| \partial_{y}^\alpha \partial_\eta^\beta \Big( b(y,\eta)-\sum_{j=-1}^m\sum_{|\alpha|\leq 2j+3}b_{\alpha,j}(y,\eta)\Big)\Big|\leq C_{\alpha\beta} \langle \eta\rangle^{-2-|\beta|},\quad|\eta|\geq 1, 
$$
(where the significance of $2j+3$ in the index of the sum is that $j- \lceil(2 j+3)/2\rceil = -2$)
where 
$$
b_{\alpha,j}(y,\eta):=a_{j,\alpha}(y,\eta)(\hsc D_z)^\alpha \re^{\frac{\ri}{\hsc}\langle \rho_{\Phi(y)}(z),\eta\rangle}\Big|_{z=\Phi(y)},\quad a_{j,\alpha}(y,\eta):=\frac{1}{\alpha!} \partial_\xi ^\alpha a_j\big(\Phi(y),[(\partial \Phi)^{-1}]^t(\Phi(y))\eta\big).
$$
Since $a_j$ is homogeneous degree $j$, $a_{j,\alpha}$ is homogeneous degree $j-|\alpha|$. Next, since $\rho_{x}(z)$ vanishes to order $2$ at $z=x$, direct calculation shows that
$$
(\hsc D_z)^\alpha \re^{\frac{\ri}{\hsc}\langle \rho_{\Phi(y)}(z),\eta\rangle}\Big|_{z=\Phi(y)}=\sum_{\ell=0}^{\lfloor \frac{|\alpha|}{2}\rfloor} \widetilde{b}_{\alpha,\ell}(y,\eta),
$$
where 
$\widetilde{b}_{\alpha,\ell}(y,\eta)$ is a polynomial in $\eta$, and hence $\hsc^\ell\widetilde{b}_{\alpha,\ell}\in S^\ell$ and is homogeneous of degree $\ell$. In particular, grouping terms with a given homogeneity in $\eta$,  $b$ satisfies
\begin{equation}
\label{e:classical1b}
\Big|\partial_y^\alpha\partial_\eta^\beta \Big(b(y,\eta)-\sum_{j=-1}^m b_j(y,\eta)\Big)\Big|\leq C_{\alpha\beta} \langle \eta\rangle^{-2-|\beta|},\, |\eta|\geq 1,
\end{equation}
with $b_j$ homogeneous of degree $j$ and defined by
$$
b_j:=\sum_{\ell=j}^{m}\sum_{|\alpha|=0}^{2(\ell-j)} a_{\ell,\alpha}\widetilde{b}_{\alpha,j-\ell+|\alpha|}.
$$

We claim that there is a conic neighbourhood $U'$ of $\{(0,y',\eta_1,0)\,:\, (0,y')\in V,\, \eta_1\in \mathbb{R}\setminus[-1,1]\}$ such that 
\begin{equation}
\label{e:oddness}
b_j(y,\eta)=(-1)^jb_j(y,-\eta),\qquad (y,\eta)\in U'.
\end{equation}
To see this, first note that  for $\eta_1\in \mathbb{R}$, $([( \partial \Phi)^{-1}]^t(\Phi(y)))(\eta_1,0)\in N^*\Gamma$ and therefore there is a conic neighbourhood, $U'$, of $\{(0,y',\eta_1,0)\,:\, (0,y')\in V,\, \eta_1\in \mathbb{R}\setminus[-1,1]\}$ such that 
$$
\big\{ \big(\Phi(y),[( \partial \Phi)^{-1}]^t(\Phi(y))\eta\big)\,:\, (y,\eta)\in U'\big\}\subset U.
$$
Therefore, since the $a_j$ satisfy~\eqref{e:classical2a},
$$
a_{j,\alpha}(y,\eta)=(-1)^{j-|\alpha|}a_{j,\alpha}(y,-\eta),\qquad (y,\eta)\in U'.
$$
Next, since $\widetilde{b}_{\alpha,\ell}(y,\eta)$ is a polynomial of degree $\ell$ in $\eta$,
$$
\widetilde{b}_{\alpha,\ell}(y,\eta)=(-1)^\ell \widetilde{b}_{\alpha,\ell}(y,-\eta).
$$
Thus, we have, for $(y,\eta)\in U'$,
\begin{align*}
b_j(y,\eta)&=\sum_{k=j}^{m}\sum_{|\alpha|=0}^{2(k-j)} a_{k,\alpha}(y,\eta)\widetilde{b}_{\alpha,j-k+|\alpha|}(y,\eta)\\
&=\sum_{k=j}^{m}\sum_{|\alpha|=0}^{2(k-j)} (-1)^{k-|\alpha|}a_{k,\alpha}(y,-\eta)(-1)^{j-k+|\alpha|}\widetilde{b}_{\alpha,j-k+|\alpha|}(y,-\eta)=(-1)^jb_j(y,\eta),
\end{align*}
and we have thus proved~\eqref{e:oddness}.

Now, there are $C>0$, $\e>0$ such that 
$$
\{ (y_1,y',\eta_1,\omega')\,:\, |y_1|<\e, \,|\omega'|<2,\, |\eta_1|>C\}\subset U'.
$$
Therefore, by~\eqref{e:oddness}, for $|y_1|<\e, |\omega'|<2$, and $|\eta_1|>C$
$$
b_j(y,\eta_1,\omega')=(-1)^jb_j(y,-\eta_1,-\omega')
$$
Furthermore, since $b_j$ is homogeneous degree $j$,  for $|y_1|<\e$,
\begin{equation}
\label{e:theCredits}
\partial_{y}^{\alpha}\partial_{\xi'}^{\beta}b_j(y,1,0)=(-1)^{j+|\beta|}\partial_y^{\alpha}\partial^{\beta}_{\xi'}b_j(y,-1,0).
\end{equation}
By homogeneity and Taylor's theorem, 
\begin{align}\nonumber
&\partial_{y}^\alpha\partial_{\omega'}^{\beta}b_j(y,\eta_1,\omega')\\ \nonumber
&=|\eta_1|^j\partial_{y}^\alpha\partial_{\omega'}^{\beta}\Big(b_j\big(y,\tfrac{\eta_1}{|\eta_1|},\tfrac{\omega'}{|\eta_1|}\big)\Big)\\ \nonumber
&=\sum_{|\beta_1|\leq j+1}\frac{1}{\beta_1!}|\eta_1|^{j-|\beta|-|\beta_1|}\partial_{y}^\alpha\partial_{\xi'}^{\beta_1+\beta} b_j(y,\operatorname{sgn}(\eta_1),0)(\omega')^{\beta_1} +O(|\eta_1|^{-2-|\beta|})\\ \nonumber
&=\sum_{|\beta_1|\leq j+1}\frac{1}{\beta_1!}\eta_1^{j-|\beta|-|\beta_1|} (\operatorname{sgn}(\eta_1))^{j-|\beta|-|\beta_1|}\partial_{y}^\alpha\partial_{\xi'}^{\beta_1+\beta} b_j(y,\operatorname{sgn}(\eta_1),0)(\omega')^{\beta_1} +O(|\eta_1|^{-2-|\beta|})\\
&=\sum_{|\beta_1|\leq j+1}\frac{1}{\beta_1!}\eta_1^{j-|\beta|-|\beta_1|}\partial_{y}^\alpha\partial_{\xi'}^{\beta_1+\beta} b_j(y,1,0)(\omega')^{\beta_1} +O(|\eta_1|^{-2-|\beta|}),\label{e:classical2b}
\end{align}
where we have used~\eqref{e:theCredits} in the last step.
The bound \eqref{e:classical1b}
 and the expansion \eqref{e:classical2b} show that $b$ satisfies the assumptions of Lemma \ref{l:computeLayerOld} (with $a$ replaced by $b$), and the result of this lemma then completes the proof.
\end{proof}

\section{New results about boundary-integral operators}\label{sec:4}

\subsection{The high-frequency components of the operators $S_k$, $\DL_k$, $\DL'_k$, and $H_k$.}\label{sec:4.1}

In this subsection, we prove results about the high-frequency components of the standard boundary-integral operators; these results are then used to prove bounds on $H_k$ (Theorem \ref{thm:Hk} below) and to prove that $\Breg$ and $\Bregp$ are Fredholm (i.e., in Part (i) of Theorem \ref{thm:invert}).

\begin{lemma}
\label{l:SDHHigh}
Let  $\psi\in C_{\rm comp}^\infty((-2,2))$ with $\psi\equiv 1$ in a neighbourhood of $[-1,1]$. Then, with 
 $\psi(|\hsc D|)$ defined by \eqref{eq:quant} and $\cR_k$ the free resolvent defined by \eqref{eq:fund},
$$
\begin{aligned}
\widetilde{S}^L_k&:=\gamma (1-\psi(|\hsc D|))\cR_k\gamma^*  \in\hsc\Psi_\hsc^{-1}(\Gamma),&\widetilde{S}^R_k&:=\gamma \cR_k(1-\psi(|\hsc D|))\gamma^*\in  \hsc\Psi_\hsc^{-1}(\Gamma),\\
(\widetilde{K}'_k)^{\pm,L}&:=\gamma^{\pm} (1-\psi(|\hsc D|))L\cR_k\gamma^*\in \Psi_\hsc^{0}(\Gamma),&(\widetilde{K}'_k)^{\pm,R}&:=\gamma^{\pm}L\cR_k (1-\psi(|\hsc D|))\gamma^*\in \Psi_\hsc^{0}(\Gamma),\\
(\widetilde{K}_k)^{\pm,L}&:=\gamma^{\pm}(1-\psi(|\hsc D|)) \cR_kL^*\gamma^*
\in \Psi_\hsc^{0}(\Gamma),&(\widetilde{K}_k)^{\pm,R}&:=\gamma^{\pm} \cR_kL^*(1-\psi(|\hsc D|))\gamma^*
\in \Psi_\hsc^{0}(\Gamma),\\
(\widetilde{H}_k)^{\pm,L}&:=\gamma^{\pm} (1-\psi(|\hsc D|)) L\cR_kL^*\gamma^*\in \hsc^{-1}\Psi_\hsc^{1}(\Gamma),&(\widetilde{H}_k)^{\pm,R}&:=\gamma^{\pm}  L\cR_kL^*(1-\psi(|\hsc D|))\gamma^*\in \hsc^{-1}\Psi_\hsc^{1}(\Gamma).
\end{aligned}
$$
Moreover, for $|\xi'|_g\geq 2,$
\begin{align}\label{eq:HFsym1}
&\sigma(\widetilde{S}^{L/R}_k)(x',\xi')=\frac{\hsc}{2\sqrt{|\xi'|_g^2-1}},\qquad \sigma\big((\widetilde{K}_k)^{\pm,{L/R}}\big)(x',\xi')=\pm\frac{1}{2},\\
& \sigma\big((\widetilde{K}_k')^{\pm,{L/R}}\big)(x',\xi')=\mp\frac{1}{2} ,\qquad \sigma\big((\widetilde{H}_k)^{\pm,{L/R}}\big)(x',\xi')=-\hsc^{-1}\frac{\sqrt{|\xi'|_g^2-1}}{2}. \label{eq:HFsym2}
\end{align}
\end{lemma}

Our plan to prove Lemma \ref{l:SDHHigh} is to apply Lemma \ref{l:computeLayer} with suitable choices of $A\in \Psi^m_\hsc(\Rea^d)$. 
For the results for $\widetilde{S}^{L/R}_k$, we want to let $A= (1-\psi(|\hsc D|))\cR_k$ and $A=\cR_k (1-\psi(|\hsc D|))$. These two operators are studied in the following lemma (which is similar to~\cite[Lemma 4.12]{Ga:19}).

Recall the following property of the free resolvent $\cR$ \eqref{eq:free_resolvent} (from, e.g., \cite[Theorem 4.1]{Ag:75}): for $s>1/2$ and  $f$ with $\mathcal{F}_{\hsc}(f)\in H^s(\Rea^d)$, $\mathcal{F}_{\hsc}(\cR_k f)\in\!~H^{-s}(\Rea^d)$ and
\beq\label{eq:resolvent}
\cR_kf=\lim_{\eps \to 0^+}(-\Delta-(k+i\eps)^2)^{-1}f=\hsc^2\lim_{\eps\to 0^+} \mathcal{F}_{\hsc}^{-1}\Big(\frac{\mathcal{F}_{\hsc}(f)(\xi)}{|\xi|^2-(1+\ri\eps)^2}\Big).
\eeq

\begin{lemma}
\label{l:resolveHighFreq}
Let  $\psi\in C_{\rm comp}^\infty$ with $\psi\equiv 1$ in a neighbourhood of $[-1,1]$. Then
\beqs
\cR_k(1-\psi(|\hsc D|))= (1-\psi(|\hsc D|))\cR_k = \Op_\hsc\left(\frac{\hsc^2(1-\psi(|\xi|))}{|\xi|^2-1}\right)\in \hsc^2\Psi_\hsc^{-2}(\Rea^d).
\eeqs
\end{lemma}
\begin{proof}
Since $(1-\psi(| \xi|)):H^s(\Rea^d)\to H^s(\Rea^d)$, for $f\in L^2(\Rea^d)$ with $\mathcal{F}_{\hsc}(f)\in H^s(\Rea^d)$,
$$
\cR_k(1-\psi(|\hsc D|))f=\lim_{\eps\to 0^+} \mathcal{F}_{\hsc}^{-1}\Big(\frac{\hsc^2\mathcal{F}_{\hsc}(f)(\xi)(1-\psi(|\xi|)}{|\xi|^2-(1+\ri\eps)^2}\Big)=\mathcal{F}_{\hsc}^{-1}\Big(\frac{\hsc^2\mathcal{F}_{\hsc}(f)(\xi)(1-\psi(|\xi|)}{|\xi|^2-1}\Big).
$$
A nearly identical argument implies that 
$$
(1-\psi(|\hsc D|))\cR_k=\cR_k(1-\psi(|\hsc D|))
$$
and the fact that $\cR_k(1-\psi(|\hsc D|))\in \hsc^2\Psi_\hsc^{-2}(\Rea^d)$ follows from the definition of $\Psi_\hsc^{-2}(\Rea^d)$. 
\end{proof}

The other choices of $A$ required to prove Lemma \ref{l:SDHHigh} (via Lemma \ref{l:computeLayer}) are covered by the following lemma.

\ble\label{l:resolveHighFreq2}
If $b\in S^m$ with $\supp (b)\cap \{|\xi|\leq 1\}=\emptyset$, $\Op_\hsc(b) \cR_k\in \hsc^2\Psi^{m-2}_\hsc(\Rea^d),$ $\cR_k\Op_\hsc(b)\in \hsc^2\Psi^{m-2}_\hsc(\Rea^d)$, and there is $\widetilde{\psi}\in C_c^\infty(\mathbb{R})$ with $\widetilde{\psi}\equiv 1$ in a neighbourhood of $[-1,1]$ such that 
\beq\label{eq:composition_pseudo_symbol1}
\Op_\hsc(b)\cR_k=\hsc^2\Op_\hsc(b)\Op_\hsc\Big(\frac{1-\widetilde{\psi}(|\xi|)}{|\xi|^2-1}\Big)+O(\hsc^\infty)_{\Psi^{-\infty}_\hsc},
\eeq
and
\beq\label{eq:composition_pseudo_symbol2}
\cR_k\Op_\hsc(b)=\hsc^2\Op_\hsc\Big(\frac{1-\widetilde{\psi}(|\xi|)}{|\xi|^2-1}\Big)\Op_\hsc(b)+O(\hsc^\infty)_{\Psi^{-\infty}_\hsc}.
\eeq
\ele

\bpf
Let $\widetilde{\psi}\in C_c^\infty(\mathbb{R})$ with $\widetilde{\psi} \equiv 1$ on a neighbourhood of $[-1,1]$ and such that 
$$
\big\{(x,\xi) \in \supp b\,:\, |\xi|\in \supp \widetilde{\psi}\big\}=\emptyset.
$$
Then, by Lemma~\ref{l:disjoint}, 
$$
\Op_\hsc(b)\cR_k=\Op_\hsc(b)(1-\widetilde{\psi}(|\hsc D|))\cR_k+O(\hsc^\infty)_{\Psi^{-\infty}_\hsc},
$$
and 
$$
 \cR_k \Op_\hsc(b)=\cR_k(1-\widetilde{\psi}(|\hsc D|))\Op_\hsc(b)+O(\hsc^\infty)_{\Psi^{-\infty}_\hsc}.
$$
By Lemma~\ref{l:resolveHighFreq}, $(1-\widetilde{\psi}(|\hsc D|))\cR_k,\,\cR_k(1-\widetilde{\psi}(|\hsc D|))\in \hsc^2\Psi^{-2}_\hsc$ and both are given by
$$
\hsc^2\Op_\hsc\Big(\frac{1-\widetilde{\psi}(|\xi|)}{|\xi|^2-1}\Big),
$$
which completes the proof.
\epf

\begin{proof}[Proof of Lemma \ref{l:SDHHigh}]
We apply Lemma \ref{l:computeLayer} and use the results of Lemmas \ref{l:resolveHighFreq} and \ref{l:resolveHighFreq2}.
For $S_k$, we let $A = \hsc^{-2} (1- \psi(|\hsc D|)) \cR_k=\hsc^{-2} \cR_k (1- \psi(|\hsc D|))$, which is in $\Psi^{-2}_\hsc(\Rea^d)$ by Lemma \ref{l:resolveHighFreq}, so that $\widetilde{S}^{L/R}_k =\widetilde{S}_k= \hsc^2 \gamma^\pm A\gamma^*$ by definition. 
Since $A = \Op_\hsc(a)$ with $a= (1-\psi(|\xi|))(|\xi|^2-1)^{-1}\in S^{-2}(T^*\Rea^d)$, Lemma \ref{l:computeLayer} applies with $m=-2$ and $a_j=0$.
Therefore $\widetilde{S}_k \in \hsc^2 \hsc^{-1}\Psi^{-1}_\hsc(\Gamma)= \hsc\Psi^{-1}_\hsc(\Gamma)$ and 
\begin{align*}
\sigma_\hsc(\widetilde{S}_k) 
&=\hsc^2 \lim_{x_1\to 0^\pm} \frac{1}{2\pi \hsc}\int_{-\infty}^\infty \frac{ \big(1-\psi\big(\sqrt{\xi_1^2+|\xi'|_g^2}\big)\big)}{\xi_1^2 + |\xi'|_g^2 - 1}
\re^{\ri \xi_1 x_1/\hsc}
\rd \xi_1.
\end{align*}
When $|\xi'|_g\geq 2$, the integrand has poles at $\pm \ri \sqrt{|\xi'|_g^2 -1}$ and evaluating the integral via the residues at these poles gives the first equation in \eqref{eq:HFsym1}.

With $n$ any extension of the normal vector field to $\Gamma$ to all of $\mathbb{R}^d$,
\beq\label{eq:symbolL}
\hsc L=\Op_{\hsc}(\ri\langle \xi,n\rangle),\qquad \hsc L^*=\Op_{\hsc}(-\ri\langle \xi,n\rangle-\hsc \operatorname{div}(n)).
\eeq
and thus $\hsc L$ and $\hsc L^* \in \Psi^1(\Rea^d)$.

For $(\DL'_k)^{\pm, R}$ we let $A= \hsc^{-1} L \cR_k(1-\psi(|\hsc D|)$, which is in $\Psi_\hsc^{-1}(\Rea^d)$ by Lemma \ref{l:resolveHighFreq} and Part (i) of Theorem \ref{thm:basicP}. For $(\DL'_k)^{\pm, L}$ we let $A= \hsc^{-1} (1-\psi(|\hsc D|)L \cR_k$; we now claim that this is in  $\Psi_\hsc^{-1}(\Rea^d)$.
Indeed, \eqref{eq:symbolL} and the composition formula for symbols \cite[Theorem 4.14]{Zw:12}, \cite[Proposition E.8]{DyZw:19} imply that
\beqs
\big(1- \psi(|\hsc D|)\big)\hsc L = \Op_\hsc(b) +O(\hsc^\infty)_{\Psi^{-\infty}_\hsc},
\eeqs 
where $b$ satisfies the conditions of Lemma \ref{l:resolveHighFreq2} with $m=1$; this lemma therefore implies that $(1-\psi(|\hsc D|)\hsc L \cR_k\in \hsc^2 \Psi^{-1}_\hsc(\Rea^d)$, and thus $A= \hsc^{-1} (1-\psi(|\hsc D|)L \cR_k\in\Psi_\hsc^{-1}(\Rea^d)$.

We now claim that, for both $(\DL'_k)^{\pm, R}$ and $(\DL'_k)^{\pm, L}$, Lemma \ref{l:computeLayer} holds with $m=-1$. 
Indeed, by \eqref{eq:composition_pseudo_symbol1} and \eqref{eq:composition_pseudo_symbol2}, in both cases
 $A=\Op_\hsc(a)$ with
\beqs
a(x,\xi):= \frac{\ri \langle\xi,n\rangle \big(1-\psi(|\xi|)\big)}{|\xi|^2-1} +\hbar r(x,\xi),
\eeqs
with $r\in S^{-2}$. In particular, for $|\langle\xi,n\rangle|\geq 2$, 
\beqs
a(x,\xi)= \frac{\ri \langle\xi,n\rangle}{|\xi|^2}\left( 1 + \frac{1}{|\xi|^2}+ \frac{1}{|\xi|^4} + \cdots\right)+\hsc r(x,\xi);
\eeqs
therefore, \eqref{e:classical} holds with $a_{-1}(x,\xi) = \ri \langle\xi,n\rangle/|\xi|^2$ -- observe that this is homogeneous of degree $-1$
and satisfies \eqref{e:classical2a} with $j=-1$.
Lemma \ref{l:computeLayer} with $m=-1$ then implies that $(\widetilde{\DL}'_k)^{\pm, L/R} \in \Psi^0_\hsc(\Gamma)$ with
\beqs
\sigma_\hsc\big((\widetilde{\DL}'_k)^{\pm, L/R}\big) =\hsc\, \sigma_\hsc(\gamma^\pm A \gamma^*) = \hsc \lim_{x_1\to 0^\pm} \frac{1}{2\pi \hsc}\int_{-\infty}^\infty \frac{ \ri \xi_1 \psi\big(\sqrt{\xi_1^2+|\xi'|_g^2}\big)}{\xi_1^2 + |\xi'|_g^2 - 1}
\re^{\ri \xi_1 x_1/\hsc}
\rd \xi_1;
\eeqs
evaluating the integral via residues gives the second equation in \eqref{eq:HFsym1}.

The proofs for $(\DL_k)^{\pm, L/R}$ are very similar to those for  $(\DL'_k)^{\pm, L/R}$; indeed, for 
 $(\DL_k)^{\pm, L}$ we let $A= \hsc^{-1} (1-\psi(|\hsc D|) \cR_k L^*$, which is in $\Psi_\hsc^{-1}(\Rea^d)$ by Lemma \ref{l:resolveHighFreq}, and 
for $(\DL_k)^{\pm, R}$, we let $A= \hsc^{-1} \cR_k L^*(1-\psi(|\hsc D|)$, which is in $\Psi_\hsc^{-1}(\Rea^d)$ using similar arguments to those used above for $\hsc^{-1} (1-\psi(|\hsc D|)L \cR_k$. The first equation in \eqref{eq:HFsym2} follows in a similar way to above,
since the symbol of $A$ for $(\DL_k)^{\pm, L/R}$ is now minus the symbol of $A$ for $(\DL'_k)^{\pm, L/R}$.

For $(\widetilde{H}_k)^{\pm,L}$ we let $A= (1-\psi(|\hsc D|)) L\cR_kL^*$ and for $(\widetilde{H}_k)^{\pm,R}$ we let $A= L\cR_kL^*(1-\psi(|\hsc D|))\in \Psi_\hsc^{0}(\Rea^d)$; note that in both cases $A\in \Psi_\hsc^{0}(\Rea^d)$ by the arguments above (using Lemma \ref{l:resolveHighFreq2}) and Part (i) of Theorem \ref{thm:basicP}.
Furthermore, in both cases, by the composition formula for symbols \cite[Theorem 4.14]{Zw:12}, \eqref{eq:symbolL},
\eqref{eq:composition_pseudo_symbol1}, and \eqref{eq:composition_pseudo_symbol2},
$A= \Op_\hsc(a)$ with
\beqs
a(x,\xi)= \frac{\langle\xi,n\rangle^2 \big(1- \psi(|\xi|)\big)}{|\xi|^2-1} -\hsc\ri\frac{  \operatorname{div}(n)\langle \xi,n\rangle \big(1-\psi(|\xi|)\big)}{|\xi|^2-1} -\hsc\ri \Big\langle \partial_\xi \frac{\langle \xi,n\rangle \big(1-\psi(|\xi|)\big)}{|\xi|^2-1},\partial_x\langle \xi,n\rangle\Big\rangle +\hsc r(x,\xi),
\eeqs
where $r\in S^{-2}$.
Therefore \eqref{e:classical} holds with  
$$
a_{-1}(x,\xi)= -\hsc \ri \Big(\frac{\operatorname{div} (n)\langle \xi,n\rangle}{|\xi|^2}+\Big\langle \partial_\xi \frac{\langle \xi,n\rangle}{|\xi|^2},\partial_x\langle \xi,n\rangle\Big\rangle\Big)
$$
and $a_{0}(x,\xi) = \langle\xi,n\rangle^2/|\xi|^2$; observe that $a_{-1}$ is homogeneous of degree $-1$, $a_0$ is homogeneous of degree $0$, and both satisfy \eqref{e:classical2a}.
\end{proof}

\begin{theorem}[The high-frequency components of the operators $S_k$, $\DL_k$, $\DL'_k$, and $H_k$.]\label{thm:HFSD}
Let $\chi \in C^\infty_{\rm comp}(\Rea)$ with $\chi(\xi)= 1$ for $|\xi|\leq 2$ and $\chi(\xi)=0$ for $|\xi|\geq 3$.
Then
\beqs
\begin{gathered}
\big(1-\chi\big(|\hsc D'|_g\big)\big) S_k ,\,\, \,S_k\big(1-\chi\big(|\hsc D'|_g\big)\big)\in \hsc \Psi_\hsc^{-1}(\Gamma),\\
\big(1-\chi\big(|\hsc D'|_g\big)\big) \DL_k' ,\,\, \, \DL_k'\big(1-\chi\big(|\hsc D'|_g\big)\big)\in \hsc \Psi_\hsc^{-1}(\Gamma),\\
\big(1-\chi\big(|\hsc D'|_g\big)\big) \DL_k ,\,\, \, \DL_k\big(1-\chi\big(|\hsc D'|_g\big)\big)\in \hsc \Psi_\hsc^{-1}(\Gamma),\\
\big(1-\chi\big(|\hsc D'|_g\big)\big) H_k ,\,\, \, H_k\big(1-\chi\big(|\hsc D'|_g\big)\big)\in \hsc^{-1} \Psi_\hsc^{1}(\Gamma).
\end{gathered}
\eeqs
Moreover,
\beqs
\sigma_\hsc\big(\big(1-\chi\big(|\hsc D'|_g\big)\big) S_k\big)=\sigma_\hsc\big(S_k\big(1-\chi\big(|\hsc D'|_g\big)\big)= \dfrac{\hsc\big(1-\chi(|\xi'|_g)\big)}{2\sqrt{|\xi'|_g^2-1}}
\eeqs
and
\beq\label{eq:H_prin_symb}
\sigma_\hsc\big(\big(1-\chi\big(|\hsc D'|_g\big)\big) H_k\big)=\sigma_\hsc\big(H_k\big(1-\chi\big(|\hsc D'|_g\big)\big)=-\big(1-\chi(|\xi'|_g)\big) \dfrac{\sqrt{|\xi'|_g^2-1}}{2\hsc}.
\eeq
\end{theorem}
\begin{proof}
We first claim that for $\chi$ as in the statement of the theorem and $\psi\in C_{\rm comp}^\infty((-2,2))$,
\begin{equation}
\label{e:leftSide1}
T:=\psi(|\hsc D|)\gamma^*(1-\chi (|\hsc D'|_g))=O(\hsc^\infty)_{H_\hsc^{-N}(\Gamma)\to H_\hsc^N(\Rea^d)}.
\end{equation}
Indeed, in the local coordinates described in \S\ref{sec:local}, the kernel of $T$ is given by
$$
T(x,y')=\frac{1}{(2\pi h)^{d}}\int_{\Rea^d} \re^{\frac{\ri}{\hsc}(\langle x'-y',\xi'\rangle+ x_1\xi_1)}a(y',\xi')b(x,\xi)\rd \xi
$$
where $\supp a\cap \supp b=\emptyset$. Hence, the kernel of $T^*T$ is given by
\begin{align*}
&(T^*T)(x',y')
=\int_{\Rea^d} \overline{T(z,x')} T(z,y') \rd z,
\\
&=\frac{1}{(2\pi \hsc)^{2d}}\int_{\Rea^d}\int_{\Rea^d}\int_{\Rea^d} \re^{\frac{\ri}{\hsc}(\langle x'-z',\xi'\rangle -z_1\xi_1+\langle z'-y',\eta'\rangle+ z_1\eta_1)} \bar{a}(x',\xi')\bar{b}(z,\xi)a(y',\eta')b(z,\eta)\rd \eta \rd z \rd \xi\\
&=\frac{1}{(2\pi \hsc)^{2d}}\int_{\Rea^d}\int_{\Rea^d}\int_{\Rea^d} \re^{\frac{\ri}{\hsc}(\langle x',\xi'\rangle +\langle z',\eta'-\xi'\rangle+ z_1(\xi_1-\eta_1)-\langle y',\eta'\rangle)} \bar{a}(x',\xi')\bar{b}(z,\xi)a(y',\eta')b(z,\eta)\rd \eta \rd z \rd \xi.
\end{align*}
Now, if $|\xi-\eta|>c>0$, then we can integrate by parts in $z$ to gain powers of $\hsc| \xi-\eta|^{-1}$ and hence obtain $O(\hsc^\infty)_{\Psi^{-\infty}_\hsc(\Gamma)}$. Similarly if $|z_1|>c$, $|z'-x'|$, or $|z'-y'|>c$, we can integrate by parts in respectively $\eta_1$, $\xi'$, or $\eta'$ to gain powers of $h|z_1|^{-1}$, $h|z'-x'|^{-1}$, or $h|z'-y'|^{-1}$. Since the integrand is compactly supported in $(x',y', \xi ,\eta)$ and when $\xi'=\eta'$, $x'=z'=y'$, and $z_1=0$, the integrand is 0, this implies \eqref{e:leftSide1}.
Taking adjoints of \eqref{e:leftSide1} implies also that 
\begin{equation}
\label{e:leftSide}
(1-\chi (|\hsc D'|_g))\gamma\psi(|\hsc D|)=O(\hsc^\infty)_{H_\hsc^{-N}(\Rea^d)\to H_\hsc^N(\Gamma)}.
\end{equation}

By \eqref{eq:SHdef} and \eqref{e:leftSide}, 
\begin{align*}
(1-\chi(|\hsc D'|_g))S_k=(1-\chi(|\hsc D'|_g))\gamma \cR_k\gamma^*
=(1-\chi(|\hsc D'|_g))\gamma (1-\psi(|\hsc D|)) \cR_k\gamma^*+O(\hsc^\infty)_{\Psi_\hsc^{-\infty}(\Gamma)},
\end{align*}
and the claim that $(1-\chi(|\hsc D'|_g))S_k \in \hsc \Psi_\hsc^{-1}(\Gamma)$ then follows from Lemma~\ref{l:SDHHigh}. 
Similarly, by \eqref{eq:KK'def}, and \eqref{e:leftSide},
\begin{align*}
(1-\chi(|\hsc D'|_g))\DL_k&=(1-\chi(|\hsc D'|_g))(\gamma^+\cR_kL^*\gamma^*-\tfrac{1}{2}I)\\
&=(1-\chi(|\hsc D'|_g))(\gamma^+(1-\psi(|\hsc D|))\cR_kL^*\gamma^*-\tfrac{1}{2}I)+O(\hsc^\infty)_{\Psi^{-\infty}_\hsc(\Gamma)},
\end{align*}
and the claim that $(1-\chi(|\hsc D'|_g))\DL_k \in  \hsc\Psi_\hsc^{-1}(\Gamma)$  follows from Lemma~\ref{l:SDHHigh}. The arguments for $(1-\chi(|\hsc D'|_g))\DL_k'$ and $(1-\chi(|\hsc D'|_g))H_k$ are similar.
To prove the results with cutoffs on the right of $S_k,\DL_k, \DL_k'$, and $H_k$, we argue similarly but with \eqref{e:leftSide} replaced by~\eqref{e:leftSide1}.
\end{proof}

We record the following corollary of Theorem \ref{thm:HFSD}, for specific use in the proof that $\Breg$ and $\Bregp$ are Fredholm (in Part (i) of Theorem \ref{thm:invert}).

\begin{corollary}\label{cor:hBEM}
Suppose that $\Reg$ satisfies Assumption \ref{ass:Reg}. If $\chi \in C^\infty_{\rm comp}(\Rea)$ with $\chi(\xi)= 1$ for $|\xi|\leq 2$ and $\chi(\xi)=0$ for $|\xi|\geq 3$, then 
\beqs
\big(1-\chi\big(|\hsc D'|_g\big)\big) \Reg H_k ,\,\, \, \big(1-\chi\big(|\hsc D'|_g\big)\big)H_k \Reg \in \Psi_{\hsc}^{0}(\Gamma),
\eeqs
and the semiclassical principal symbols of both these operators are real.
\end{corollary}

\bpf
The fact that $(1-\chi(|\hsc D'|_g))H_k \Reg\in \Psi_{\hsc}^{0}(\Gamma)$ follows immediately from Part (i) of Theorem \ref{thm:basicP}, Theorem \ref{thm:HFSD}, and Assumption \ref{ass:Reg}, and that its semiclassical principal symbol is real follows from \eqref{eq:H_prin_symb}, Assumption \ref{ass:Reg}, and the first equation in \eqref{eq:multsymb}.

To prove the result about $(1-\chi(|\hsc D'|_g))\Reg H_k$, let $\widetilde{\chi}\in C_{\rm comp}^\infty(\{\chi\equiv 1\})$ and $\widetilde{\chi}\equiv 1$ in a neighbourhood of $[-1,1]$. Then, by the composition formula for symbols \cite[Theorem 4.14]{Zw:12}, \cite[Proposition E.8]{DyZw:19} and Lemma~\ref{l:disjoint},
$$
(1-\chi\big(|\hsc D'|_g\big)\big) \Reg \widetilde{\chi}(|hD'|)=O(\hsc^\infty)_{\Psi_\hsc^{-\infty}(\Rea^d)}.
$$
Therefore 
$$
(1-\chi\big(|\hsc D'|_g\big)\big) \Reg H_k 
=(1-\chi\big(|\hsc D'|_g\big)\big) \Reg
\big(1- \widetilde{\chi}(|hD'|)\big)H_k+O(\hsc^\infty)_{\Psi_\hsc^{-\infty}(\Rea^d)};
$$
the result for $(1-\chi(|\hsc D'|_g))\Reg H_k$ then follows using the same arguments used above for $(1-\chi(|\hsc D'|_g))H_k \Reg$.
\epf

\begin{remark} Using an elliptic parametrix construction (see e.g.~\cite[Proposition E.32]{DyZw:19}) it is easy to check that every instance of $(1-\chi(|k^{-1}D'|_g))$ in Theorem~\ref{thm:HFSD} can be replaced by any semiclassical pseudodifferential operator whose wavefront set does not intersect $\{|\xi'|_g\leq 1\}$ (see~\cite[Definition E.27]{DyZw:19} for a definition of the wavefront set of a semiclassical pseudodifferential operator).

Indeed, if $A\in \Psi^\ell(\Gamma)$ has wavefront set contained in $\{|\xi'|_g> 1\}$, then an elliptic parametrix shows that there exists $Q\in \Psi_\hsc^\ell(\Gamma)$ such that
\beqs
A = Q\big(1-\chi(|k^{-1}D'|_g)\big) + O(\hsc^\infty)_{\Psi^{-\infty}_\hsc(\Gamma)}.
\eeqs
Acting on the right with the boundary-integral operators, we obtain Theorem \ref{thm:HFSD} with $(1-\chi(|k^{-1}D'|_g))$ replaced by $A$ (where we have used that the composition of an $O(\hsc^\infty)_{\Psi^{-\infty}_\hsc(\Gamma)}$ operator and an operator with norm bounded by $\hsc^{-N}$ for some $N>0$ is a $O(\hsc^\infty)_{\Psi^{-\infty}_\hsc(\Gamma)}$ pseudodifferential operator).

In particular, one can show that if $\psi \in C_{\rm comp}^\infty(\Rea)$ with $\psi \equiv 1$ in a neighbourhood of $[-1,1]$, 
then the wavefront set of $(1-\psi(-k^{-2}\Delta_\Gamma))$ is contained in $\{|\xi'|_g> 1\}$,
where $\Delta_\Gamma$ is the surface Laplacian, a.k.a.~the Laplace--Beltrami operator.
Therefore, one can replace $(1-\chi(|k^{-1}D'|_g))$ by $(1-\psi(-k^{-2}\Delta_\Gamma))$; this fact is used in \cite{GaSp:22} to present the results of Theorem \ref{thm:HFSD} without explicitly using pseudodifferential operators.
\end{remark}

\subsection{Bounds on $H_k$}

The following result improves the $k$-dependence of the bounds in \cite[Theorems 4.5 and 4.37]{Ga:19}. 

\begin{theorem}[Bounds on $H_k$]\label{thm:Hk}
If $\Gamma$ is Lipschitz and piecewise smooth, then given $k_0>0$ there exists $C_1>0$ such that
\beq\label{eq:H_kbound1}
\N{H_k}_{H^{t}_k(\Gamma)\rightarrow H^{t-1}_k(\Gamma)} \leq C_1 k \log (k+2) \quad\tfa k\geq k_0 \tand t \in [0,1].
\eeq
If $\Oi$ is convex and $\Gamma$ is $C^\infty$ and curved, then given $k_0>0$ there exists $C_2>0$ such that 
\beq\label{eq:H_kbound2}
\N{H_k}_
{H^{t}_k(\Gamma)\rightarrow H^{t-1}_k(\Gamma)}
\leq C_2 k \quad\tfa k\geq k_0 \tand t\in \Rea.
\eeq
\end{theorem}

In the proof, and in the rest of the paper, $\langle\cdot,\cdot\rangle_{\Gamma, \Rea}$ denotes the \emph{real}-valued duality pairing between $H^s(\Gamma)$ and $H^{-s}(\Gamma)$, so that $\langle \phi,\psi\rangle_{\Gamma, \Rea}$ is the real-valued $L^2(\Gamma)$ inner product when $\phi,\psi\in\LtG$. 

\bpf
We first show that, for all $k>0$,
\beq\label{eq:Hdual}
\big\langle H_k \phi, \psi\big\rangle_{\Gamma,\Rea} = \big\langle\phi,H_k \psi \big\rangle_{\Gamma,\Rea} \quad \tfa \phi \in \LtG\tand\psi \in \HoG.
\eeq
By the density of $H^{1/2}(\Gamma)$ in $\LtG$ and the fact that $H_k$ is bounded $H^1(\Gamma)\rightarrow \LtG$ and $\LtG\rightarrow H^{-1}(\Gamma)$ by \eqref{eq:mapping}, we only need to show that \eqref{eq:Hdual} holds for all $\phi,\psi\in H^{1/2}(\Gamma)$.
Given $\phi,\psi\in H^{1/2}(\Gamma)$, let $u= \cK_k \phi$ and $v=\cK_k\psi$. 
The relation  \eqref{eq:Hdual} then holds by applying Green's second identity to $u$ and $v$ in both $\Oi$ and in $\Oe\cap B_R$ with $R>\diam(\Oi)$, subtracting the two resulting equations, using the third and fourth jump relations in \eqref{eq:jumprelations}, letting $R\tendi$, and using that
\beq\label{eq:NtDdual}
\big\langle \partial_n^+ u , \gamma^+ v\rangle_{\Gamma,\Rea} = \big\langle
 \partial_n^+ v , \gamma^+ u \rangle_{\Gamma,\Rea}, 
\eeq
which holds since both $u$ and $v$ satisfy the Sommerfeld radiation condition; note that here it is important that $\langle\cdot,\cdot\rangle_{\Gamma,\Rea}$ is the \emph{real}-valued duality pairing -- see \cite[Lemma 6.13]{Sp:15}.

By \eqref{eq:Hdual}, 
\beqs
\N{H_k}_{\LtG\rightarrow H^{-1}_k(\Gamma)} =
\N{H_k}_{H^1_k(\Gamma)\rightarrow\LtG}.
\eeqs
Using this and interpolation (see,  e.g., \cite[\S 2.3]{ChSpGiSm:20}, \cite[\S 4]{ChHeMo:15}), 
it is therefore sufficient to prove \eqref{eq:H_kbound1} for $t=1$.
Lemma~\ref{l:SDHHigh} and Part (ii) of Theorem \ref{thm:basicP} implies that 
\begin{equation}
\label{e:estTHk}
\big\|(\widetilde{H}_k)^\pm\big\|_{H_k^s(\Gamma)\to H_k^{s-1}(\Gamma)}\leq Ck.
\end{equation}
The bound
$$
\|\gamma^{\pm}L\cR_k\psi(|\hsc D|)L^*\gamma^*\|_{L^2(\Gamma)\to L^2(\Gamma)}\leq C \langle k\rangle \log (k+2)
$$
follows from~\cite[Lemmas 4.6, 4.11]{Ga:19}; combining this with~\eqref{e:estTHk} implies \eqref{eq:H_kbound1} with $t=1$.
 
The bound \eqref{eq:H_kbound2} when $\Oi$ is convex and $\Gamma$ is $C^\infty$ and curved follows from~\cite[Theorem 4.37]{Ga:19} (or, more precisely~\cite[Lemmas 4.27 and 4.36]{Ga:19}).
\epf

 \subsection{Bounds on \(S_{\ri k}\)}
 
\ble\label{lem:Sik}
$S_{\ri k}$ satisfies Assumption \ref{ass:Reg}. 
\ele

\bpf
We claim that $S_{\ri k}= k^{-1} \widetilde{S}$ where $\widetilde{S}\in \Psi_\hsc^{-1}(\Gamma)$, with 
\beq\label{eq:psSik}
\sigma_\hsc(\widetilde{S}) = \frac{1}{2\sqrt{|\xi'|^2+1}}.
\eeq
By the first equation in \eqref{eq:SHdef}, $S_{\ri k} = \gamma^\pm \cR_{\ri k} \gamma^*$, and by \eqref{eq:resolvent} $\cR_{\ri k}$ 
is the Fourier multiplier with (semiclassical) symbol $\hsc^2(|\xi|^2+1)^{-1}$.
Let $A= \hsc^2 \cR_{\ri k} \in \Psi^{-2}(\Rea^d)$, so that $\widetilde{S} = \hsc^{-1}S_{\ri k}= \hsc \gamma^\pm A\gamma^*$, 
Since $A = \Op_\hsc(a)$ with $a= (|\xi|^2+1)^{-1}\in S^{-2}$, Lemma \ref{l:computeLayer} applies with $m=-2$ and $a_j=0$, and implies that $\widetilde{S}\in\hsc \hsc^{-1} \Psi^{-1}_\hsc(\Gamma)= \Psi^{-1}_\hsc(\Gamma)$ with 
\beqs
\sigma_\hsc(\widetilde{S}) =\hsc \lim_{x_1\to 0^\pm} \frac{1}{2\pi \hsc}\int_{-\infty}^\infty \frac{ \exp( \ri \xi_1 x_1/\hsc)}{\xi_1^2 + |\xi'|_g^2 + 1}\rd \xi_1;
\eeqs
calculating the integral via residues gives \eqref{eq:psSik}.
\epf

The bounds in Corollary \ref{cor:Rbound} therefore hold with $R= S_{\ri k}$.
We now show that, modulo an additional $(\log (k+2))^{1/2}$ factor, the upper bound in \eqref{eq:Rbound1} holds when $R=S_{\ri k}$ and $\Gamma$ is piecewise $C^\infty$ (as opposed to $C^\infty$ in Corollary \ref{cor:Rbound}).

\begin{theorem}\label{thm:newSikbound}
Let $\Oi$ be Lipschitz with $\Gamma$ piecewise smooth (in the sense of Definition \ref{def:piecewisesmooth}).
Then given $k_0>0$ there is $C>0$ such that for all $k\geq k_0$
\begin{equation}
\label{e:singleBound}
\|S_{\ri k}\|_{L^2(\Gamma)\to H_k^1(\Gamma)}+\|S_{\ri k}\|_{H_k^{-1}(\Gamma)\to L^2(\Gamma)}\leq Ck^{-1}(\log (k+2))^{1/2}.
\end{equation}
\end{theorem}
\begin{proof}
We prove the first estimate in~\eqref{e:singleBound}; the second estimate follows since $S_{\ri k}$ is self-adjoint on $L^2(\Gamma)$ (see, e.g., \cite[Equation 2.37]{ChGrLaSp:12}).
Recall from the first equation in \eqref{eq:SHdef} that $S_{\ri k} = \gamma (-\Delta+k^2)^{-1}\gamma^*$ and recall
that, 
since $k^2(-\Delta+k^2)^{-1} \in \Psi_\hsc^{-2}(\Rea^d)$,
for $k\geq k_0$,
\begin{equation}
\label{e:basicElliptic}
\|(-\Delta+k^2)^{-1}\|_{H_{k}^s(\Rea^d)\to H_k^{s+2}(\Rea^d)}\leq Ck^{-2}.
\end{equation}
Therefore, fixing $0<\e<1/2$, and using 
\eqref{e:basicElliptic}, the trace bounds \eqref{e:basicTrace}, \eqref{e:basicTrace}, and the inequalities \eqref{eq:weightednorm}, we have   
\begin{align}\nonumber
\|S_{\ri k}\|_{L^2(\Gamma)\to L^2(\Gamma)}&\leq \|\gamma^*\|_{H_k^{-\e}(\Gamma)\to H_k^{-\frac{1}{2}-\e}(\Rea^d)}\|(-\Delta+k^2)^{-1}\|_{H_k^{-\frac{1}{2}-\e}(\Rea^d)\to H_k^{\frac{3}{2}-\e}(\Rea^d)}\|\gamma\|_{H_k^{\frac{3}{2}-\e}(\Rea^d)\to L^2(\Gamma)}
\\
&\leq Ck^{-1}.
\label{e:L2Bound}
\end{align}
By \eqref{eq:1knorm},
\beq\label{eq:Sik1knorm}
\|S_{\ri k}\|_{L^2(\Gamma)\to H_k^1(\Gamma)}\leq C\big(\|S_{\ri k}\|_{L^2\to L^2}+k^{-1}\|\nabla_{\Gamma}S_{\ri k}\|_{L^2(\Gamma)\to L^2(\Gamma)}\big),
\eeq
therefore, we only need to estimate $k^{-1}\nabla_{\Gamma}S_{\ri k}:L^2(\Gamma)\to L^2(\Gamma)$. For this, we let $\psi\in C_{\rm comp}^\infty(\mathbb{R})$ with $\psi \equiv 1$ on $[-1,1]$, and decompose $S_{\ri k}$ as follows: 
\begin{equation}
\label{e:decompose}
\begin{aligned}
S_{\ri k}&=S_{\ri} + (S_{\ri k}-S_\ri),\\
&= S_{\ri} + \gamma \Big[(-\Delta+k^2)^{-1}-(-\Delta+1)^{-1}\Big]\gamma^*,\\
&=S_{\ri}+\gamma\psi(k^{-1}|D|)(-\Delta+k^2)^{-1}\gamma^*\\
&\quad+\gamma(1-\psi(k^{-1}|D|))\Big[(-\Delta+k^2)^{-1}-(-\Delta+1)^{-1}\Big]\gamma^*-\gamma \psi(k^{-1}|D|)(-\Delta+1)^{-1}\gamma^*\\
&=: {\rm I+II+III+IV}
\end{aligned}
\end{equation}
We now estimate each term ${\rm I}$ through ${\rm IV}$ individually. 
First, for ${\rm I}$, we recall from \eqref{eq:mapping} that 
\begin{equation}
\label{e:singleI}
\|S_\ri\|_{L^2(\Gamma)\to H^1(\Gamma)}\leq C.
\end{equation}
To estimate ${\rm II}$, recall that, 
since $\psi$ has compact support, $\psi(k^{-1}|D|):H_k^{s}(\Gamma)\to H_k^{s+N}(\Gamma)$ for any $N$ is bounded (uniformly in $k$) (cf.~\eqref{eq:frequencycufoff}). 
Therefore,  using \eqref{eq:weightednorm}, the trace bounds \eqref{e:basicTrace} and \eqref{e:basicTrace2}, and \eqref{e:basicElliptic}, we have
\begin{equation}
\label{e:lowFreq}
\|{\rm II}\|_{L^2(\Gamma)\to H^1_k(\Gamma)}=\|\gamma \psi(k^{-1}|D|)(-\Delta+k^2)^{-1}\gamma^*\|_{L^2(\Gamma)\to H_k^1(\Gamma)}\leq Ck^{-1}.
\end{equation}
We now claim that 
\begin{equation}\label{eq:III}
\Big\|\big(1-\psi(k^{-1}|D|)\big)\big[(-\Delta+k^2)^{-1}-(-\Delta+1)^{-1}\big]\Big\|_{H_{k}^s(\Rea^d)\to H_k^{s+4}(\Rea^d)}\leq k^{-2}.
\end{equation}
and thus, using~\eqref{e:basicTrace}, \eqref{e:basicTrace2}, and \eqref{eq:weightednorm} again,
\begin{equation}
\label{e:highFreqDiff}
\|{\rm III}\|_{L^2(\Gamma)\to H^1_k(\Gamma)}=\big\|\gamma\big(1-\psi(k^{-1}|D|)\big)\big[(-\Delta+k^2)^{-1}-(-\Delta+1)^{-1}\big]\gamma^*\big\|_{L^2(\Gamma)\to H_k^{1}(\Gamma)}\leq Ck^{-1}.
\end{equation}
To prove \eqref{eq:III}, 
observe that $(1-\psi(k^{-1}|D|))[(-\Delta+k^2)^{-1}-(-\Delta+1)^{-1}]$ is a Fourier multiplier with (semiclassical) symbol
\beqs
 \big(1- \psi(\xi)\big)\hsc^2 \left[ \frac{1}{|\xi|^2+1} - \frac{1}{|\xi|^2 + \hsc^2}\right] = 
\frac{\hsc^2 (-1+\hsc^2)  \big(1- \psi(\xi)\big)}{
(|\xi|^2+1)(|\xi|^2 + \hsc^2)};
\eeqs
therefore 
\beqs
\Big\|\big(1-\psi(k^{-1}|D|)\big)\big[(-\Delta+k^2)^{-1}-(-\Delta+1)^{-1}\big]\Big\|_{H_{k}^s(\Rea^d)\to H_k^{s+4}(\Rea^d)}\leq 
\sup_{\xi\in \Rea^d} \left|
\langle\xi\rangle^4
\frac{\hsc^2 (-1+\hsc^2)  \big(1- \psi(\xi)\big)}{
(|\xi|^2+1)(|\xi|^2 + \hsc^2)}
\right|,
\eeqs
and  \eqref{eq:III} follows since $1- \psi(\xi)=0$ for $|\xi|\leq 1$.

To estimate ${\rm IV}$, we claim that
$$
\|\psi (k^{-1}|D|)\gamma^*\|_{L^2(\Gamma)\to H^{-1/2}(\Rea^d)}\leq C(\log (k+2))^{1/2}; 
$$
indeed,
\begin{align*}
\int_{\Rea^d} \langle \xi\rangle^{-1} \big|\cF_\hsc u (\xi')\big|^2 (\psi( k^{-1} |\xi|))^2 \rd \xi 
&\leq \int_{\Rea^d} \langle \xi_1\rangle^{-1} \big|\cF_\hsc u (\xi')\big|^2 (\psi( k^{-1} |\xi|))^2 \rd \xi_1\rd \xi'  \\
&\lesssim \log(k+2) \int_{\Rea^{d-1}} \big|\cF_\hsc u (\xi')\big|^2\rd \xi'.
\end{align*}
Then, since $(-\Delta+1)^{-1}:H^{-1/2}(\Rea^d)\to H^{3/2}(\Rea^d)$ and $\gamma:H^{3/2}(\Rea^d)\to H^1(\Gamma)$ are bounded,
\begin{equation}
\label{e:lowFreqI}
\|{\rm IV}\|_{L^2(\Gamma)\to H^1(\Gamma)}=\|\gamma(-\Delta+1)^{-1}\psi(k^{-1}|D|)\gamma^*\|_{L^2(\Gamma)\to H^1(\Gamma)}\leq C(\log (k+2))^{1/2}.
\end{equation}
Therefore, combining~\eqref{e:singleI},~\eqref{e:lowFreq},~\eqref{e:highFreqDiff}, and~\eqref{e:lowFreqI} with~\eqref{e:decompose}, we have
\begin{align}\nonumber
\|k^{-1}\nabla_{\Gamma}S_{\ri k}\|_{L^2(\Gamma)\to L^2(\Gamma)}&\leq k^{-1}\|{\rm I}\|_{L^2\to H^1}+\|{\rm II}\|_{L^2(\Gamma)\to H_k^{1}(\Gamma)}+\|{\rm III}\|_{L^2(\Gamma)\to H_k^1(\Gamma)}+k^{-1}\|{\rm IV}\|_{L^2(\Gamma)\to H^1(\Gamma)}\\
&\leq Ck^{-1}(\log (k+2))^{1/2};\label{e:derivativeBound}
\end{align}
the result then follows by combining \eqref{eq:Sik1knorm}, \eqref{e:L2Bound}, and \eqref{e:derivativeBound}.
\end{proof}

\section{New bounds on the exterior Neumann-to-Dirichlet map $\NtD$ and their proofs}\label{sec:NtD}

Recall that  $\NtD : \HmhG\rightarrow\HhG$ denotes the Neumann-to-Dirichlet map for the Helmholtz equation posed in $\Oe $ with the Sommerfeld radiation condition \eqref{eq:src}. 

\ble\label{lem:NtD1}
For all $k>0$, $\NtD$ has a unique extension to a bounded operator $\NtD: H^{-1}(\Gamma)\rightarrow \LtG$.
Furthermore $\NtD: H^{s-1/2}(\Gamma)\rightarrow H^{s+1/2}(\Gamma)$ and 
\beq\label{eq:NtDinter}
\N{ \NtD}_{H_k^{s-1/2}(\Gamma)\rightarrow H^{s+1/2}_k(\Gamma)}=
\N{ \NtD}_{L^{2}(\Gamma)\rightarrow H^{1}_k(\Gamma)}\quad\tfa |s|\leq 1/2.
\eeq
\ele

\bpf[Proof of Lemma \ref{lem:NtD1}]
The extension to $H^{-1}(\Gamma)\rightarrow \LtG$ follows from, e.g., \cite[Theorem 2.31]{ChGrLaSp:12}. By Green's second identity, the Sommerfeld radiation condition and the fact that $\langle\cdot,\cdot\rangle_{\Gamma,\Rea}$ is the \emph{real}-valued duality pairing (as opposed to the complex one; see \cite[Lemma 6.13]{Sp:15}).
\beq\label{eq:NtDdual2}
\big\langle \NtD \phi, \psi\rangle_{\Gamma,\Rea} = \big\langle\phi,\NtD \psi\rangle_{\Gamma,\Rea} \quad\tfa \phi,\psi \in H^{1/2}(\Gamma)
\eeq
(this relation was used in the form \eqref{eq:NtDdual} in the proof of Theorem \ref{thm:Hk}).
By density of $H^{1/2}(\Gamma)$ in $H^1(\Gamma)$, this last equation holds for all $\psi \in H^1(\Gamma)$, and thus
\beqs
\N{ \NtD}_{H^{-1}_k(\Gamma)\rightarrow L^{2}(\Gamma)}=\N{ \NtD}_{L^{2}(\Gamma)\rightarrow H^{1}_k(\Gamma)};
\eeqs
see, e.g., \cite[Lemma 2.3]{Sp:14}.
The bound \eqref{eq:NtDinter} then holds by interpolation; see, e.g., \cite[\S 2.3]{ChSpGiSm:20}, \cite[\S 4]{ChHeMo:15}.
\epf

Part (i) of the following theorem is from \cite[Theorem 1.5]{BaSpWu:16}; the other parts are stated and proved here for the first time (using the PDE results of \cite{LaSpWu:20, Bu:98, Vo:00}).

\begin{theorem}[Bounds on $\NtD$]\label{thm:NtD}

\

(i) If $\Oi$ is $C^\infty$ and nontrapping, then given $k_0>0$  there exists $C>0$ such that 
\beqs
\N{ \NtD}_{H_k^{s-1/2}(\Gamma)\rightarrow H^{s+1/2}_k(\Gamma)}
\leq C  k^{-\beta} \quad\tfa k\geq k_0 \text{ and for all } |s|\leq 1/2,
\eeqs
where $\beta=2/3$ if $\Gamma$ is curved, $\beta=1/3$ otherwise.

(ii) If $\Oi$ is Lipschitz then, given $k_0>0$ and $\delta>0$, there exists a set $J\subset [k_0,\infty)$ with $|J|\leq \delta$ such that given $\eps>0$ there exists $C=C(k_0,\delta,\eps)>0$ such that 
\beqs
\N{ \NtD}_{H_k^{s-1/2}(\Gamma)\rightarrow H^{s+1/2}_k(\Gamma)}
\leq C k^{5d/2+1 + \eps} \quad\tfa k \in [k_0,\infty)\setminus J  \text{ and for all } |s|\leq 1/2.
\eeqs 

(iii) If $\Oi$ is $C^{1,\sigma}$ for some $\sigma>0$, then, given $k_0>0$ and $\delta>0$, there exists a set $J\subset [k_0,\infty)$ with $|J|\leq \delta$ such that given $\eps>0$ there exists $C=C(k_0,\delta,\eps)>0$ such that 
\beqs
\N{ \NtD}_{H_k^{s-1/2}(\Gamma)\rightarrow H^{s+1/2}_k(\Gamma)}
\leq C k^{5d/2 + \eps} \quad\tfa k \in [k_0,\infty)\setminus J  \text{ and for all } |s|\leq 1/2.
\eeqs 

(iv) If $\Oi$ is $C^\infty$, then, given $k_0>0$ there exists $\alpha'>0$ and $C>0$ such that 
\beqs
\N{ \NtD}_{H_k^{s-1/2}(\Gamma)\rightarrow H^{s+1/2}_k(\Gamma)}
\leq C \exp (\alpha' k) \quad\tfa k\geq k_0  \text{ and for all } |s|\leq 1/2.
\eeqs 
\end{theorem}

Parts (ii)-(iv) of Theorem \ref{thm:NtD} are proved using Lemma \ref{lem:NtD1} and the following lemma.

\ble\label{lem:res}
Assume that, given $k_0>0$ and \(f\in L^2(\Oe )\) with support in $B_R$ for some $R>0$, the solution \(u\in H^1_{\mathrm{loc}}(\Oe )\) of the Helmholtz equation \(\Delta u + k^2 u=-f\) in \(\Oe \) that satisfies the Sommerfeld radiation condition \eqref{eq:src} and $\partial_n u=0$ on $\Gamma$ satisfies
\beq\label{eq:res}
\N{ \nabla u }_{L^2(\Oe \cap B_R)}+k \N{ u }_{L^2(\Oe \cap B_R)} \leq C_1\, K(k)\N{ f}_{L^2(\Oe )}
\eeq
for $k$ in some subset of $[k_0,\infty)$ and $C_1>0$ independent of $k$.
Then there exists $C_2>0$ 
such that, for $k$ in the same subset of $[k_0,\infty)$,
\beqs
\N{\NtD}_{\LtG\rightarrow H^1_k(\Gamma)} \leq C_2\, K(k). 
\eeqs
\ele

This result is analogous to the Dirichlet result in \cite[Lemma 4.2]{ChSpGiSm:20}. However, whilst the $k$-dependence in \cite{ChSpGiSm:20} is sharp, the $k$-dependence in Lemma \ref{lem:res} is not. Indeed, 
when $\Oi $ is nontrapping $K(k)\sim 1$ by the results of \cite{Va:75, MeSj:82} (see, e.g., the discussion in \cite[\S1.2]{BaSpWu:16}), but the sharp bound on $\NtD$ in this case is $\|\NtD\|_{\LtG\rightarrow H^1_k(\Gamma)} \leq C k^{-1/3}$ 
given by Part (i) of Theorem \ref{thm:NtD}.

 \bpf[Proof of Lemma \ref{lem:res}]
This result with $K(k) =1$ is proved in \cite[Theorem 1.5]{Sp:14}. The result for more general $K(k)$ follows in exactly the same way.
\epf

\bpf[Proof of Theorem \ref{thm:NtD}]
Part (i) is proved for $s=1/2$ in \cite[Theorem 1.8]{BaSpWu:16}, and then holds for all $|s|\leq 1/2$ by Lemma \ref{lem:NtD1}.
Under the assumptions of Part (ii), by \cite[Theorem 1.1 and Lemma 2.1]{LaSpWu:20}, 
given $k_0>0$ and $\delta>0$, there exists $J\subset [k_0,\infty)$ with $|J|\leq \delta$ such that 
 \eqref{eq:res} holds for $k\in [k_0,\infty)\setminus J$ with $K(k) = k^{5d/2+1+\eps}$.
  Under the assumptions of Part (iii), an analogous result holds with $K(k) = k^{5d/2+\eps}$ by \cite[Corollary 3.7]{LaSpWu:20}.
 Under the assumptions of Part (iv), \eqref{eq:res} holds for all $k$ with $K(k)= \exp(\alpha k)$, for some $\alpha>0$, by \cite[Theorem 2]{Bu:98}, \cite{Vo:00}.
 \epf

 \section{New results about the interior impedance-to-Dirichlet map $\ItDS$ and their proofs}\label{sec:ItDS}

\begin{lemma}[Existence of $\ItDR:\HhG\rightarrow \HhG$]\label{lem:ItDS1}
Let $\Gamma$ be Lipschitz and $s>1/2$. Suppose $R:H^{-1/2}(\Gamma)\to H^{1/2}(\Gamma)$ is bounded, there exist $C_1, C_2>0$ such that
\beq\label{eq:Rcoer}
|\Re\big\langle \Reg\psi,\psi \big\rangle_{\Gamma}|\geq C_1 \N{\psi}^2_{\HmhG} \quad\tfa \psi \in \HmhG, 
\eeq
\beq\label{eq:ImR}
|\Im\big\langle \Reg\psi,\psi \big\rangle_{\Gamma}|\leq C_2 \N{\psi}^2_{H^{-s}(\Gamma)} \quad\tfa \psi \in H^{-1/2}(\Gamma), 
\eeq
and $R^{-1}:H^{1-s}(\Gamma)\to H^{-s}(\Gamma)$ is bounded.

Then, given $g\in \HhG$, $k>0$, and $\eta\in \Rea\setminus\{0\}$, there exists a unique solution $u\in H^1(\Oi)$ to the boundary-value problem 
\beq\label{eq:ItDSdef2}
\Delta u + k^2 u =0 \quad\tin \Oi  \quad\tand \quad R\partial_n^- u - \ri \eta \gamma^- u = g \quad\ton \Gamma,
\eeq
and thus $\ItDR:\HhG\rightarrow\HhG$ is well-defined.
\ele

\bpf
The variational formulation of the boundary-value problem \eqref{eq:ItDSdef} with $g\in H^{1/2}(\Gamma)$ is: 
\beq\label{eq:vp}
\text{ find } u\in H^1(\Oi) \tst a(u,v)= F(v) \quad\tfa v\in H^1(\Oi), 
\eeq
where
\beqs
a(u,v):= \int_\Oi \Big(\nabla u \cdot \overline{\nabla v} - k^2 u\,\overline{v} \Big)- \ri \eta \big\langle \Reg^{-1} \gamma^- u , \gamma^- v\big\rangle_\Gamma
\quad\tand\quad
F(v):=  \big\langle \Reg^{-1}g , \gamma^- v\big\rangle_\Gamma.
\eeqs
We now show that the solution of this problem, if it exists, is unique.  
By \eqref{eq:Rcoer} 
if $g=0$, then
\begin{align*}
0=|\Im a(u,u) |= \big| \eta \Re\big\langle \Reg^{-1} \gamma^- u,\gamma^- u\big\rangle_\Gamma\big| \geq  \eta \,C_1\N{R^{-1}\gamma^- u}^2_{H^{-1/2}(\Gamma)},
\end{align*}
and uniqueness follows since $R$ is invertible from $\HmhG\rightarrow \HhG$ by \eqref{eq:Rcoer} and the Lax--Milgram theorem (see, e.g., \cite[Lemma 2.32]{Mc:00}).

To prove existence, first observe that, by \eqref{eq:ImR} and the assumption that $R^{-1}:H^{1-s}(\Gamma)\to H^{-s}(\Gamma)$ is bounded,
for $\psi\in H^{1/2}(\Gamma)$, 
\begin{align*}
\big|\Im \big\langle R^{-1}\psi, \psi\big\rangle_{\Gamma}\big|&= \big|\big\langle (R^{-1})^*(\Im R)R^{-1} \psi,\psi\big\rangle_{\Gamma}\big|\\
&= \big|\big\langle (\Im R)R^{-1} \psi,R^{-1}\psi\big\rangle_{\Gamma}\big|\leq C\|R^{-1}\psi\|_{H^{-s}(\Gamma)}^2\leq C\|\psi\|_{H^{1-s}(\Gamma)}^2.
\end{align*}
Therefore, for $v\in H^1(\Oi)$ and $s>1/2$,
\begin{align*}
\Re a(v,v) = \N{\nabla v}^2_{L^2(\Oi)} - k^2 \N{v}^2_{L^2(\Oi)}+\eta \Im \langle R^{-1}\gamma^{-}v,\gamma^{-}v\rangle_{\Gamma} \geq \|v\|_{H^1(\Oi)}^2-C\|v\|_{H^{\frac{3}{2}-s}(\Oi)}^2.
\end{align*}
Since $H^1(\Oi)$ is compactly contained in $H^{3/2-s}(\Oi)$ with $s>1/2$ (see, e.g., \cite[Theorem 3.27]{Mc:00}),
the solution of the variational problem \eqref{eq:vp} exists and is unique in $H^1(\Oi)$ by Fredholm theory (see, e.g., \cite[Theorem 2.34]{Mc:00}).
\epf

We now show that if $\Gamma$ is $C^\infty$  and $\Reg$ satisfies Assumption \ref{ass:Reg} then $\Reg$ satisfies the assumptions of Lemma \ref{lem:ItDS1} for sufficiently large $k$, and hence that $\ItDS:\HhG\rightarrow \HhG$ exists for sufficiently large $k$.

\ble\label{lem:Rcoer}
If $\Gamma$ is $C^\infty$ and $\Reg$ satisfies Assumption \ref{ass:Reg}, then there exists $k_0>0$ and $C>0$ such that, for all $k \geq k_0$, 
\beq\label{eq:Rcoer2}
\pm \Re \big\langle \Reg\psi,\psi \big\rangle_{\Gamma}\geq C k^{-1}\N{\psi}^2_{H^{-1/2}_{k}(\Gamma)} \quad\tfa \psi \in \HmhG,
\eeq
where the plus sign is chosen if $\sigma_\hsc(R)$ is positive, and the minus sign is chosen if $\sigma_\hsc(R)$ is negative. Moreover, $R^{-1}:H^{1-s}_k\to H^{-s}_k$ for any $s$ and 
\beq
\label{e:mitten}
|\Im \langle R\psi,\psi\rangle_{\Gamma}|\leq Ck^{-1}\|\psi\|_{H_k^{-1}(\Gamma)}^2\quad \tfa \psi\in H^{-1/2}(\Gamma).
\eeq
\ele

\bpf
Let $\widetilde{R}=\hsc^{-1}R$. Since $\sigma_\hsc(\widetilde{R})$ is real and $\widetilde{R}\in \Psi_{\hsc}^{-1}$ is elliptic,
\beq\label{eq:Rcoer3}
\pm \sigma_\hsc(\widetilde{R}) \geq \frac{C_1}{\langle \xi' \rangle} \quad\tfa (x',\xi')\in T^*\Gamma.
\eeq
If \eqref{eq:Rcoer3} holds with the plus sign, then $A:= \widetilde{R} - C_1/\langle \xi'\rangle$ satisfies the assumption of Theorem \ref{thm:sharpG} with $\ell=-1$, and thus
\beqs
\Re \left\langle \left( \widetilde{R} - \frac{C_1}{\langle \xi'\rangle}\right) \phi,\phi\right\rangle_\Gamma \geq - C_2 \hsc \N{\phi}^2_{H^{-1}_\hsc(\Gamma)} \quad \tfa \phi \in H^{-1/2}(\Gamma).
\eeqs
Therefore, by the definition of the $H^{-1/2}_\hsc(\Gamma)$ norm,
\beqs
\Re \big\langle  \widetilde{R} \phi,\phi\big\rangle_\Gamma \geq C_3 \N{\phi}^2_{H^{-1/2}_\hsc(\Gamma)}- C_2 \hsc \N{\phi}^2_{H^{-1}_\hsc(\Gamma)} \quad \tfa \phi \in H^{-1/2}(\Gamma),
\eeqs
and the bound \eqref{eq:Rcoer2} with the plus sign follows if $\hsc$ is sufficiently small (i.e., for all $k$ sufficiently large). 
If \eqref{eq:Rcoer3} holds with the minus sign, then $A:= -\widetilde{R} - C/\langle \xi\rangle$ satisfies the assumption of Theorem \ref{thm:sharpG} with $\ell=-1$, and the bound \eqref{eq:Rcoer2} with the minus sign follows in a similar way to above.
 
 The fact that $R^{-1}:H_{\hsc}^{1-s}\to H^{-s}_{\hsc}$ follows from the fact that $R^{-1} \in \Psi_\hsc^{1}(\Gamma)$ by Lemma~\ref{thm:elliptic} and the mapping property in Part (ii) of Theorem \ref{thm:basicP}. To check that \eqref{e:mitten} holds, observe that since $\sigma^{-1}_{\hbar}(\widetilde{R})$ is real, by the second equation in \eqref{eq:multsymb},
 $$
 \sigma^{-1}_{\hbar}(\widetilde{R}-\widetilde{R}^*)=0.
 $$
Therefore  $\Im \widetilde{R}=(\widetilde{R}-\widetilde{R}^*)/(2\ri) \in \Psi_{\hbar}^{-2}(\Gamma)$ and hence, by the 
boundedness properties of elements of $\Psi_{\hsc}^{-2}$,
 $$
 |\Im \langle R\psi,\psi\rangle_{\Gamma}|=
\hsc |\Im \langle \widetilde{R}\psi,\psi\rangle_{\Gamma}|\leq \hsc\|\Im \widetilde{R}\psi\|_{H_\hbar^{1}}\|\psi\|_{H_\hbar^{-1}}\leq C\hsc\|\psi\|_{H_\hbar^{-1}}^2,
 $$
 which is \eqref{e:mitten}.
\epf

In \S\ref{sec:proofinvert}, Lemma \ref{lem:ItDS1} is used to prove invertibility of $\Breg$ and $\Breg'$ when $\Reg$ satisfies Assumption \ref{ass:Reg}, $k>0$, and $\eta \in \Rea\setminus\{0\}$ (i.e., Part (i) of Theorem \ref{thm:invert}).
We now use this invertibility of $\Breg'$ to prove the following result about $\ItDS$ and $\ItDSp$.

Recall that \(\langle\cdot , \cdot \rangle_{\Gamma,\Rea}\) denotes the real-valued duality pairing on $\Gamma$ between $H^{s}(\Gamma)$ and $H^{-s}(\Gamma)$.

\ble\label{lem:ItDS2}
Assume that $\Gamma$ is $C^\infty$, $\Reg$ satisfies Assumption \ref{ass:Reg}, $\eta \in \Rea\setminus\{0\}$, and $k>0$ is sufficiently large so that $\ItDS:\HhG\rightarrow \HhG$ exists.

(i) $\ItDS$ and $\ItDSp$ have unique extensions to bounded operators
$H^1(\Gamma)\rightarrow H^1(\Gamma)$, 
where $\Reg'$ denotes the adjoint of $\Reg$ with respect to the real-valued $L^2$ inner product.

(ii) $(\Reg^{-1}  \ItDS)'= (\Reg')^{-1}\ItDSp$, i.e., 
    \begin{align}\label{eq:ItDdual}
        \big\langle \Reg^{-1}  \ItDS \phi, \,\psi\big\rangle_{\Gamma,\Rea}=\big\langle\phi,\, (\Reg')^{-1}\ItDSp \psi\big\rangle_{\Gamma,\Rea} \quad\tfa \phi,\psi\in \HhG.
    \end{align}

(iii) 
\beqs
 \N{\Reg^{-1} \ItDS}_{\LtG \rightarrow H^{-1}_k(\Gamma)}=\N{(\Reg')^{-1}\ItDSp}_{H^1_k(\Gamma)\rightarrow L^2(\Gamma)},
\eeqs
and thus both $\ItDS$ and $\ItDSp$ are also bounded operators $\LtGt$.
\end{lemma}

\bpf
(i) 
The plan is to express $\ItDS$ as an operator on $\HhG$ in terms of $(\Bregp)^{-1}$; indeed we show that
\beq\label{eq:ItDSformula}  
\ItDS = -R\big(\Bregp\big)^{-1}\left( \frac{1}{2}I - \DL_k'\right) \Reg^{-1},
\eeq
and then show that this expression extends $\ItDS$ to an operator on $\HoG$.

Given $g\in \HhG$, let $u$ be the solution to \eqref{eq:ItDSdef2}. By Green's integral representation (see, e.g., \cite[Theorem 2.20]{ChGrLaSp:12}), 
\beqs
u = \cS_k \partial_n^- u - \cK_k \gamma^- u.
\eeqs
Taking the Neumann trace and using the jump relations \eqref{eq:jumprelations}, we find that
\beqs
\left( \frac{1}{2}I-\DL_k'\right)\partial_n^- u + H_k \gamma^- u=0.
\eeqs
The boundary condition in \eqref{eq:ItDSdef} implies that 
\beqs
\partial_n^- u = \ri \eta \Reg^{-1} \gamma^- u + \Reg^{-1} g, 
\eeqs
and combining the last two displayed equations, we find that 
\beqs
\ri \eta\left( \frac{1}{2}I - \DL_k'\right) \big( \Reg^{-1} \gamma^-u\big) + H_k \gamma_- u = -\left( \frac{1}{2}I - \DL_k'\right)\Reg^{-1}g,
\eeqs
that is, by the definition of $\Bregp$ \eqref{eq:BIEs},  
\beq\label{eq:rain1}
\Bregp \big( \Reg^{-1} \gamma^-u\big)= -\left( \frac{1}{2}I - \DL_k'\right) \Reg^{-1}g.
\eeq

The arguments in the proof of Theorem \ref{thm:invert} (in \S\ref{sec:proofinvert}) that show $\Bregp$ is invertible from $\LtGt$ also show that 
$\Bregp$ is invertible from $\HmhG$ to $\HmhG$ (indeed, the proof of Theorem \ref{thm:invert} shows injectivity on $\HmhG$, and the proof that $\Bregp$ is Fredholm on $\LtG$ also shows that $\Bregp$ is Fredholm on $\HmhG$). 
Combining this result with the mapping properties of $\Reg, \Reg^{-1}$, and $\DL_k'$ (see \eqref{eq:mapping}), we see that, given $g \in \HhG$, 
\beqs
\Reg\big(\Bregp)^{-1} \left( \frac{1}{2}I - \DL_k'\right) \Reg^{-1}g \in \HhG.
\eeqs
 By Lemma \ref{lem:ItDS1} and Assumption \ref{ass:Reg}, $\ItDS$ is well-defined on $\HhG$.
Therefore, \eqref{eq:rain1} implies that \eqref{eq:ItDSformula} holds
with both sides well-defined on $\HhG$.

Using that $\Reg:\LtG\rightarrow \HoG$ is bounded and invertible, $K_k'$ is bounded on $\LtG$, and $\Bregp$ is bounded and invertible on $\LtG$, 
we find that \eqref{eq:ItDSformula} extends $\ItDS$ to a well-defined operator on $\HoG$.
Since $\Reg'$ also satisfies Assumption \ref{ass:Reg}, $\ItDSp$ also extends to a well-defined operator on $\HoG$.

(ii) To prove \eqref{eq:ItDdual}, 
   let $u$ be the solution of   \eqref{eq:ItDSdef} with data $\phi$ and 
   let $v$ be the solution of   \eqref{eq:ItDSdef} with data $\psi$ and $\Reg$ replaced by $\Reg'$; 
   then
   \beq\label{eq:uv}
   \gamma^- u = \ItDS \phi \quad\tand\quad    \gamma^- v = \ItDSp \psi.
   \eeq
    By Green's second identity applied in $\Oi$ (see, e.g., \cite[Theorem 2.19]{ChGrLaSp:12}),
    \begin{align*}
\big\langle \gamma^- u, \partial_n^- v \big\rangle_{\Gamma,\Rea} - \langle \gamma^- v, \partial_n^- u \rangle_{\Gamma,\Rea} = 
 \int_{\Oi } u \Delta v - v \Delta u = 0,
    \end{align*}
and thus, by using the boundary conditions satisfied by $u$ and $v$,
\beqs
        \big\langle \gamma^-u , \,(\Reg')^{-1} (\psi + \ri \eta \gamma^-v)\big \rangle_{\Gamma,\Rea}  = \big\langle \gamma^-v , \,\Reg^{-1} (\phi + \ri \eta \gamma^-u) \big\rangle_{\Gamma,\Rea}.
 \eeqs       
By the definition of $\Reg'$,
the last equality becomes 
        \beqs
        \big\langle \gamma^-u , (\Reg')^{-1} \psi \big\rangle_{\Gamma,\Rea}  = \big\langle \gamma^-v , \Reg^{-1} \phi \big\rangle_{\Gamma,\Rea},
        \eeqs
and then \eqref{eq:ItDdual} follows by using \eqref{eq:uv}.
           
\sloppy (iii) To prove $\Reg^{-1}\ItDS$ is bounded $L^2(\Gamma)\rightarrow H^{-1}(\Gamma)$, it is sufficient to show that
$\langle \Reg^{-1}\ItDS\phi,\psi\rangle_{\Gamma,\Rea}$ is bounded for all $\phi \in \LtG$ and $\psi\in \HoG$.
Given $\phi\in\LtG$, there exists $\phi_j\in \HhG$ such that $\phi_j\rightarrow \phi$ in $\LtG$ as $j\tendi$. Then, by \eqref{eq:ItDdual}, for all $\psi\in \HoG$,
\beqs
        \big\langle \Reg^{-1} \ItDS \phi_j, \psi \big\rangle_{\Gamma,\Rea} =         \big\langle (\Reg')^{-1} \ItDSp \psi, \phi_j \big\rangle_{\Gamma,\Rea} \rightarrow \big\langle (\Reg')^{-1} \ItDSp \psi, \phi \big\rangle_{\Gamma,\Rea} \quad\tas j\tendi,
\eeqs        
since $\ItDSp: \HoG\rightarrow\HoG$ is bounded, and thus $ \Reg^{-1} \ItDS:\HoG\rightarrow \LtG$ is bounded. 
        
Therefore, \eqref{eq:ItDdual} holds for all $\phi\in\LtG$ and $\psi\in\HoG$, and thus
    \begin{align*}
        \N{\Reg^{-1} \ItDS}_{L^2(\Gamma)\rightarrow H^{-1}_k(\Gamma) } 
        &=\sup_{\substack{\phi \in L^2(\Gamma),\psi \in\HoG}} \dfrac{\big\langle \Reg^{-1}\ItDS \phi , \psi \big\rangle_{\Gamma,\Rea}}{\N{ \phi}_{L^2(\Gamma)}\N{\psi}_{H^1_k(\Gamma)}} \\
        &=\sup_{\substack{\phi \in L^2(\Gamma),\psi \in\HoG}} \dfrac{\big\langle (\Reg')^{-1}\ItDSp \psi , \phi \big\rangle_{\Gamma,\Rea}}{\N{ \phi}_{L^2(\Gamma)}\N{\psi}_{H^1_k(\Gamma)}} \\
        &=        \N{(\Reg')^{-1} \ItDSp}_{H^1_k(\Gamma)\rightarrow L^2(\Gamma) }. 
                    \end{align*}
\end{proof}

\begin{corollary}\label{cor:ItDS}
If $\Gamma$ is $C^\infty$,  $k>0$, and $\Re \eta \neq 0$, then
\beq\label{eq:ItDSnormequiv}
    \N{\ItDS}_{\LtGt} \sim    \N{\ItDSp}_{H^1_k(\Gamma)\rightarrow H^1_k(\Gamma)}.
\eeq
\end{corollary}

\bpf[Proof of Corollary \ref{cor:ItDS}]
We claim that 
\beqs
k \N{\ItDSp}_{H^1_k(\Gamma)\rightarrow H^1_k(\Gamma)} \lesssim \N{(\Reg')^{-1} \ItDSp}_{H^1_k(\Gamma)\rightarrow \LtG}
\lesssim k \N{\ItDSp}_{H^1_k(\Gamma)\rightarrow H^1_k(\Gamma)} 
\eeqs
and 
\beqs
k \N{\ItDS}_{\LtGt} \lesssim \N{\Reg^{-1} \ItDS}_{\LtG\rightarrow H^{-1}_k(\Gamma)}
\lesssim k \N{\ItDS}_{\LtGt}.
\eeqs
Indeed, the first bound in each inequality follows from the upper bound in \eqref{eq:Rbound1} (first applied to $\Reg'$ with $t=0$ and then applied to $\Reg$ with $t=-1$), and the second bound in each inequality follows from the upper bound in \eqref{eq:Rbound2} (first applied to $\Reg'$ with $t=0$ and then applied to $\Reg'$ with $t=-1$); the result then follows from Lemma \ref{lem:ItDS2}.
\epf

\begin{theorem}\label{thm:ItDS}
Suppose $\Gamma$ is $C^\infty$, $R$ satisfies Assumption \ref{ass:Reg}, and $\eta\in\mathbb{R}\setminus\{0\}$ is independent of $k$. Then there exists $k_0>0$ and $C>0$ 
such that, for all $k\geq k_0$,  
\beqs
\N{\ItDS}_{\LtGt} \leq C.
\eeqs
\end{theorem}

\begin{proof}
By Lemma \ref{lem:ItDS2} and \eqref{eq:ItDSnormequiv}, it is sufficient to show that 
\beqs
        \N{\ItDSp}_{H^1_k(\Gamma)\rightarrow H^1_k(\Gamma)}\leq C;
\eeqs
i.e., that 
given $g\in H^1(\Gamma)$, the solution $u$ to \eqref{eq:ItDSdef} with $\Reg$ replaced by $\Reg'$ exists and satisfies
        \beq\label{eq:PItD2}
        \N{u}_{H^1_k(\Gamma)}\leq C \N{g}_{H^1_k(\Gamma)}
        \eeq
        with $C>0$ independent of $k$.
        By Assumption \ref{ass:Reg}, $\Reg' = \hsc \widetilde{\Reg}$ with $\widetilde{\Reg} \in \Psi_{\hsc}^{-1}(\Gamma)$. Using this 
 in the boundary-value problem defining $u$ \eqref{eq:ItDSdef} and multiplying by $\widetilde{\Reg}^{-1}$, we obtain that 
        \beq\label{eq:PItD3}
        (-\hsc^2\Delta -1)u =0 \quad\tin \Oi  \quad\tand \quad \frac{1}{\ri}\hsc\partial_n u -(\widetilde{\Reg})^{-1} \eta u =\widetilde{g}\quad\ton \Gamma,
        \eeq
        where 
        \beqs
        \widetilde{g}:= \frac{1}{\ri} (\widetilde{\Reg})^{-1} g \in \LtG.
        \eeqs
By Part (ii) of Theorem \ref{thm:basicP} (or equivalently Corollary \ref{cor:Rbound}),
$\|\widetilde{\Reg}^{-1}\|_{H^1_\hsc(\Gamma)\to L^2(\Gamma)}\leq C$, and thus 
        \beq\label{eq:gtilde}
        \big\|\widetilde{g}\big\|_{L^2(\Gamma)} \leq C\N{g}_{H^1_\hsc(\Gamma)}.
        \eeq

        The boundary-value problem \eqref{eq:PItD3} fits in the framework studied in \cite[Section 4]{GaLaSp:21} with, in the notation of \cite{GaLaSp:21}, $\mathcal{N} = 1, \mathcal{D}=\eta(\widetilde{\Reg})^{-1}$, $m_1=0$, and $m_0=1$.
Whereas \cite[Section 4]{GaLaSp:21} studies this problem when $\Oi$ is curved, the results hold for general smooth $\Oi$ if \cite[Lemma 3.3]{GaMaSp:21} is used instead of \cite[Lemma 4.8]{GaLaSp:21}. When applying \cite[Lemma 3.3]{GaMaSp:21} to the set up in \cite{GaLaSp:21}, we note that, since $\Oi$ is bounded, the set $A$ in \cite[Lemma 3.3]{GaMaSp:21} is the whole of $S^*_{\Oi}\Rea^d$. 
Therefore, when $\eta\, \sigma_\hsc(\widetilde{\Reg})>0$, 
        the result \cite[Theorem 4.6]{GaLaSp:21} (combined with \cite[Lemma 3.3]{GaMaSp:21} as indicated above) shows that 
the solution  $u$ to \eqref{eq:ItDSdef} with $\Reg$ replaced by $\Reg'$ exists and satisfies
        \beq\label{eq:PItD4}
        \N{u}_{H^1_\hsc(\Oi )}\leq C \N{\widetilde{g}}_{L^2(\Gamma)},
        \eeq
        with $C>0$ is independent of $\hsc$.
        Then, inputting $m_1=0$ and $m_0=1$ into the trace result \cite[Theorem 4.1]{GaLaSp:21} and choosing $\ell=0$, we find that, 
                for $C>0$ (independent of $\hsc$), 
        \beq\label{eq:PItD5}
        \N{u}_{H^1_\hsc(\Gamma)} \leq C \big( \N{u}_{L^2(\Oi )} + \N{\widetilde{g}}_{L^2(\Gamma)}\big).
        \eeq

                The combination of  \eqref{eq:PItD5}, \eqref{eq:PItD4}, and \eqref{eq:gtilde} therefore gives the required result \eqref{eq:PItD2} when $\eta\, \sigma_\hsc(\widetilde{\Reg})>0$.
When $\eta\, \sigma_\hsc(\widetilde{\Reg})<0$, these results of \cite[Section 4]{GaLaSp:21} 
(combined with \cite[Lemma 3.3]{GaMaSp:21} as indicated above) 
apply to the boundary-value problem
        \beqs
        (-\hsc^2\Delta -1)v =0 \quad\tin \Oi  \quad\tand \quad \frac{1}{\ri}\hsc\partial_n v +(\widetilde{\Reg}^*)^{-1} \eta v =\overline{\widetilde{g}}\quad\ton \Gamma.
        \eeqs
        Since $v = \overline{u}$ where $u$ is the solution of \eqref{eq:PItD3}, the bound \eqref{eq:PItD2} also holds when 
        $\eta\, \sigma_\hsc(\widetilde{\Reg})<0$, and the proof is complete.
    \end{proof}

In \S\ref{sec:eta} (the discussion on the choice of $\eta$) we use the following lemma.

\ble\label{lem:generalboundeta}
If $\eta \in \Rea\setminus\{0\}$ and $\Reg$ satisfies Assumption \ref{ass:Reg} then there exists $k_0>0$ and $C>0$ such that for all $k\geq k_0$
\beq\label{eq:generalboundeta}
\big\|\ItDS\big\|_{H^{1/2}_k(\Gamma) \rightarrow H^{1/2}_k(\Gamma)} \leq \frac{C }{|\eta|}.
\eeq
\ele

\bpf 
By Lemmas \ref{lem:ItDS1} and \ref{lem:Rcoer}, the solution $u$ of the variational problem \eqref{eq:vp} defining $\ItDS$ exists. Choosing $v=u$ \eqref{eq:vp} and taking the imaginary part, we obtain that
\beqs
\eta \Re \big\langle R^{-1} \gamma^- u, \gamma^- u\big\rangle_\Gamma = \Im \big\langle R^{-1} g, \gamma^-u \big\rangle_\Gamma.
\eeqs
By the coercivity of $R$ \eqref{eq:Rcoer2},
\beqs 
C |\eta| k^{-1} \N{R^{-1}\gamma^- u }^2_{H^{-1/2}_k(\Gamma)} \leq \N{R^{-1}g}_{H^{-1/2}_k(\Gamma)} \N{\gamma^- u}_{H^{1/2}_k(\Gamma)}.
\eeqs
where $C$ is as in \eqref{eq:Rcoer2}; the result then follows by using both the upper- and lower-bounds on $R^{-1}$ in \eqref{eq:Rbound2}.
\epf

By considering Neumann eigenfunctions, we see that the bound \eqref{eq:generalboundeta} is sharp in its $\eta$-dependence. Indeed, if $k^2$ is a Neumann eigenvalue of the Laplacian with $u$ the corresponding eigenfunction, then 
$$
\Reg \partial_n^-u-\ri\eta \gamma^-u=-\ri\eta \gamma^-u
\quad\text{ and thus } \quad
\ItDS \gamma^- u=\frac{\ri}{\eta}\gamma^-u.
$$

\section{Proofs of the main results}\label{sec:mainproofs}

\subsection{Proof of Theorem \ref{thm:upper_bound_norm}}

\bpf[Proof of Theorem \ref{thm:upper_bound_norm}]
By the triangle inequality, 
\begin{align*}
    \N{\Breg}_{\LtGt} 
    \leq |\eta| \big( 1/2 +  \N{ \DL_k}_{L^2(\Gamma)\rightarrow L^2(\Gamma)}\big) + \N{ \Reg}_{H^{-1}_k(\Gamma)\rightarrow L^2(\Gamma)}\N{ H_k}_{L^2(\Gamma)\rightarrow H^{-1}_k(\Gamma)}.
\end{align*}
The bounds on $\|\Breg\|_{\LtGt}$ in Theorem~\ref{thm:upper_bound_norm} then follow by using the bounds in Theorems~\ref{th:\DL_k_bounds}, Theorem~\ref{thm:Hk}, and \ref{thm:newSikbound}. The bounds on $\|\Bregp\|_{\LtGt}$ follow similarly.
\epf

Because of the interest in choosing $R=S_0$ (see the discussion in \S\ref{sec:rationale}), we also record the following bound on $\|B_{k,\eta,S_0}\|_{\LtGt}$ and $\|B'_{k,\eta,S_0}\|_{\LtGt}$ (which we refer to in \S\ref{sec:num}). For simplicity we assume that $|\eta|$ is bounded independently of $k$ (since then the norm of $S_0 H_k$ dominates for all geometries), but it is straightforward to obtain bounds for general $\eta$ analogous to those in Theorem \ref{thm:upper_bound_norm}.

\ble\label{lem:S0}
If $|\eta|\leq c$ with $c$ independent of $k$ and $\Gamma$ is piecewise smooth, then given $k_0>0$ there exists $C>0$ such that for all $k\geq k_0$
\beqs
    \N{B_{k,\eta,S_0}}_{\LtGt} +    \N{B'_{k,\eta,S_0}}_{\LtGt} \leq C k\log (k+2).
\eeqs
If, in addition, $\Gamma$ is $C^\infty$ and curved then given $k_0>0$ there exists $C>0$ such that for all $k\geq k_0$
\beqs
    \N{B_{k,\eta,S_0}}_{\LtGt} +    \N{B'_{k,\eta,S_0}}_{\LtGt} \leq C k.
\eeqs
\ele

\bpf
By the triangle inequality, 
\begin{align*}
    \N{B_{k,\eta,S_0}}_{\LtGt} 
    \leq |\eta| \big( 1/2 +  \N{ \DL_k}_{L^2(\Gamma)\rightarrow L^2(\Gamma)}\big) + \N{ S_0}_{H^{-1}(\Gamma)\rightarrow L^2(\Gamma)}\N{ H_k}_{L^2(\Gamma)\rightarrow H^{-1}(\Gamma)}.
\end{align*}
The result then follows from Theorems~\ref{th:\DL_k_bounds} and Theorem~\ref{thm:Hk}, using the fact that
\beqs
\N{H_k}_{L^2(\Gamma)\rightarrow H^{-1}(\Gamma)} = \N{H_k}_{H^1(\Gamma)\rightarrow \LtG}\leq k_0^{-1}
 \N{H_k}_{H^1_k(\Gamma)\rightarrow \LtG},
\eeqs
where the equality follows from \eqref{eq:Hdual} and the inequality holds since $\|\phi\|_{H^1_k(\Gamma)}\leq k_0^{-1}\|\phi\|_{H^1(\Gamma)}$ for $k\geq k_0$ and all $\phi \in H^1(\Gamma)$.
\epf

\subsection{Proof of Theorem \ref{thm:invert}}\label{sec:proofinvert}

\ble\label{lem:quasiadjoint}
\beq\label{eq:norms2}
\big\langle \Breg \phi, \psi\big\rangle_{\Gamma,\Rea} = \big\langle\phi, B_{k,\eta,R'}' \psi \big\rangle_{\Gamma,\Rea} \quad \tfa \phi,\psi\in \LtG;
\eeq
i.e., $B_{k,\eta,R'}'$ is the adjoint of $\Breg$ with respect to the real-valued $\LtG$ inner product. 
\ele

\bpf
By, e.g., \cite[Equation 2.40]{ChGrLaSp:12}, $\langle \DL_k \phi,\psi\rangle_{\Gamma,\Rea} = \langle \phi, \DL'_k \psi\rangle_{\Gamma,\Rea}$ for all $\phi,\psi\in\LtG$.
Since $R:\LtG\rightarrow \HoG$ is bounded, 
the result then follows from \eqref{eq:Hdual} and the definitions of $\Breg$ and $B_{k,\eta,R'}'$ \eqref{eq:BIEs}.
\epf

\begin{corollary}\label{cor:norms}
\beqs
\N{\Breg}_{\LtGt} = \N{B_{k,\eta,R'}'}_{\LtGt}
\eeqs
Furthermore, if $\Breg$ is invertible on $\LtG$, then 
\beqs
\N{(\Breg)^{-1}}_{\LtGt} = \N{(B_{k,\eta,R'}')^{-1}}_{\LtGt}.
\eeqs
\end{corollary}

\bpf
This follows from Lemma \ref{lem:quasiadjoint} and, e.g., \cite[Remark 2.24]{ChGrLaSp:12}.
\epf

\bpf[Proof of Theorem \ref{thm:invert}]
We first prove that if $R$ is 
satisfies the assumptions of Lemma  \ref{lem:ItDS1}, then $\Bregp$ and $\Breg$ are injective on $\HmhG$; 
 injectivity of $\Bregp$ on $\LtG$ immediately follows.
Suppose $\phi\in \HmhG$ is such that $\Bregp\phi=0$. Let $u = (\cK_k R - \ri \eta \cS_k)\phi$.
The jump relations \eqref{eq:jumprelations} imply that $\partial_n^+ u = \Bregp \phi=0$ (this is the same argument 
used to derive the BIE \eqref{eq:direct}). Since $u$ satisfies the Sommerfeld radiation condition \eqref{eq:src}, and the solution of the exterior Neumann problem is unique, $u=0$ in $\Oe$, and thus $\gamma^+ u=0$. The jump relations imply that
\beq\label{eq:jump3}
\partial_n^+ u -\partial_n^- u = \ri \eta \phi \quad\tand\quad \gamma^+ u- \gamma^- u = R\phi,
\eeq
and thus $\partial_n^- u = -\ri \eta \phi$ and $\gamma^- u = -R\phi$. Therefore, $u$ solves the boundary-value problem \eqref{eq:ItDSdef2} with $g=0$. By Lemma \ref{lem:ItDS1}, $u=0$ in $\Oi $. Therefore $\partial_n^- u =0$, and the first equation in \eqref{eq:jump3} implies that $\phi =0$. 
Now suppose $\phi\in \HmhG$ is such that $\Breg\phi=0$. Let $u = \cK_k \phi$; the third and fourth jump relations in \ref{eq:jumprelations} imply that
        \begin{align}\label{eq:jumpproof}
            \gamma^\pm u = \left(\pm \frac{1}{2}I + \DL_k\right)\phi, \quad \partial^\pm_n u= H_k \phi.
        \end{align}
Therefore, since $\Breg\phi=0$, $R\partial_n^- u - \ri \eta \gamma^- u=0$. Similar to above, $u=0$ in $\Oi $ by Lemma \ref{lem:ItDS1}, and thus $\partial^-_n u=0$. By \eqref{eq:jumpproof}, $\partial_n^+ u=0$, and by uniqueness of the Helmholtz exterior Neumann problem with the radiation condition \eqref{eq:src} (which holds when $\Gamma$ is Lipschitz by, e.g.,\cite[Corollary 2.9]{ChGrLaSp:12}), $u=0$ in $\Oe $. Therefore, $\phi = \gamma^+ u - \gamma^- u=0$.

We now need to check that $R$ satisfies the assumptions of Lemma  \ref{lem:ItDS1} if one of the following three conditions holds:~(a) $R$ satisfies Assumption \ref{ass:Reg} and $k$ is sufficiently large, (b) $R=S_{\ri k}$ and $k>0$, and (c) $R=S_0$ and (in 2-d) the constant $a$ in \eqref{eq:LaplaceFund} is larger than the capacity of $\Gamma$. Indeed, the assumptions of Lemma  \ref{lem:ItDS1} hold under (a) by Lemma \ref{lem:Rcoer}, the assumptions of Lemma  \ref{lem:ItDS1} hold under (b) by Theorem \ref{thm:Sikcoer} and \cite[Theorem 7.17]{Mc:00}, and 
the assumptions of Lemma  \ref{lem:ItDS1} hold under (c) by \cite[Corollary 8.13, Theorem 8.16, and Theorem 7.17]{Mc:00}); the injectivity results in both Parts (i) and (ii) therefore follow.

We now complete the proof of Part (i) by showing that, when $\Gamma$ is $C^\infty$ and $\Reg$ satisfies Assumption \ref{ass:Reg}, $\Breg$ and $\Bregp$ are Fredholm on $\LtG$ for $k$ sufficiently large.
Let $\chi\in C_{\rm comp}^\infty(\mathbb{R})$ 
be as in Theorem \ref{thm:HFSD} (i.e., $\chi(\xi)= 1$ for $|\xi|\leq 2$ and $\chi(\xi)=0$ for $|\xi|\geq 3$).
By the definition of $\Bregp$ \eqref{eq:BIEs}, $\Bregp:= B'_1 + B'_2$, where
\beqs
B_1':= \frac{\ri \eta}{2}I + \big(1 - \chi(|\hsc D'|_g)\big) H_k R
\quad\tand\quad
B_2':= \chi(|\hsc D'|_g) H_k R - \ri \eta K_k'.
\eeqs
We first claim that $B_2'$ is compact on $\LtG$; indeed, this follows since $H_kR:\LtG\to\LtG$ is bounded, $\chi(|\hsc D'|_g):\LtG\to H_k^{1}(\Gamma)$ is bounded (cf.~\eqref{eq:frequencycufoff}) and hence compact on $\LtG$, and $K_k'$ is compact on $\LtG$ by \cite[Theorem 1.2]{FaJoRi:78}. We next claim that $B_1'$ is invertible on $\LtG$; indeed, by Corollary \ref{cor:hBEM}, 
$(1 - \chi(|\hsc D'|_g)) H_k R\in \Psi^0_\hsc(\Gamma)$ with real-valued semiclassical principal symbol. Since $\eta \in \Rea\setminus \{0\}$, $B_1'$ is therefore elliptic and hence invertible for sufficiently large $k$ by Theorem \ref{thm:elliptic}. Thus $\Bregp$ is the sum of an invertible operator ($B_1'$) and a compact operator ($B_2'$) and so is Fredholm; the result for $\Breg$ follows either from very similar arguments (using the result in Corollary \ref{cor:hBEM} about 
$(1 - \chi(|\hsc D'|_g)) \Reg H_k$) or from the adjoint relation \eqref{eq:norms2}.

To complete the proof of Part (ii), we prove that $\Breg$ and $\Bregp$ are second-kind when $R=S_{\ri k}$ and $\Gamma$ is $C^1$; the proof for $R=S_0$ is very similar. Observe that
\beqs
\Breg= \left(\frac{\ri \eta}{2} - \frac{1}{4}\right)I + L_{k,\eta}
\quad
\tand
\quad
\Bregp= \left(\frac{\ri \eta}{2} - \frac{1}{4}\right)I + L'_{k,\eta}
\eeqs
where
\beqs
L_{k,\eta}:= -\ri \eta \DL_k +S_{\ri k}H_k + \frac{1}{4}I
\quad
\tand
\quad
L_{k,\eta}':= -\ri \eta \DL'_k +H_k S_{\ri k} + \frac{1}{4}I.
\eeqs
By the Calder\'on relations \eqref{eq:Calderon},
\beqs
L_{k,\eta}= -\ri \eta \DL_k + \big(S_{\ri k}-S_k\big)H_k + (\DL_k)^2,
\eeqs 
and
\beqs
L'_{k,\eta}= -\ri \eta \DL_k' + \big(H_k-H_{\ri k}\big)S_{\ri k} + (\DL_k')^2,
\eeqs 
 When $\Gamma$ is $C^1$, $\DL_k$ and $\DL_k'$ are compact on $\LtG$ by \cite[Theorem 1.2]{FaJoRi:78}. 
By this, and the mapping properties of $S_{\ri k}$ and $H_k$ from \eqref{eq:mapping}, to show that $L_{k,\eta}$ and $L_{k,\eta}'$ are compact it is sufficient to prove that
 (a) $ S_{k}-S_{\ri k}:H^{-1}(\Gamma)\rightarrow \LtG$ is compact, and 
(b) $ H_{k}-H_{\ri k}:H^{1}(\Gamma)\rightarrow \LtG$ is compact.
Since $\Phi_k - \Phi_{\ri k} = (\Phi_k - \Phi_0) - (\Phi_{\ri k }- \Phi_0)$, the bounds on $\Phi_k-\Phi_0$
 in \cite[Equation 2.25]{ChGrLaSp:12} (valid for $k\in \Com$) show that $S_{\ri k}-S_k: H^{-1}(\Gamma) \rightarrow H^1(\Gamma)$ and  
 $ H_{k}-H_{\ri k}:H^{1}(\Gamma)\rightarrow \HoG$. Since the inclusion $\HoG\rightarrow \LtG$ is compact
 (see, e.g., \cite[Theorem 3.27]{Mc:00}), both the properties (a) and (b) hold\footnote{In fact, when $\Gamma$ is $C^{1,1}$ (so that $H^2(\Gamma)$ is well-defined) these bounds on $\Phi_k-\Phi_0$ show that $S_{\ri k}-S_k:
 H^{-1}(\Gamma) \rightarrow H^2(\Gamma)$ and  $H_{\ri k}-H_k:
 H^{1}(\Gamma) \rightarrow H^2(\Gamma)$.};
see also \cite[Theorem 2.2]{BoDoLeTu:15} for a proof of these mapping properties using regularity results about transmission problems and standard trace results.
\epf

\subsection{Proof of Theorem \ref{thm:upper_bound_inverse}}

\begin{lemma}\label{lem:fav_formula_reg}
    \begin{align}\label{eq:fav_formula_reg}
        (\Breg)^{-1}= \NtD \Reg^{-1} - (I-\ri\eta \NtD \Reg^{-1})\ItDR
    \end{align}
    and
    \begin{align}\label{eq:fav_formula_reg2}
        (\Bregp)^{-1}= \Reg^{-1} \NtD -  \Reg^{-1}\ItDR( \Reg-\ri\eta\NtD).
    \end{align}
    \end{lemma}
    
    \begin{proof}[Proof of Lemma \ref{lem:fav_formula_reg}]
We first show that \eqref{eq:fav_formula_reg2} follows from \eqref{eq:fav_formula_reg}.
By Lemma \ref{lem:quasiadjoint}, 
\beq\label{eq:eye1}
(\Bregp)^{-1} = \big( (B_{k, \eta, \Reg'})')^{-1} = \big( (B_{k,\eta,\Reg'})^{-1}\big)'.
\eeq
By \eqref{eq:NtDdual2}, $(\NtD)' = \NtD$. By Part (ii) of Lemma \ref{lem:ItDS2}, $((\Reg')^{-1}\ItDSp)'=\Reg^{-1}\ItDS $ and thus 
\beq\label{eq:eye2}
(\ItDSp)' =\Reg^{-1} \ItDS \Reg.
\eeq
Replacing $\Reg$ by $\Reg'$ in  \eqref{eq:fav_formula_reg}, taking the $'$, and using \eqref{eq:eye1} and \eqref{eq:eye2}, the result \eqref{eq:fav_formula_reg2} follows.

We now prove \eqref{eq:fav_formula_reg}.
    Given $g$ and $\varphi$ satisfying  \(\Breg \varphi = g\), let  \(u :=\cK_k \varphi\); the motivation for this choice is that $\Breg$ is the direct BIE arising from Green's integral representation where, for the Neumann problem, $u$ is the sum of a double-layer potential and $u^I$; see \eqref{eq:Green}.
     Our goal is to express \(\varphi\) as a function of \(g\). 
The equation \(\Breg \varphi = g\) and the jump relations \eqref{eq:jumpproof} then imply that \(\Reg \partial^-_n u - \ri \eta  \gamma^- u = g\). By the definition of $\ItDR$, \(\gamma^- u = \ItDR g\). Then, using this last equation, \eqref{eq:jumpproof}, and the fact that $\partial_n^+ u = \partial_n^- u$, we find that
        \begin{align*}
            \varphi = \gamma^+ u - \gamma^- u = \NtD (\partial^+_n u) - \ItDR g
            &= \NtD \Reg^{-1}(g + \ri \eta \gamma^- u) - \ItDR g,  \\
            &= \Big(\NtD \Reg^{-1} - (I - \ri\eta \NtD \Reg^{-1} )\ItDR\Big)g,
        \end{align*}
and the result \eqref{eq:fav_formula_reg} follows.
    \end{proof}

\bpf[Proof of Theorem \ref{thm:upper_bound_inverse}]
We prove the upper bounds on  $\| (\Breg)^{-1}\|_{\LtGt}$.
The same argument proves upper bounds on $\| (B_{k,\eta,\Reg'})^{-1}\|_{\LtGt}$ with identical $k$-dependence, and the upper bounds on $\| (\Bregp)^{-1}\|_{\LtGt}$ then follows from Corollary \ref{cor:norms}. 

By \eqref{eq:fav_formula_reg} and the triangle inequality,
\begin{align*}
    \N{ (\Breg)^{-1}}_{L^2(\Gamma)\rightarrow L^2(\Gamma)}&\leq \N{ \NtD \Reg^{-1}}_{L^2(\Gamma)\rightarrow L^2(\Gamma)} \\ &\quad+ \big\| 
    \ItDS\big\|_{L^2(\Gamma)\rightarrow L^2(\Gamma)}\Big(1+\lvert \eta\rvert \N{ \NtD \Reg^{-1}}_{L^2(\Gamma)\rightarrow L^2(\Gamma)}\Big).
\end{align*}
Furthermore,
\begin{align*}
    \N{ \NtD \Reg^{-1}}_{L^2(\Gamma)\rightarrow L^2(\Gamma)} \leq \N{ \NtD }_{H^{-1}_k(\Gamma)\rightarrow L^2(\Gamma)}\N{   \Reg^{-1}}_{L^2(\Gamma)\rightarrow H_k^{-1}(\Gamma)}.
\end{align*}
Combining these last two inequalities, we obtain that
\begin{align*}
&    \lVert (\Breg)^{-1} \rVert_{L^2(\Gamma)\rightarrow L^2(\Gamma)}\leq \big\| \ItDS\big\|_{L^2(\Gamma)\rightarrow L^2(\Gamma)} \\
    &\hspace{4cm}+ (1+\lvert \eta \rvert \big\| \ItDS\big\|_{L^2(\Gamma)\rightarrow L^2(\Gamma)} )\N{ \NtD }_{H^{-1}_k(\Gamma)\rightarrow L^2(\Gamma)}\N{   \Reg^{-1}}_{L^2(\Gamma)\rightarrow H_k^{-1}(\Gamma)}.
\end{align*}
By Theorem~\ref{thm:ItDS} and Corollary~\ref{cor:Rbound},
\begin{align}\label{eq:hot1}
    \N{ (\Breg)^{-1}}_{L^2(\Gamma)\rightarrow L^2(\Gamma)} \leq C \Big( 1+ (1+\lvert \eta \rvert)k \N{ \NtD}_{H^{-1}_k(\Gamma)\rightarrow L^2(\Gamma)}\Big).
\end{align}
Theorem~\ref{thm:upper_bound_inverse} can then be obtained by using the upper bounds on \(\N{ \NtD}_{H^{-1}_k(\Gamma)\rightarrow L^2(\Gamma)}\) from Theorem~\ref{thm:NtD}.
\epf

\subsection{Proof of Theorem \ref{thm:lower_bound_inverse}} 

This proof follows the same ideas as the proof of the corresponding result for the BIEs \eqref{eq:Aketa} (used to solve the exterior Dirichlet problem) in \cite[Theorem 2.8]{BeChGrLaLi:11}; see also \cite[Pages 222 and 223]{ChGrLaSp:12}.
 
We prove the lower bound on  $\| (\Breg)^{-1}\|_{\LtGt}$;
the same argument proves the analogous lower bound on $\| (B_{k,\eta,\Reg'})^{-1}\|_{\LtGt}$, and then the lower bound on $\| (\Bregp)^{-1}\|_{\LtGt}$ follows from Corollary \ref{cor:norms} and the fact that $\|R\|_{\LtGt}= \|R'\|_{\LtGt}$.

Let $(u_j,k_j)_{m=1}^\infty$ be a quasimode with \(\supp u_j \subset \mathcal{K} \subset \Oe \), \(\N{ u_j}_{L^2(\Oe )}=1\), and 
$    \Delta u_j +k_j^2 u_j = g_j$,
where (by definition) 
\beq\label{eq:gbound}
\N{ g_j}_{L^2(\Oe )}\leq\QMC(k_j).
\eeq
 Let
\begin{align*}
    u_j^I (x):= \cR_k g_j (x)= \int_{\Oe } \Phi_{k_j} (x,y)g_j (y)\dif y.
\end{align*}
By standard properties of the free resolvent $\cR_k$ (see, e.g., \cite[Pages 197, 282]{Mc:00}),
\begin{align*}
    \Delta u_j^I +k_j^2 u_j^I =  g_j.
\end{align*}
We now think of \(u^I_j\) as an incident field and observe that, since $u_j^I$  and $u_j$ satisfy the Sommerfeld radiation condition and $\partial_n^+u_j=0$, $u_j$ is the corresponding solution of the scattering problem \eqref{eq:Helmholtz} and \eqref{eq:src}.
Green's integral representation theorem~\cite[Theorem 2.21]{ChGrLaSp:12} implies that
\begin{align*}
    -(\calS_{k_j} \partial_n u_j^S)(x) + (\cK_{k_j} \gamma u_j^S)(x) = 
    \left\{
        \begin{aligned}
            u_j^S & \quad \text{for } x\in \Oe ,\\
            0 & \quad \text{for } x\in \Oi ,
        \end{aligned}
    \right. \\
    -(\calS_{k_j} \partial_n u_j^I)(x) + (\cK_{k_j} \gamma u_j^I)(x) =
    \left\{
        \begin{aligned}
          0& \quad \text{for } x\in \Oe ,\\
            u_j^I & \quad \text{for } x\in \Oi .
        \end{aligned}
    \right.
\end{align*}
Adding the two equations in $\Oe$, and using the fact that $\partial_n^+ u_j=0$, we find that 
\beq\label{eq:uGreen}
u_j= u_j^I +\cK_{k_j} \gamma^+ u_j \quad\tin \Oe .
\eeq
Applying the operator \(\Reg \partial^+_n - \ri \eta \gamma^+\) and using the jump relations \eqref{eq:jumprelations}, we obtain that
\begin{align}\label{eq:fm}
    \Breg (\gamma^+u_j) = f_j, \quad\text{ where } \quad f_j:=-\big(\Reg\partial^+_n - \ri \eta \gamma^+\big)u^I.
\end{align}
Therefore, to prove Theorem~\ref{thm:lower_bound_inverse}, we only need to show that
\begin{align}\label{eq:quasimode_growth}
 \N{ \gamma^+ u_j }_{L^2(\Gamma)}\geq C    \N{ f_j }_{L^2(\Gamma)} k_j^{1/2}\Big(\lVert \Reg\rVert_{L^2(\Gamma)\rightarrow L^2(\Gamma)} k_j + \lvert\eta \rvert\Big)^{-1} \left(\frac{1}{\QMC(k_j)} - \frac{1}{k_j}\right).
\end{align}
By \eqref{eq:uGreen} and the definition of the quasimode,
\begin{align}\nonumber
    1= \N{ u_j }_{L^2(\Oe )}=\N{ u_j }_{L^2(\calK)}\leq  \N{\calD_{k_j} \gamma^+ u_j }_{L^2(\calK)}+  \N{ \cR g_j}_{L^2(\calK)} 
    \leq  & C\Big(\N{ \gamma^+ u_j }_{L^2(\Gamma)}+ \frac{1}{k_j} \N{ g_j}_{L^2(\calK)}\Big)\\
    \leq & C\Big(\N{ \gamma^+ u_j }_{L^2(\Gamma)}+ \frac{1}{k_j}\epsilon(k_j)\Big),\label{eq:quasimode_lower_bound_trace}
\end{align}
where we used Theorem \ref{thm:HT} to bound $\cK_k$, Theorem \ref{thm:Newton} to bound $\cR_k$, and the bound \eqref{eq:gbound} on $g_j$.

Having proved the bound \eqref{eq:quasimode_lower_bound_trace} on $\gamma^+ u_j$ from below, to prove \eqref{eq:quasimode_growth}, we now need an upper bound on $\|f_j\|_\LtG$.
Let $\chi \in C^\infty_{\rm comp}(\Rea^d)$ with $\chi = 1$ on a neighbourhood of $\Oi$. By the norm relation \eqref{eq:weightednorm}, 
the trace bound \eqref{e:basicTrace}, and Theorem \ref{thm:Newton},
\beqs
    \N{ \gamma^+ u^I_j}_{L^2(\Gamma)} \leq 
    \N{ \gamma^+ u^I_j}_{H^{1/2}_{k_j}(\Gamma)} \leq C k_j^{1/2} \N{\chi u^I_j}_{H^1_{k_j}(\Oe)} \leq C k_j^{-1/2} \N{ g_j }_{L^2(\Oe )}.
\eeqs
Similarly, 
\begin{align*}
    \N{ \partial^+_n u^I_j}_{L^2(\Gamma)}  \leq
        C \N{\gamma^+ \nabla (\chi u_j^I)}_{H^{1/2}_{k_j}(\Gamma)}\leq C k^{1/2} \N{\nabla (\chi u^I)}_{H^1_k(\Oe)} 
        &\leq C k^{3/2}     \N{\chi u^I}_{H^2_k(\Oe)}\\
&    \leq  C k_j^{1/2}\N{ g_j }_{L^2(\Oe )}.
\end{align*}
Using these last two displayed bounds in the definition of $f_j$ \eqref{eq:fm}, we find that 
\beqs
\N{ f_j }_{L^2(\Gamma)}\leq C k_j^{-1/2} \Big(k_j \lVert \Reg \rVert_{L^2(\Gamma)\rightarrow L^2(\Gamma)}  + \lvert\eta \rvert\Big) \QMC(k_j).
\eeqs
Combining this last inequality with \eqref{eq:quasimode_lower_bound_trace}, we obtain \eqref{eq:quasimode_growth} and the result follows.

\section{Numerical experiments illustrating the main results}\label{sec:num}	

\subsection{Obstacles considered}

\begin{figure}
    \begin{subfigure}[t]{0.45\textwidth}
        \includegraphics[width=\textwidth]{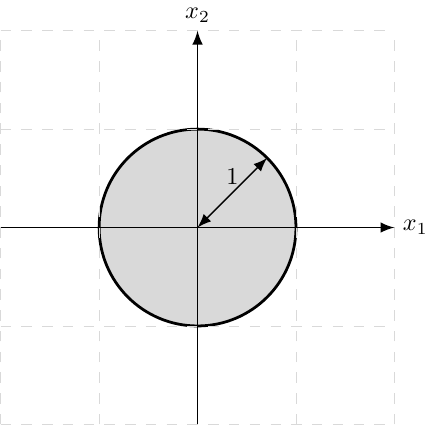}
        \caption{Circle}\label{fig:circle}
    \end{subfigure}
    \hfill
    \vspace{0.5cm}
    \noindent
    \begin{subfigure}[t]{0.45\textwidth}
        \includegraphics[width=\textwidth]{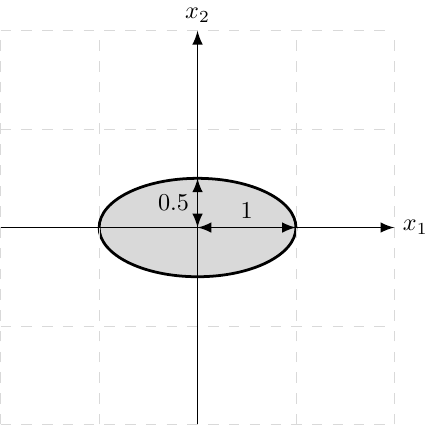}
        \caption{Ellipse}\label{fig:ellipse}
    \end{subfigure}
    \begin{subfigure}[t]{0.45\textwidth}
        \includegraphics[width=\textwidth]{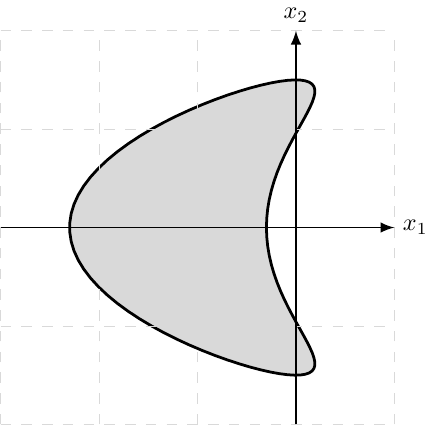}
        \caption{Kite}\label{fig:kite}
    \end{subfigure}
    \hfill
    \vspace{0.5cm}
    \noindent
    \begin{subfigure}[t]{0.45\textwidth}
        \includegraphics[width=\textwidth]{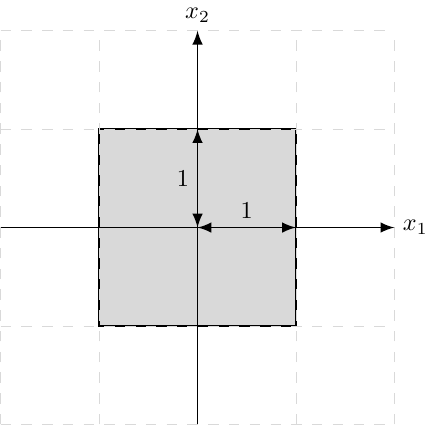}
        \caption{Square}\label{fig:square}
    \end{subfigure}
    \begin{subfigure}[t]{0.45\textwidth}
        \includegraphics[width=\textwidth]{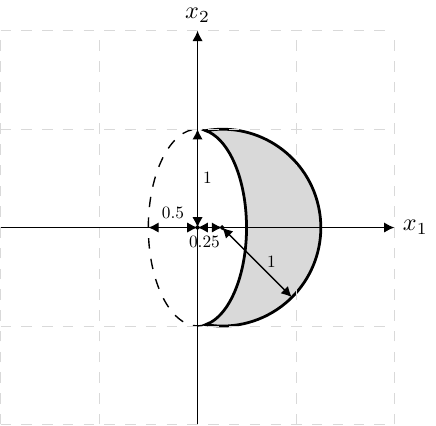}
        \caption{Moon}\label{fig:moon}
    \end{subfigure}
    \hfill
    \vspace{0.5cm}
    \noindent
    \begin{subfigure}[t]{0.45\textwidth}
        \includegraphics[width=\textwidth]{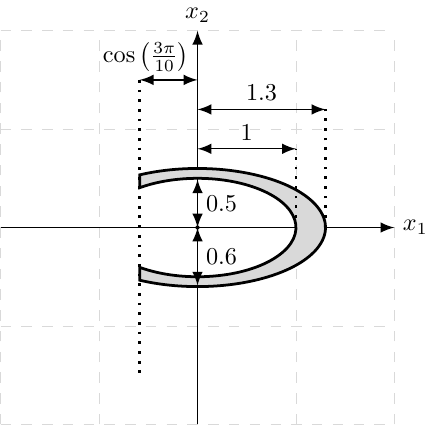}
        \caption{Elliptic cavity}\label{fig:elliptic_cavity}
    \end{subfigure}
    \caption{Obstacles considered in the numerical experiments}\label{fig:geometries}
\end{figure}

We consider the following obstacles $\Oi$, shown in Figure \ref{fig:geometries}, and inspired by those considered in the experiments in \cite{BeChGrLaLi:11}.
\bit
\item The unit circle and an ellipse whose minor and major axis are respectively 0.5 and 1; these are both examples where $\Gamma$ is $C^\infty$ and curved (in the sense of Definition \ref{def:curved}).
\item The ``kite'' domain defined by \((\cos (t)-0.65 \cos (2t)-0.65,1.5\sin(t))\) with \(t\in [0,2\pi]\); this $\Gamma$ is smooth.
\item A square with side length \(2\); this $\Gamma$ is piecewise smooth (in the sense of Definition \ref{def:piecewisesmooth}).
\item The ``moon'' domain defined as the union of an elliptic arc and a circular arc, where the particular ellipse is \((0.5 \cos (t),\sin(t))\) with \(t\in [-\pi/2,\pi/2]\) and the particular circle is \((\cos (t)+0.25,\sin(t))\) with \(t\in [-\pi/2,\pi/2]\); this $\Gamma$ is both piecewise smooth and piecewise curved  (in the sense of Definition \ref{def:piecewisecurved}).
\item The ``elliptic cavity'' defined as the region between the two elliptic arcs
\begin{align*}
    &(\cos (t), 0.5 \sin (t)), \quad t\in [-\phi_0,\phi_0]
    \quad\tand \quad 
    (1.3\cos (t), 0.6 \sin (t)),\quad t\in [-\phi_1,\phi_1] \\
    &\qquad\qquad\text{ with } \phi_0=7\pi/10\quad\tand \quad \phi_1= \arccos \left(\frac{1}{1.3}\cos (\phi_0)\right);
\end{align*}
this corresponds to the shared interior of the solid lines in Figure~\ref{fig:elliptic_cavity}. 
\eit 

All these $\Oi$ are nontrapping (in the sense of Definition \ref{def:nontrapping}), apart from the elliptic cavity, which is trapping. The elliptic cavity also satisfies the assumptions of Part (ii) of Theorem \ref{thm:ellipse}, and so there exists a quasimode with exponentially-small quality.

When considering $R=S_0$, we choose the constant $a$ in the Laplace fundamental solution \eqref{eq:LaplaceFund} to be $4$. Since the maximal diameter of the considered $\Oi$ is $\leq 3$ and the capacity of $\Oi$ is $\leq {\rm diam}(\Oi)$ (see \cite[Exercise 8.12]{Mc:00}), this choice of $a$ ensures that $S_0$ is coercive and that $(B_{k,\eta,S_0})^{-1}$ exists when $\Gamma$ is $C^1$ by Part (ii) of Theorem \ref{thm:invert}.

For all nontrapping domains, we compute norms of $\Breg$ and $(\Breg)^{-1}$ for \(k=5\times 2^n\) with \(n=0,1,\dots,8\), 
i.e., $k\in (5,1280)$. For the elliptic cavity, we compute at \(k=5\times 2^n\) with \(n=0,1,\dots,7\), 
i.e., $k\in (5,640)$, but we also compute at (approximations of) particular  frequencies in the quasimode. The particular frequencies are denoted \(k_{m,0}^e\), with this notation explained in the following remark.

\begin{remark}[The quasimode frequencies \(k_{m,0}^e\)]\label{rk:freq_quasimode_computations}
        The functions \(u_j\) in the Neumann quasimode construction in
Part (ii) of Theorem~\ref{eq:ellipse} (from \cite[Theorem 3.1]{NgGr:13} and analogous to the Dirichlet quasimode construction in \cite{BeChGrLaLi:11})
        are based on the family of eigenfunctions of the Laplacian operator in the ellipse \(E\) \eqref{eq:ellipse} localising around the periodic orbit \(\{(0,x_2) : \lvert x_2 \rvert \leq a_2\}\), i.e., the minor axis of the ellipse; the square root of the associated eigenvalues defines frequencies in the quasimode.
        We use the method introduced in~\cite{Wi:06} and the associated MATLAB toolbox to compute the eigenvalues of the ellipse, and hence the frequencies in the quasimode. 
In this paper we consider  the frequencies  \(k_{m,0}^e\); the superscript `e' is because the associated eigenfunctions are even functions of the ``angular'' variable, the subscript `$m,0$' means that the associated eigenfunction has no zeros when the angular variable is in $[0,\pi)$ and $m$ zeros when the radial variable, $\mu$, is in $(0,\mu_0)$, where $\mu_0:= \cosh^{-1}(a_1/\sqrt{a_1^2-a_2^2})$ and the boundary of the ellipse is  $\mu=\mu_0$; see~\cite[Appendix E]{MaGaSpSp:21} for more details.
        \end{remark}

\subsection{Description of the discretisation used for the experiments}\label{sec:82}

We consider $\Breg$ \eqref{eq:direct} with $R=S_{\ri k}$ and $R=S_0$.
These operators are discretised using the boundary-element method (BEM) with continuous piecewise-linear functions. We choose \(\eta=1/2\) as in~\cite[Equation 23]{BrElTu:12}, and we define the mesh using ten points per wavelength. In more detail: given a finite-dimensional subspace $V_n\subset \LtG$, the Galerkin method is
\beq\label{eq:Galerkin}
\text{ find } v_n \in V_n \tst \big(\Breg v_n, w_n\big)_\LtG = \big(f, w_n\big)_\LtG \quad\tfa w_n \in V_n,
\eeq
where $f$ denotes the right-hand side of the BIE in \eqref{eq:direct}; the Galerkin solution $v_n$ is then an approximation to $\gamma^+  u$. We denote the continuous piecewise-linear basis functions by \(\phi_j \in V_n\) for \(j=1,\ldots,n\). The matrix of the Galerkin linear system \eqref{eq:Galerkin} can be written
\begin{align*}
    \mathbf{B}_{k,\eta,R}:= \ri \eta \left(\frac{1}{2}\mathbf{M} - \mathbf{K}_k\right) + \mathbf{R}\mathbf{M}^{-1} \mathbf{H}_k,
\end{align*}
where \((\mathbf{B}_{k,\eta, R})_{j,k}=\big(\Breg \phi_k, \phi_j\big)_\LtG\); the matrices arising from the operators \(K_k\), \(R_k\) and \(H_k\) are defined similarly, and the mass matrix \(\mathbf{M}\) is the discretisation of the \(L^2\) scalar product on \(V_n\).
The meshwidth $h$ was chosen so that $10 h = 2\pi/k$; this corresponds to having ten gridpoints per wavelength, which, at least empirically, ensures the accuracy does not deteriorate as $k\tendi$ (but see \cite{GaSp:22} for more discussion on this). 
To check the accuracy of this choice, we re-ran all the experiments for twenty gridpoints per wavelength, and this resulted in essentially no visible changes in the plots of the norms.
This choice of $h$ means that $n\sim h^{-(d-1)} \sim k^{d-1}$.

Approximations to the norms of $\Breg$ and $(\Breg)^{-1}$ are computed as the maximal singular value of $\mathbf{M}^{-1}\mathbf{B}_{k,\eta,R}$ and the inverse of the minimum singular value of $\mathbf{M}^{-1}\mathbf{B}_{k,\eta,R}$, respectively. As $h\tendo$ for fixed $k$, we expect these approximations to converge by the following lemma combined with (a) the fact that $\cond(\bfM)$ is bounded independently of $h$ for standard BEM spaces (see \cite[Theorem 4.4.7 and Remark 4.5.3]{SaSc:11} and \cite[Corollary 10.6]{St:08}), and (b) the fact that $\Breg$ is a compact perturbation of a multiple of the identity when $\Gamma$ is $C^1$ by Part (ii) of Theorem \ref{thm:invert}.

\ble\mythmname{\cite[Lemma B.1]{MaGaSpSp:21}}
\label{lem:discon}
Let $V_n\subset \LtG$ be a finite-dimensional space with real basis $\{\phi_j\}_{j=1}^n$. Given $A:\LtGt$, let $\bfA$ be defined by \((\bfA)_{j,k}=(A\phi_k,\phi_j)_{L^2(\Gamma)}\).
Let $P_h:\LtG\rightarrow V_h$ be the orthogonal projection, and let 
\beqs
\widetilde{A}:= P_h A|_{V_h}.
\eeqs

(i) 
\beqs
\N{\bfM^{-1}\bfA}_2 \leq \sqrt{ \cond(\bfM)}\,\, \big\|\widetilde{A}\big\|_{\LtGt}
\eeqs
where $\cond(\bfM):= \|\bfM\|_2 \|\bfM^{-1}\|_2$, and if $(\bfM^{-1}\bfA)^{-1}$ exists, then
\beqs
\N{(\bfM^{-1}\bfA)^{-1}}_2 \leq \sqrt{ \cond(\bfM)}\,\, \big\|\widetilde{A}^{-1}\big\|_{\LtGt}.
\eeqs

(ii) If $P_h \phi \rightarrow \phi$ as $h\tendo$ for all $\phi\in\LtG$, then
\beqs
\big\|\widetilde{A}\big\|_{\LtGt} \rightarrow \N{A}_{\LtGt} \quad\tas h\tendo;
\eeqs
if, in addition, $A = a I + K$, where $a\neq 0$ and $K$ is compact, then 
\beqs
\big\|\widetilde{A}^{-1}\big\|_{\LtGt} \rightarrow \big\|A^{-1}\big\|_{\LtGt} \quad\tas h\tendo.
\eeqs
\ele

The numerical experiments were conducted with the software FreeFEM~\cite{Hecht2012} using
\begin{itemize}
    \item the  interface of FreeFEM with BemTool\footnote{\url{https://github.com/xclaeys/BemTool}} and HTool\footnote{\url{https://github.com/htool-ddm/htool}} to generate the dense matrices stemming from the BEM discretisation of the considered operators, and
    \item the  interface of FreeFEM  with PETSc~\cite{petsc-efficient,petsc-user-ref} and SLEPc~\cite{slepc-toms,slepc-manual} to solve singular value problems; in particular, we used ScaLAPACK~\cite{Blackford1997} to obtain the largest and smallest singular values of $\mathbf{M}^{-1}\mathbf{B}_{k,\eta,R}$  and in Figure \ref{fig:5} (for the kite obstacle) we used the \texttt{cross} method to compute the largest singular values of the Galerkin matrices of $D_k, S_0 H_{ k}$, and $S_{\ri k}H_{k}$.
\end{itemize}

\subsection{Numerical results}

With $R=S_{\ri k}$ and $R=S_0$,
the maximum singular value of $\mathbf{M}^{-1}\mathbf{B}_{k,\eta,R}$ and the inverse of the minimum singular value of 
$\mathbf{M}^{-1}\mathbf{B}_{k,\eta,R}$ (which equals the maximum singular value of $(\mathbf{M}^{-1}\mathbf{B}_{k,\eta,R})^{-1}$) are plotted in Figure \ref{fig:1} (circle and ellipse), Figure \ref{fig:2} (moon), Figure \ref{fig:3} (kite and square), and Figure \ref{fig:4} (elliptic cavity).
In the captions of the figures we abuse notation and write ``$\sigma_{\max}$ for $B_{k,\eta,S_{\ri k}}$'' and 
``$1/\sigma_{\min}$ for $B_{k,\eta,S_{\ri k}}$''.
The computed growth rates with $k$ are summarised in Tables \ref{table:sum_results_curved}, \ref{table:sum_results_piecewise_curved}, and \ref{table:sum_results_piecewise_smooth}, and compared with those in the bounds in \S\ref{sec:main_results}. 

We now discuss separately (i) the norms for all $\Oi$ other than the elliptic cavity, (ii) the norms of the inverses for all $\Oi$ other than the elliptic cavity, and (iii) the norms and the norms of the inverses for the elliptic cavity.

\begin{table}[h!]
    \centering
    \begin{tabular}{ccccccc} 
        \toprule
        & \multicolumn{2}{c}{Circle} &\phantom{a}& \multicolumn{2}{c}{Ellipse}  \\
        \cmidrule{2-3}\cmidrule{5-6}
        & Observed & Proved && Observed & Proved  \\ 
        \midrule
        \(\lVert \BregSik \rVert \) & \(\sim 1\) & \(\lesssim 1\) && \(\sim 1\) & \(\lesssim 1\)\\
        \(\lVert (\BregSik)^{-1} \rVert \) &\(\sim k^{0.34}\) & \(\lesssim k^{1/3}\) &&  \(\sim k^{0.28}\) & \(\lesssim k^{1/3}\)\\
        \(\lVert \BregSO \rVert \) & \(\sim k^{0.94}\) & \(\lesssim k\)  && \(\sim k^{0.99}\) & \(\lesssim k\)\\
        \bottomrule
    \end{tabular}
    \caption{Comparison of the $k$-dependence of the computed norms for the circle 
   (Figure~\ref{fig:circle}) and the ellipse (Figure~\ref{fig:ellipse}) (column ``Observed'')
   and the bounds in Section~\ref{sec:main_results} (column  ``Expected'') 
   }
    \label{table:sum_results_curved}
\end{table}
\begin{table}[h!]
    \centering
    \begin{tabular}{ccccc} 
        \toprule
        & \multicolumn{2}{c}{Moon} &\phantom{a}  \\
        \cmidrule{2-3}
        & Observed & Proved  \\ 
        \midrule
        \(\lVert \BregSik \rVert \) & \(\sim k^{0.15}\) & \(\lesssim k^{1/6}\log(k+2)\)  \\
        \(\lVert (\BregSik)^{-1} \rVert \) & \(\sim k^{0.41}\) & \text{ None } \\
        \(\lVert \BregSO \rVert \) & \(\sim k^{1.00}\) & \(\lesssim k \log(k+2)\)\\
        \bottomrule
    \end{tabular}
    \caption{Comparison of the $k$-dependence of the computed norms for the moon obstacle in 
   Figure~\ref{fig:moon} (column ``Observed'')
   and the bounds in Section~\ref{sec:main_results} (column  ``Expected'').
(``None'' appears in the ``Proved'' column for $(\BregSik)^{-1}$, since Theorem  \ref{thm:upper_bound_inverse} does not apply to the moon obstacle since $\Gamma$ is not $C^\infty$.)    }
    \label{table:sum_results_piecewise_curved}
\end{table}
\begin{table}[h!]
    \centering
    \begin{tabular}{ccccccc} 
        \toprule
        & \multicolumn{2}{c}{Kite} &\phantom{a}& \multicolumn{2}{c}{Square}  \\
        \cmidrule{2-3}\cmidrule{5-6}
        & Observed & Proved && Observed & Proved  \\ 
        \midrule
        \(\lVert \BregSik \rVert \) & \(\sim k^{0.21}\) & \(\lesssim k^{1/4}\log(k+2)\) && \(\sim k^{0.16}\) & \(\lesssim k^{1/4}\log(k+2)\) \\
        \(\lVert (\BregSik)^{-1} \rVert \) & \(\sim k^{0.41}\) & \(\lesssim k^{2/3}\) && \(\sim k^{0.13}\) & \text{ None} \\
        \(\lVert \BregSO \rVert \) & \(\sim k^{1.00}\) & \(\lesssim k \log(k+2)\) && \(\sim k^{0.98}\) & \(\lesssim k \log(k+2)\)\\
        \bottomrule
    \end{tabular}
        \caption{Comparison of the $k$-dependence of the computed norms for the kite 
   (Figure~\ref{fig:kite}) and the square (Figure~\ref{fig:square}) (column ``Observed'')
   and the bounds in Section~\ref{sec:main_results} (column  ``Proved'').
(``None'' appears in the ``Proved'' column for $(\BregSik)^{-1}$, since Theorem  \ref{thm:upper_bound_inverse} does not apply to the square since $\Gamma$ is not $C^\infty$.)
}
    \label{table:sum_results_piecewise_smooth}
\end{table}

\paragraph{Discussion of the norms of $B_{k,\eta,S_0}$ and $B_{k,\eta,S_{\ri k}}$}

The computed norms of $B_{k,\eta,S_0}$ and $B_{k,\eta,S_{\ri k}}$ agree well with the theory for all obstacles apart from the square and the elliptic cavity, where the norms of $B_{k,\eta,S_{\ri k}}$ grow slightly slower than expected. The explanation of this discrepancy for the elliptic cavity is given below while for the square it appears that we have not computed large enough frequencies to reach the asymptotic regime (we highlight that, as mentioned in \S\ref{sec:82}, the plots remain essentially unchanged when re-computed with twenty points per wavelength as opposed to ten).

Figure \ref{fig:5} plots, for the kite obstacle, the norms of $B_{k,\eta,S_0}$ and $B_{k,\eta,S_{\ri k}}$ and the norms of their component parts, i.e., $S_0 H_{k}$ and $\DL_k$ for $B_{k,\eta,S_0}$ and $S_{\ri k} H_k$ and $\DL_k$ for $B_{k,\eta,S_{\ri k}}$. For $B_{k,\eta,S_0}$ we see that the $\|S_0 H_k\|_{\LtGt}$ grows like $k$ (as expected from Lemma \ref{lem:S0}) and dominates $\|\DL_k\|_{\LtGt}$, which grows just slightly slower than the $k^{1/4}$ predicted by Theorem \ref{th:K_k_bounds} (in this discussion we ignore the $\log k$ terms in the bounds, since these are essentially impossible to see numerically).
For $B_{k,\eta,S_{\ri k}}$ we see that $\|S_{\ri k} H_k\|_{\LtGt}$ is bounded independently of $k$ (as expected from Theorems \ref{thm:Hk} and \ref{thm:newSikbound}) and $B_{k,\eta,S_{\ri k}}$ is determined by $\|\DL_k\|_{\LtGt}$.

\paragraph{Discussion of the norms of the inverses of $B_{k,\eta,S_0}$ and $B_{k,\eta,S_{\ri k}}$}

Since $S_0$ does not satisfy Assumption \ref{ass:Reg}, this paper does not prove any bounds on $\|(B_{k,\eta,S_0})^{-1}\|_{\LtGt}$. However, for all the considered $\Oi$, the norm of $\|(B_{k,\eta,S_0})^{-1}\|_{\LtGt}$ grows with $k$ at approximately at the same rate as $\|(B_{k,\eta,S_{\ri k}})^{-1}\|_{\LtGt}$.

For the curved domains (i.e., the circle and ellipse), the experiments show $\|(B_{k,\eta,S_{\ri k}})^{-1}\|_{\LtGt}$ growing approximately like $k^{1/3}$, exactly as in the upper bound \eqref{eq:ballinverse}.
The upper bound \eqref{eq:smoothinversebound} for general smooth nontrapping domains shows that $\|(B_{k,\eta,S_{\ri k}})^{-1}\|_{\LtGt}$ grows at most like $k^{2/3}$, but the largest growth observed is $k^{0.41}$ for both the moon and the kite.

\paragraph{Discussion of the experiments for the elliptic cavity.}

The left-hand plot of Figure \ref{fig:4} shows $\|B_{k,\eta,S_0}\|_{\LtGt}$ growing like $k$, which is as expected from the discussion above. 
The left-hand plot also shows $\|B_{k, \eta, S_{\ri k}}\|_{\LtGt}$ being essentially constant for the range of $k$ considered, although,
 at least in 2-d, $\|B_{k, {\rm reg}}\|_{\LtGt} \gtrsim k^{1/4}$ for large enough $k$ (i.e., we have not computed large enough frequencies to reach the asymptotic regime). Indeed, \cite[Theorem 4.6]{ChGrLaLi:09} shows that $\|D_k'\|_{\LtGt}\gtrsim k^{1/4}$ for a certain class of 2-d domains (to see that the elliptic cavity falls in this class, take the points $x_1$ and $x_2$ in the statement of \cite[Theorem 4.6]{ChGrLaLi:09} to lie on one of the flat ends of the cavity, with $x_2$ in the middle of this end, and $x_1$ at one of the corners)
and 
Theorems \ref{thm:Hk} and \ref{thm:newSikbound} show that $\|S_{\ri k }H_k\|_{\LtGt}\lesssim (\log k)^{3/2}$. 

Although Theorems \ref{thm:lower_bound_inverse} and \ref{thm:ellipse} predict exponentially growth of $\|(B_{k,\eta,S_0})^{-1}\|_{\LtGt}$ and 
$\|(B_{k,\eta,S_{\ri k}})^{-1}\|_{\LtGt}$ through $k=k^{e}_{m,0}$, we do not see this in the right-hand plot of Figure \ref{fig:4}. This feature is partially explained by the bound \eqref{eq:inverse_bound_formost}; indeed, this bound shows that $\|(B_{k,\eta,S_{\ri k}})^{-1}\|_{\LtGt}$ is bounded polynomially in $k$ for all $k$ except an arbitrarily-small set, demonstrating that the growth of $\|(B_{k,\eta,S_{\ri k}})^{-1}\|_{\LtGt}$ is very sensitive to the precise value of $k$. This result indicates that the exponential growth of  $\|(B_{k,\eta,S_{\ri k}})^{-1}\|_{\LtGt}$ through $k=k^{e}_{m,0}$ is not captured in Figure \ref{fig:4} due to discretisation error; see \cite{MaGaSpSp:21} for further discussion and results on this feature.

\begin{figure}[h!]
    \begin{subfigure}[t]{0.9\textwidth}
        \includegraphics[width=\textwidth]{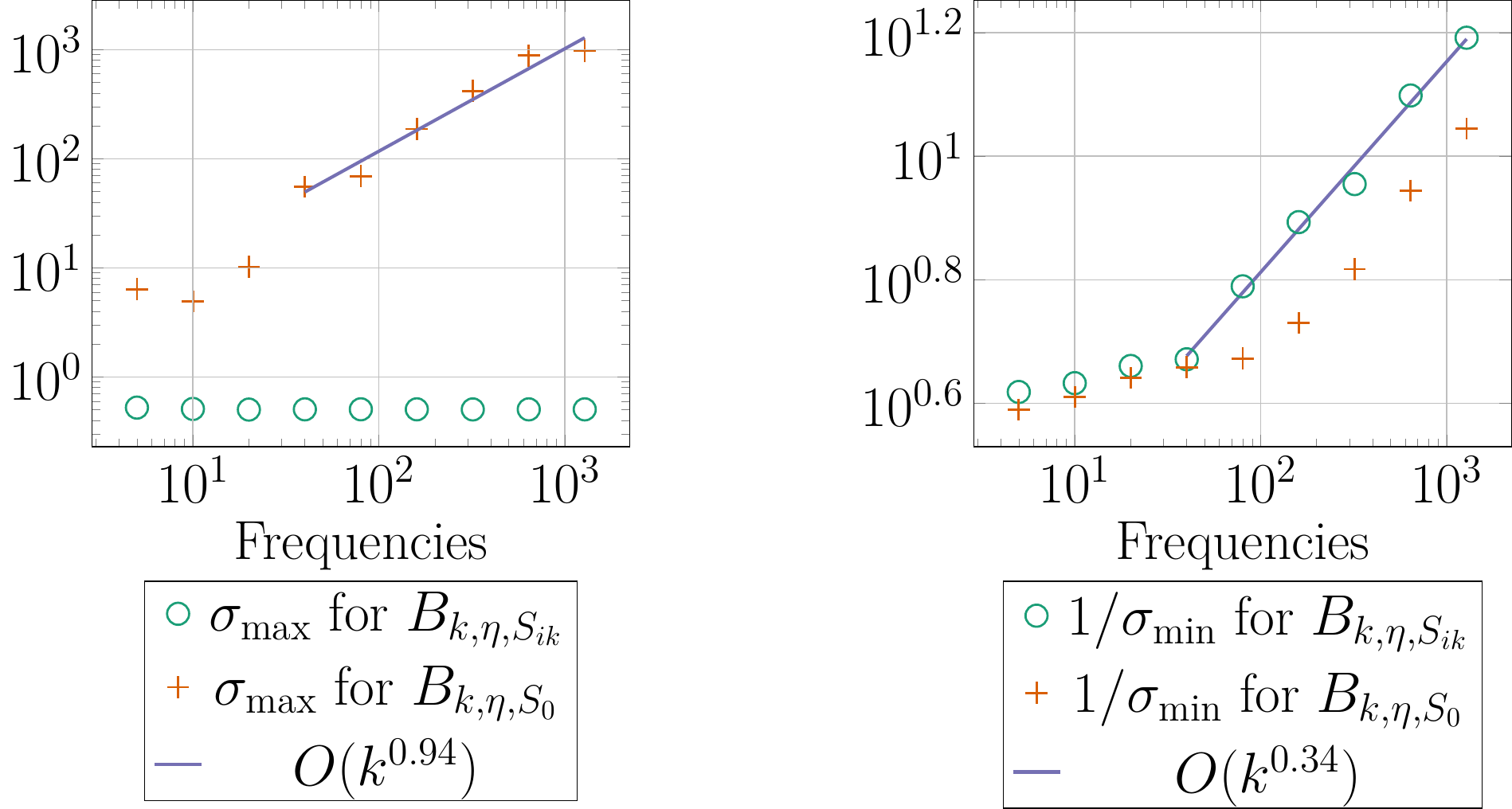}
        \caption{Circle}
    \end{subfigure}
    \hfill
    \vspace{0.5cm}
    \noindent
    \begin{subfigure}[t]{0.9\textwidth}
        \includegraphics[width=\textwidth]{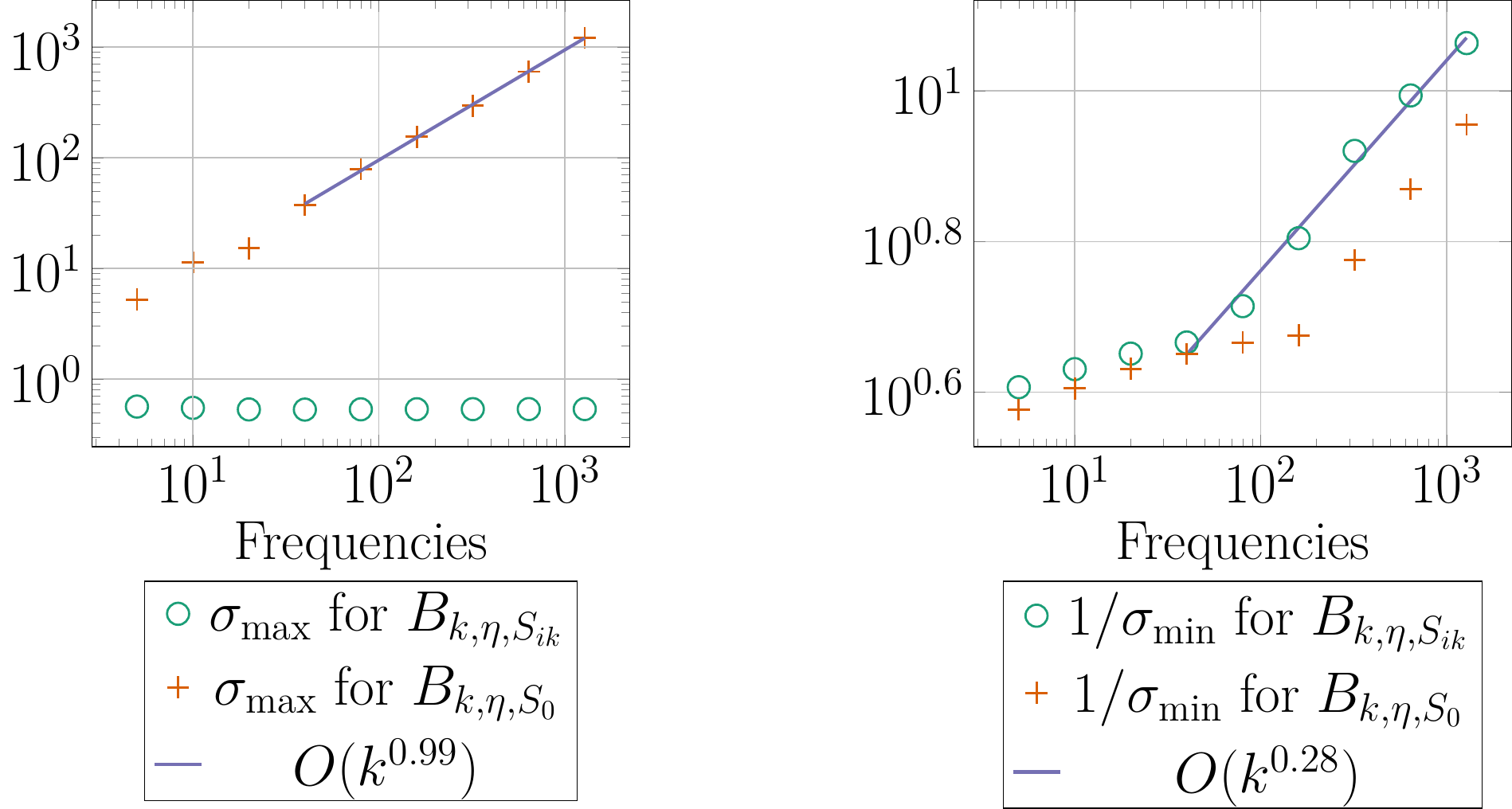}
        \caption{Ellipse}
    \end{subfigure}
    \caption{The computed norms for the circle and ellipse}\label{fig:1}
\end{figure}

\begin{figure}[h!]
    \includegraphics[width=\textwidth]{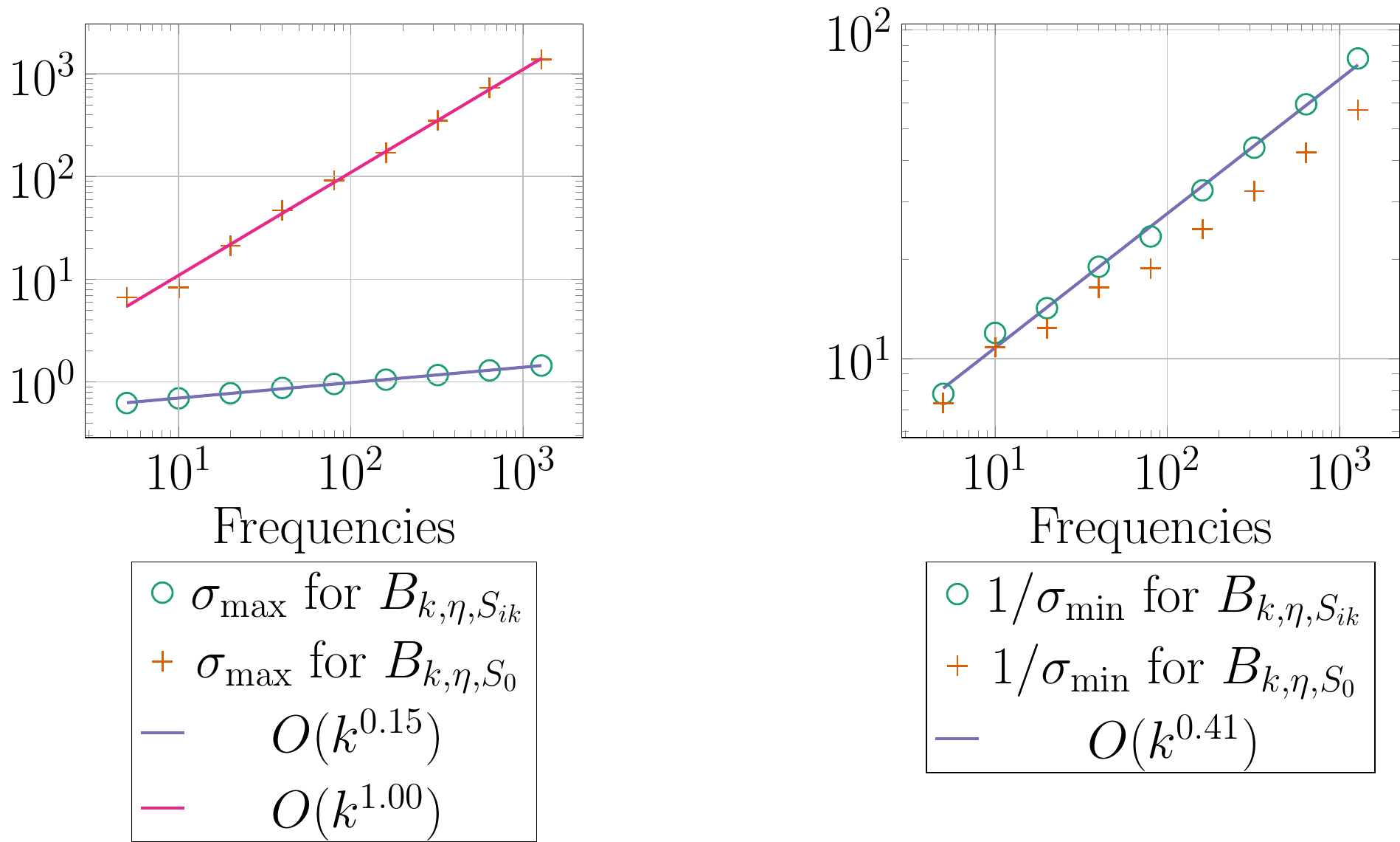}
    \caption{The computed norms for the moon obstacle}
    \label{fig:2}
\end{figure}

\begin{figure}[h!]
    \begin{subfigure}[t]{0.9\textwidth}
        \includegraphics[width=\textwidth]{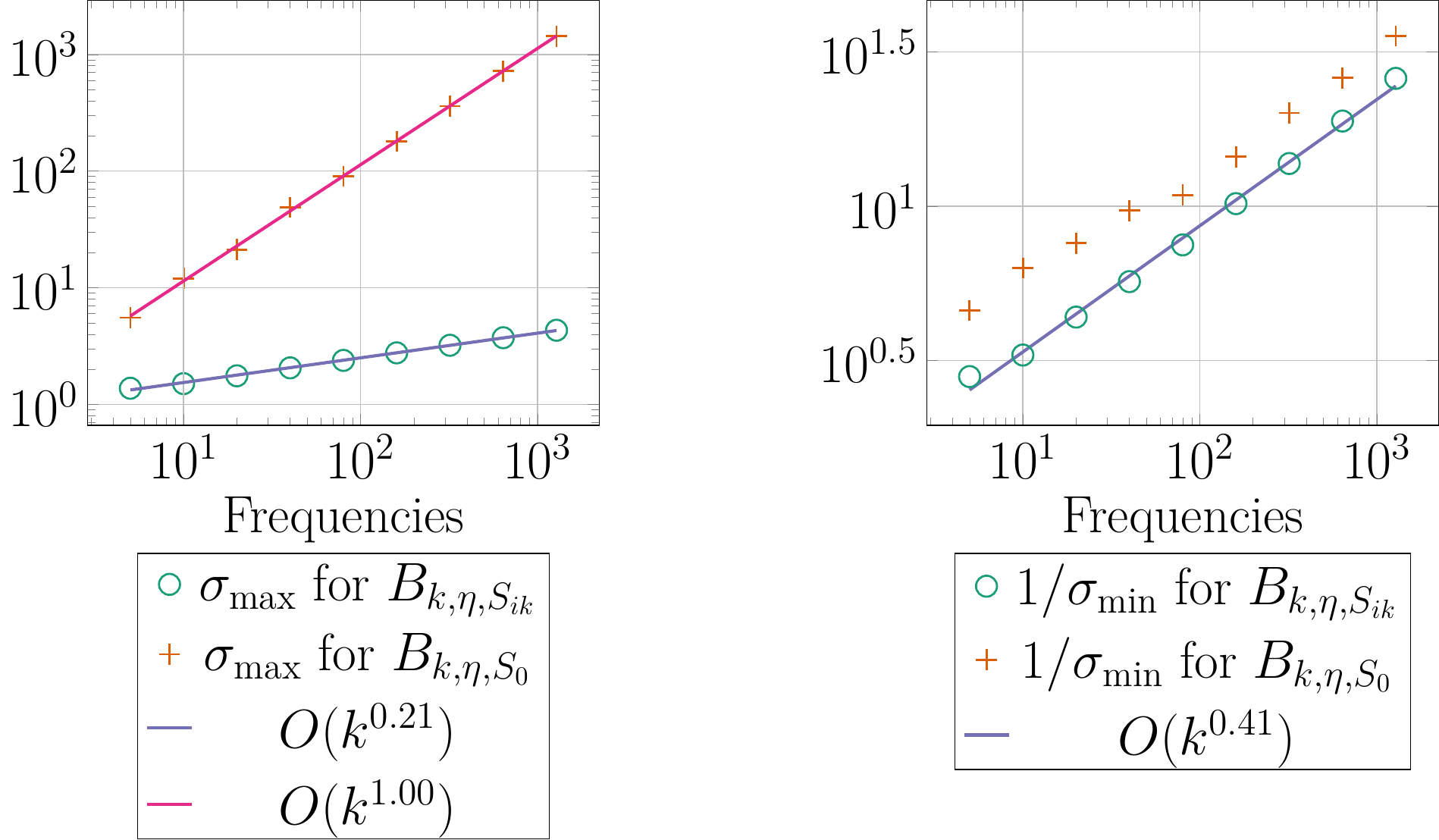}
        \caption{Kite}
    \end{subfigure}
    \hfill
    \vspace{0.5cm}
    \noindent
    \begin{subfigure}[t]{0.9\textwidth}
        \includegraphics[width=\textwidth]{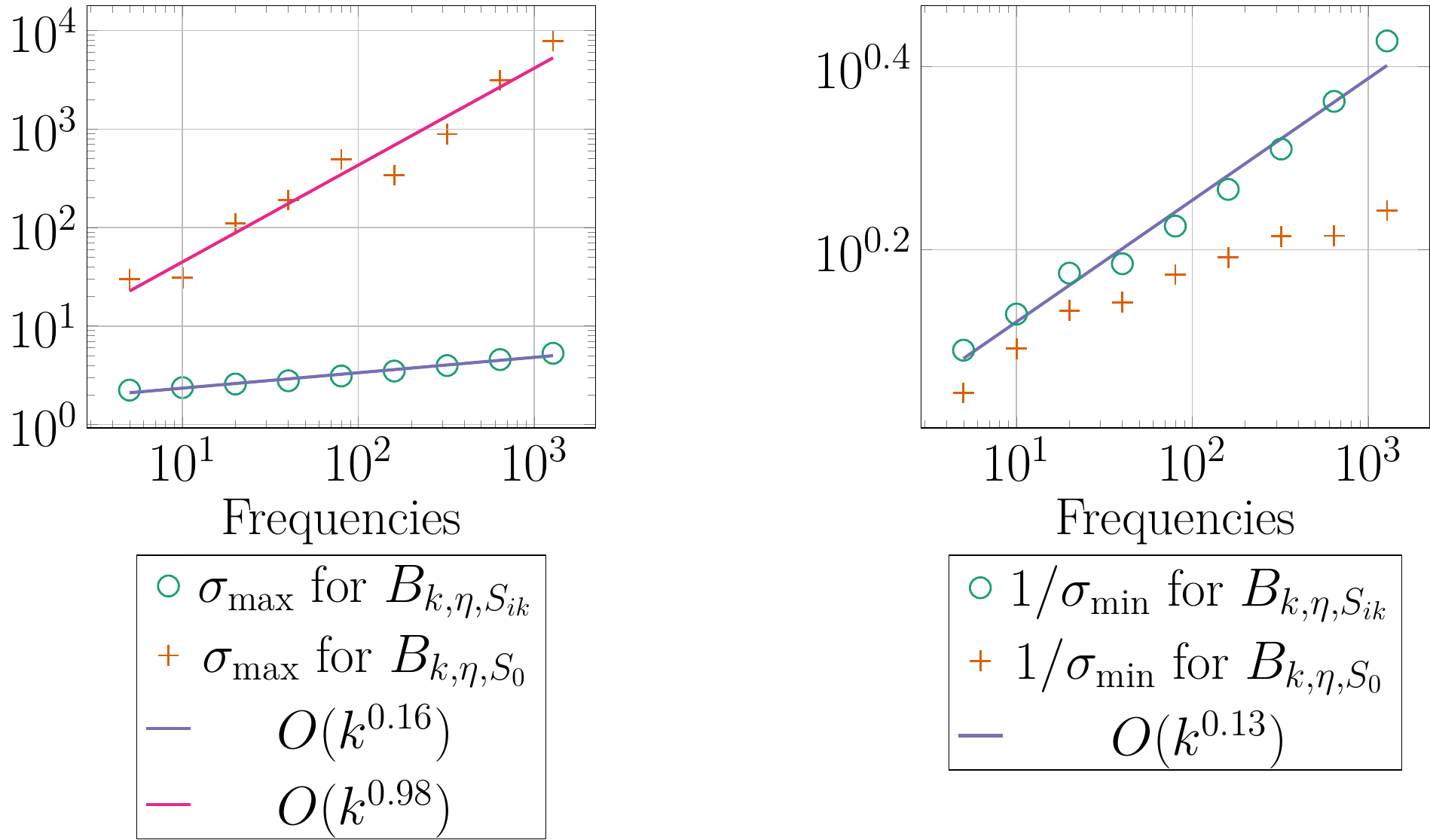}
\caption{Square}
    \end{subfigure}
        \caption{The computed norms for the kite and the square}
    \label{fig:3}
\end{figure}

\begin{figure}[h!]
    \centering
    \includegraphics[width=\textwidth]{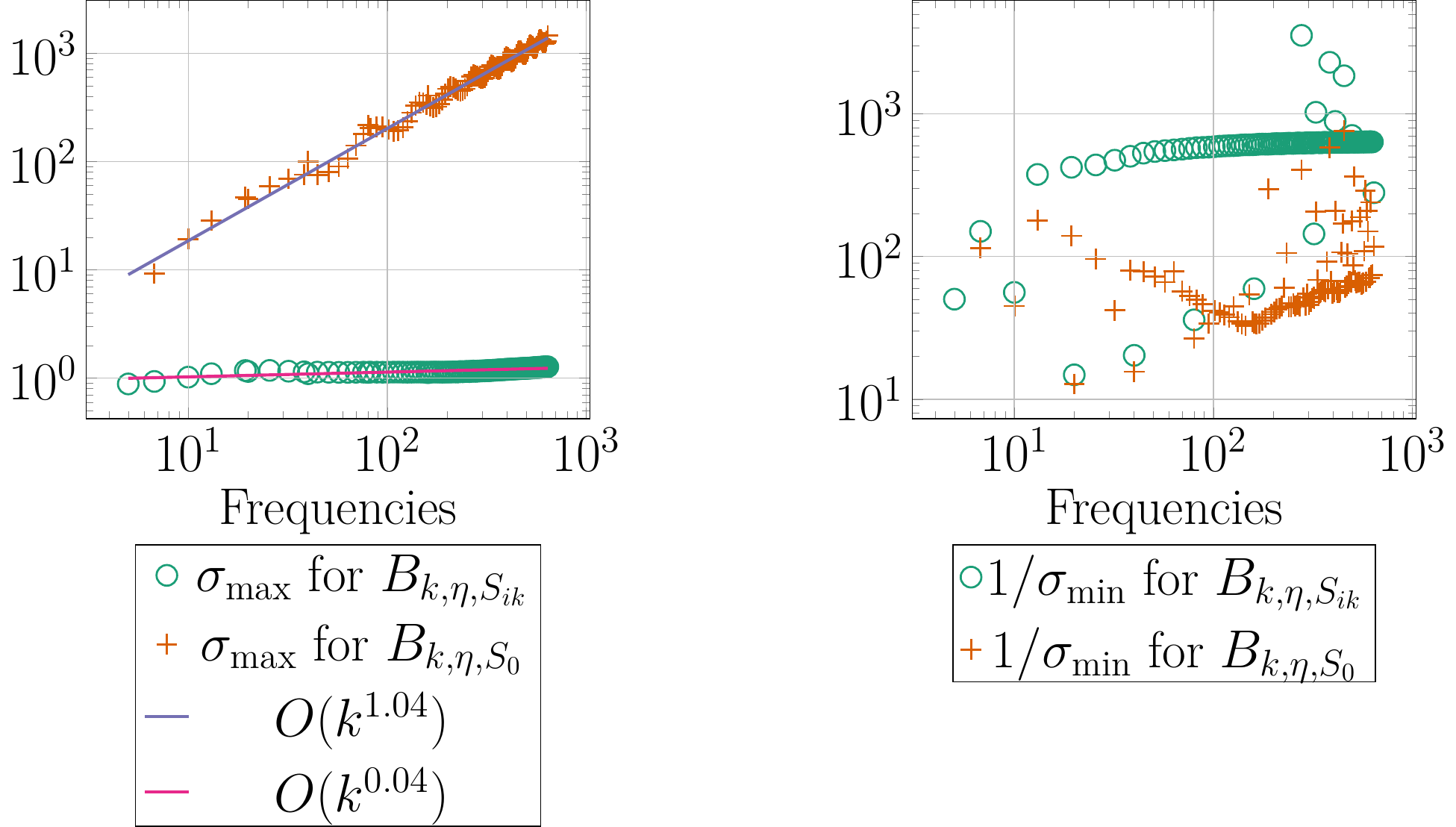}
    \caption{The computed norms for the elliptic cavity}\label{fig:4}
\end{figure}

\begin{figure}[h!]
    \centering
    \includegraphics[width=\textwidth]{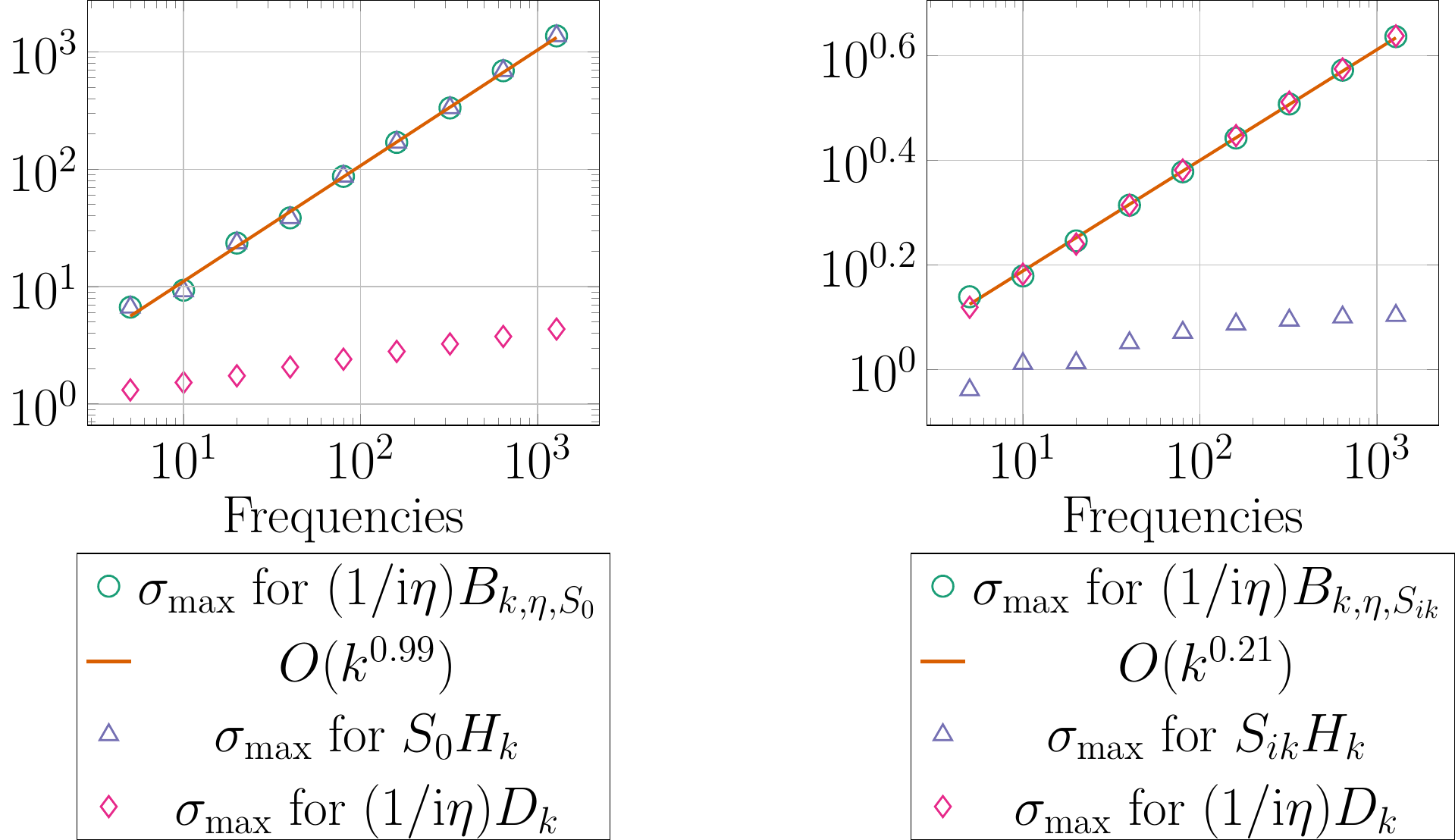}
    \caption{The norms of $B_{k,\eta,S_0}$ and $B_{k,\eta,S_{\ri k}}$, and the norms of their component parts, for the kite obstacle (note the different scales on the $y$-axes).}\label{fig:5}
\end{figure}

\section{The choice of $\eta$: heuristic discussion and numerical experiments}\label{sec:eta}

The bounds on $\|(\Breg)^{-1}\|_{\LtGt}$ in 
Theorem \ref{thm:upper_bound_inverse} are proved under the assumption that $\eta$ is independent of $k$; the reason for this is that we only have upper bounds on $\|\ItDS\|_{\LtGt}$ for this choice of $\eta$ (see Theorem \ref{thm:ItDS}).

The purpose of this section is 
to provide evidence that non-constant choices of $\eta$ can give slower rates of growth of the condition number and the number of GMRES iterations than constant $\eta$. More specifically, we show 
the following.
\bit
\item Under the assumption that $\|\ItDS\|_{\LtGt}\lesssim |\eta|^{-1}$ (which is plausible because of Lemma \ref{lem:generalboundeta}), the bounds in \S\ref{sec:main_results} indicate that $\cond(\Breg)$ (defined by \eqref{eq:conditionnumber}) is smaller for certain choices of $|\eta|$ that decrease with $k$ than for $|\eta|\sim 1$. This is confirmed by numerical experiments for the kite domain of Figure \ref{fig:kite}.

\item For the kite domain, when GMRES is applied to $\mathbf{M}^{-1}  \mathbf{B}_{k,\eta,R}$, the number of iterations grows more slowly for 
certain non-constant choices of $\eta$ than for constant $\eta$.
\eit 

These observations are particularly interesting because (as recalled in Remark \ref{rem:eta1}) \cite{BrElTu:12,BoTu:13} advocated that
choosing $\eta$ constant leads to a ``small number''/``nearly optimal numbers'' of GMRES iterations.

\paragraph{Bounding the condition number assuming $\|\ItDS\|_{\LtGt}\leq C|\eta|^{-1}$.}

\ble
Assume that there exists $k_0>0$ and $C>0$ such that $\|\ItDS\|_{\LtGt}\leq C|\eta|^{-1}$ for all $k\geq k_0$. Then there exists $k_1>0$ and $C'>0$ such that for all $k\geq k_1$
\beq\label{eq:cond_bound}
\cond(\Breg) \leq C' \left( |\eta|\big( \N{\DL_k}_{\LtGt} + 1 \big) + \log k \right) \left( k\N{\NtD}_{L^2(\Gamma)\rightarrow H^1_k(\Gamma)} + \frac{1}{|\eta|}\right).
\eeq
\ele

\bpf
$\|\Breg\|_{\LtGt}$ is bounded by a $k$- and $\eta$-independent multiple of the terms in the first set of brackets on the right-hand side of \eqref{eq:cond_bound} by the definition of $\Breg$, Corollary \ref{cor:Rbound}, and Theorem \ref{thm:Hk}. Furthermore, 
$\|\Breg^{-1}\|_{\LtGt}$ is bounded by a multiple of the terms in the second set of brackets on the right-hand side of \eqref{eq:cond_bound}
by \eqref{eq:fav_formula_reg}, the assumption $\|\ItDS\|_{\LtGt}\leq C|\eta|^{-1}$, Corollary \ref{cor:Rbound}, and the equality of norms 
\eqref{eq:NtDinter}.
\epf

\paragraph{The $k$-dependence of $|\eta|$ that minimises the upper bound in \eqref{eq:cond_bound}.}
Observe that 
\beqs
\big( a|\eta| + \log k\big) \big( b + |\eta|^{-1}\big) = |\eta| ab + a + b\log k + |\eta|^{-1}\log k
\eeqs
achieves its minimum over $|\eta|>0$ of 
\beqs
2 (a b \log k)^{1/2} + a + b \log k
\quad\text{ when } \quad
|\eta|= \left(\frac{\log k }{ab}\right)^{1/2}.
\eeqs
Therefore, the upper bound in \eqref{eq:cond_bound} is minimised when 
\beq\label{eq:opt_eta}
|\eta|\sim \left(\frac{\log k}{k\|\NtD\|_{L^2(\Gamma)\rightarrow H^1_k(\Gamma)} \N{\DL_k}_{\LtGt}}\right)^{1/2}
\eeq
with this minimum equal
\beq\label{eq:min}
2\Big( k\log k \N{\NtD}_{L^2(\Gamma)\rightarrow H^1_k(\Gamma)} \N{\DL_k}_{\LtGt}\Big)^{1/2} + \N{\DL_k}_{\LtGt} + k \log k \N{\NtD}_{L^2(\Gamma)\rightarrow H^1_k(\Gamma)}.
\eeq

From here on, we ignore all factors of $\log k$ (i.e, we set each occurrence of $\log k$ to $1$) and assume that the bounds on $\N{\NtD}_{L^2(\Gamma)\rightarrow H^1_k(\Gamma)}$ in Theorem \ref{thm:NtD} are sharp;
recall that the bounds on $\|\DL_k\|_{\LtGt}$ in Theorem \ref{th:K_k_bounds} are sharp modulo the factors of $\log k$ by \cite[\S3]{GaSp:19} and \cite[\S A]{HaTa:15}.

When $\Oi$ is a ball, inputting the bounds $\|\NtD\|_{L^2(\Gamma)\rightarrow H^1_k(\Gamma)}\sim k^{-2/3}$ and $\|\DL_k\|_{\LtGt}\sim 1$ into \eqref{eq:opt_eta} and \eqref{eq:min}, we see that the optimal $|\eta|$ is  $|\eta|\sim k^{-1/6}$ and the corresponding right-hand side of \eqref{eq:cond_bound} $ \sim k^{1/3}$. This is the same $k$-dependence of this right-hand side when $|\eta|\sim 1$.

When $\Oi$ is nontrapping, inputting the bounds $\|\NtD\|_{L^2(\Gamma)\rightarrow H^1_k(\Gamma)}\sim k^{-1/3}$ and $\|\DL_k\|_{\LtGt}\sim k^{1/4}$ into \eqref{eq:opt_eta} and \eqref{eq:min}, we see that the optimal $|\eta|$ is $|\eta|\sim k^{-11/24}$ and the corresponding right-hand side of \eqref{eq:cond_bound} $ \sim k^{2/3}$. However, under the choice $|\eta|\sim 1$ the right-hand side of \eqref{eq:cond_bound} $ \sim k^{11/12}$, which is larger.

In summary, these arguments indicate that the condition number of $\Breg$ may grow slower with $k$ for choices of $\eta$ that decrease with $k$ than for the standard choice that $\eta\in \Rea\setminus\{0\}$ is independent of $k$. We now investigate this numerically for the specific example of the kite of Figure \ref{fig:kite}.

\paragraph{Computation of the condition number for the kite with varying $\eta$.}

Figure \ref{fig:kite_eta} plots the computed condition number for $\eta=0.5$, $\eta=0.5k^{-1/6}$, $\eta=0.5k^{-1/3}$, and $\eta=0.5k^{-1/2}$ for $k\in $  \((5,640)\) (where the set up for these numerical experiments is as described in \S\ref{sec:82}). In particular, the condition numbers for both $\eta=0.5k^{-1/6}$ and $\eta=0.5k^{-1/3}$ are smaller than those for $\eta=0.5$, and they also grow with $k$ at a slower rate; the condition number for $\eta=0.5k^{-1/2}$ grows at the same rate with $k$ as the condition number for $\eta=0.5$. These results may seem surprising, since the arguments above indicate that the optimal $|\eta|$ for generic nontrapping $\Oi$ is $|\eta|\sim k^{-11/24}$. However, these arguments were based on the assumption that the bound $\|\NtD\|_{L^2(\Gamma)\rightarrow H^1_k(\Gamma)}\lesssim k^{-1/3}$ is sharp. The fact that the computed growth of $\|(B_{k,\eta,S_{\ri k}})^{-1}\|_{\LtGt}$ in Figure \ref{fig:kite} is lower than expected from Theorem \ref{thm:upper_bound_inverse} (see Table \ref{table:sum_results_piecewise_smooth}) indicates that $\|\NtD\|_{L^2(\Gamma)\rightarrow H^1_k(\Gamma)}$ for the kite may be smaller than $k^{-1/3}$; this would mean that (from \eqref{eq:opt_eta}) the optimal $|\eta|$ is larger than $k^{-11/24}$, which is consistent with Figure \ref{fig:kite_eta}.

\paragraph{Number of GMRES iterations for the kite with varying $\eta$.}

The left-hand plot in Figure \ref{fig:kite_GMRES} shows the number of iterations when GMRES is applied to $\mathbf{M}^{-1}  \mathbf{B}_{k,\eta,R}$ for the kite with 
$\eta=0.5k^{-\alpha}$, for $\alpha=0, 1/6,1/3,1/2$, with $k\in (5,1280)$  and incoming plane wave at angle $\pi$ to the horizontal (i.e., $\hat{a}$ in \S\ref{sec:Neumann} equals $(-1,0)$). We apply GMRES to $\mathbf{M}^{-1}  \mathbf{B}_{k,\eta,R}$, i.e., the Galerkin matrix preconditioned with the mass matrix, rather than $\mathbf{B}_{k,\eta,R}$ itself, since the former better inherits properties of the operator $\Breg$ at the continuous level; see Lemma \ref{lem:discon}.
The number of iterations is smallest for $\eta=0.5$, although the rate of growth with $k$ is smallest for $\eta=0.5 k^{-1/6}$ over the range of $k$ considered. 

In the right-hand plot, we show the number of iterations for 
$\eta=0.5$, $\eta=0.5k^{-1/6}$, $\eta=k^{-1/6}$, and $\eta=2k^{-1/6}$.
Of these choices of $\eta$, the number of iterations is now lowest for $\eta=k^{-1/6}$ for $64 \leq k\leq 1280$
 (observe that when $k=64, k^{-1/6}=0.5$) and the rate of growth of the number of iterations for $\eta=k^{-1/6}$ is lower than that for $\eta=0.5$ for $64 \leq k\leq 1280$.

\begin{figure}[h!]
        \includegraphics[width=\textwidth]{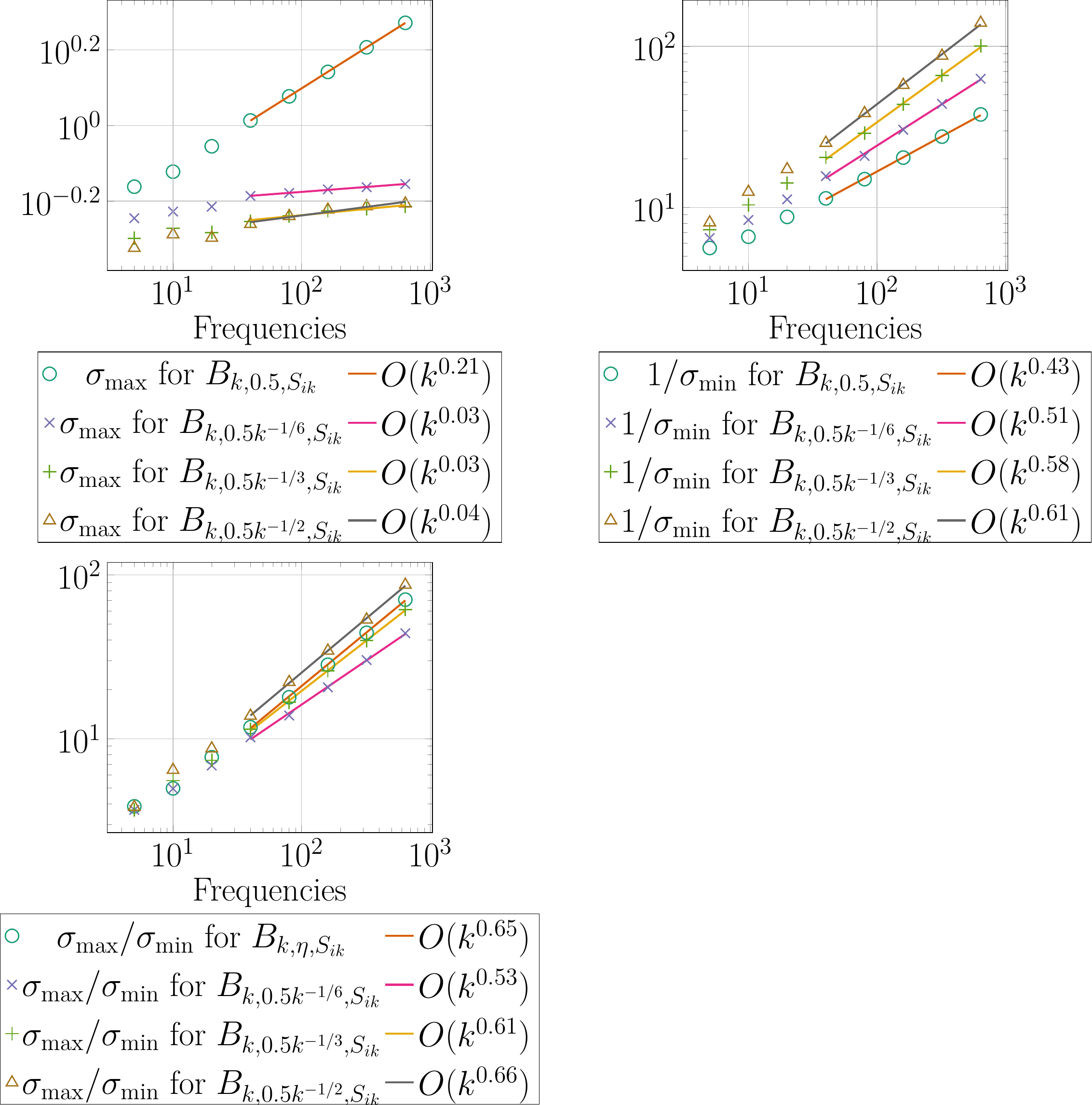}
    \caption{The computed norms and condition number of $\Breg$ for the kite with different $\eta$}\label{fig:kite_eta}
\end{figure}

\begin{figure}[h!]
\scalebox{0.5}{
        \includegraphics[width=\textwidth]{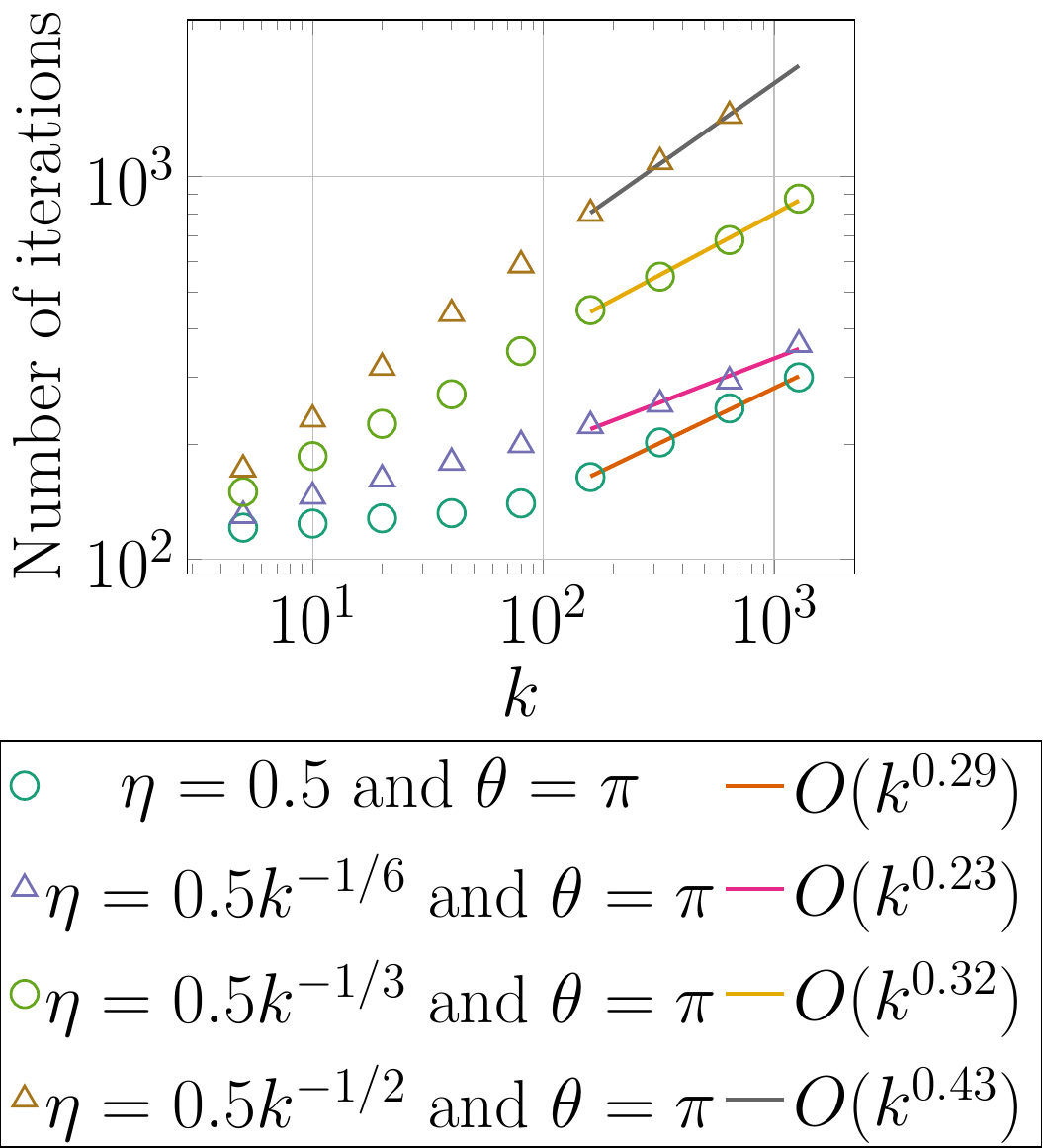}
        }
\scalebox{0.5}{
        \includegraphics[width=\textwidth]{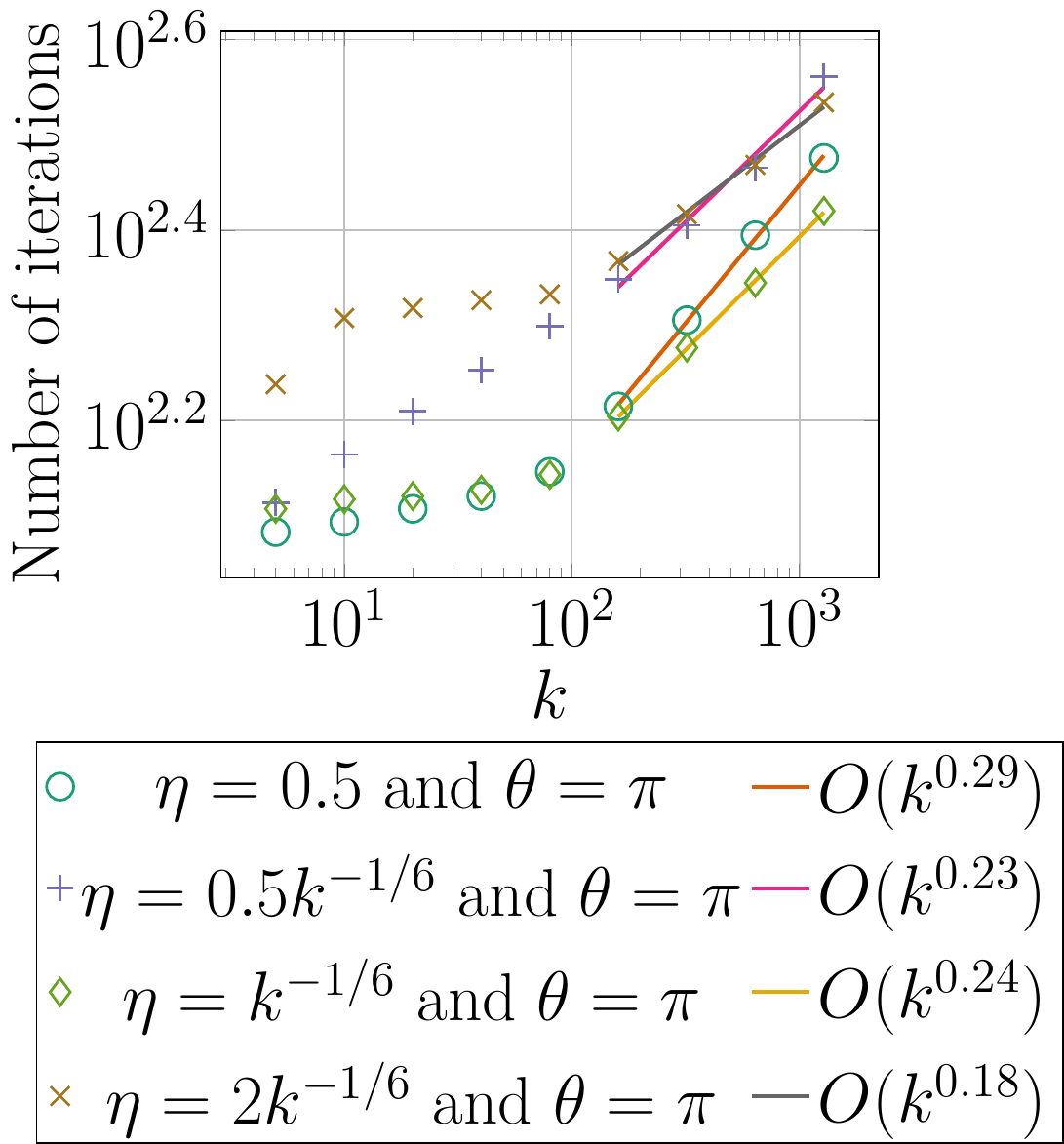}
        }
 \caption{The number of iterations when GMRES is applied to $\mathbf{M}^{-1} \mathbf{B}_{k,\eta,S_{\ri k}}$ for the kite with different $\eta$}\label{fig:kite_GMRES}
\end{figure}

\appendix

\section{Recap of layer potentials, jump relations, and Green's integral representation}\label{sec:recap}

The single-layer and double-layer potentials, $\cS_k$ and $\cK_k$ respectively, are defined for $\phi\in L^1(\Gamma)$, $\bx \in \bbR^d \setminus \Gamma$, and $k \in \Com\setminus\{0\}$, by
\begin{align}\label{eq:SLPDLP}
    \calS_k \varphi (\bx) = \int_{\Gamma} \Phi_k (\bx,\by) \varphi (\by) \dif s (\by) \quad\tand\quad     \cK_k \varphi (\bx) = \int_{\Gamma} \dfrac{\partial \Phi_k (\bx,\by)}{\partial n(\by)} \varphi (\by) \dif s (\by),
\end{align}
where $\Phi_k(x,y)$ is the fundamental solution of the Helmholtz equation defined by
\beq\label{eq:fund}
\Phi_k(x,y):= \frac{\ri}{4}\left(\frac{k}{2\pi |x-y|}\right)^{(d-2)/2}H_{(d-2)/2}^{(1)}\big(k|x-y|\big)= \left\{\begin{array}{cc}
                                                                                                            \displaystyle{\frac{\ri}{4}H_0^{(1)}\big(k|x-y|\big)}, & d=2, \\
                                                                                                            \displaystyle{\frac{\re^{\ri k |x-y|}}{4\pi |x-y|}}, & d=3,
                                                                                                          \end{array}\right.
\eeq
where $H^{(1)}_\nu$ denotes the Hankel function of the first kind of order $\nu$.
The fundamental solution of the Laplace equation is defined by
\beq\label{eq:LaplaceFund}
\Phi_0(\bx,\by):=
\displaystyle{\frac{1}{2\pi} \log \big(a|x-y|^{-1}\big),}  \quad d= 2, \qquad 
:= \dfrac{1}{(d-2)C_d |\bx-\by|^{d-2}}, \quad d\geq 3,
\eeq
where $C_d$ is the surface area of the unit sphere $S^{d-1}\subset \Rea^d$ and  $a\in \Rea$.
If $u$ is the solution to the scattering problem \eqref{eq:Helmholtz}, then Green's integral representation implies that, for $\bx \in \Oe$, 
\begin{align}
    u(\bx) = u^I(\bx) +\cK_k \big(\gamma^+  u\big)(\bx)- \calS_k \big(\partial_{n}^+ u\big) (\bx)
    = u^I(\bx) +\cK_k \big(\gamma^+  u\big)(\bx);\label{eq:Green}
\end{align}
see, e.g., \cite[Theorems 2.21 and 2.43]{ChGrLaSp:12}.
The potentials \eqref{eq:SLPDLP} are related to the integral operators \eqref{eq:SD'} and \eqref{eq:DH} via the jump relations
\beq\label{eq:jumprelations}
\gamma^\pm \cS_k = S_k, \quad \partial^\pm_n \cS_k = \mp \frac{1}{2}I + \DL_k',
\quad \gamma^\pm \cK_k = \pm \frac{1}{2}I + \DL_k, \quad \partial^\pm_n \cK_k = H_k;
\eeq
see, e.g., \cite[\S7, Page 219]{Mc:00}.
We recall the mapping properties
(see, e.g., \cite[Theorems 2.17 and 2.18]{ChGrLaSp:12}), valid when $\Gamma$ is Lipschitz, $k\in \Com$, and $|s|\leq 1/2$,
\begin{align}\nonumber
S_k : H^{s-1/2}(\Gamma)\rightarrow H^{s+1/2}(\Gamma), \quad\qquad& H_k: H^{s+1/2}(\Gamma)\rightarrow H^{s-1/2}(\Gamma),\\
\DL_k : H^{s+1/2}(\Gamma)\rightarrow H^{s+1/2}(\Gamma), \quad\qquad& \DL_k': H^{s-1/2}(\Gamma)\rightarrow H^{s-1/2}(\Gamma).
\label{eq:mapping}
\end{align}

\section{Geometric definitions}\label{sec:geo}

\begin{definition}[Nontrapping]\label{def:nontrapping}
$\Oi\subset \Rea^d$ is \emph{nontrapping} if $\bound$ is $C^\infty$ and,
given $R$ such that $\overline{\Oi}\subset B_R(\bze)$, there exists a $T(R)<\infty$ such that 
all the billiard trajectories (in the sense of Melrose--Sj{\"o}strand~\cite[Definition 7.20]{MeSj:82})
that start in $\Oe\cap B_R(\bze)$ at time zero leave $\Oe\cap B_R(\bze)$ by time $T(R)$.
\end{definition}

\begin{definition}[Smooth hypersurface]\label{def:sh}
$\Gamma\subset \R^d$ is a {\em smooth hypersurface} if there exists $\Gammaext$, a compact, embedded,
 smooth, $(d-1)$-dimensional submanifold of $\R^d$, possibly with boundary,
 such that $\Gamma$ is an open subset of $\Gammaext$, with $\Gamma$ strictly away from $\partial \Gammaext$, and the boundary of $\Gamma$ can be written as a disjoint union
\beqs
\partial \Gamma=\left(\bigcup_{\ell=1}^n Y_\ell\right)\cup \Sigma,
\eeqs
where each $Y_\ell$ is an open, relatively compact, smooth embedded manifold of dimension $d-2$ in $\Gammaext$, $\Gamma$ lies locally on one side of $Y_\ell$, and  $\Sigma$ is closed set with $d-2$ measure $0$ and $\Sigma \subset \overline{\bigcup_{l=1}^nY_l}$. We then refer to the manifold $\Gammaext$ as an extension of $\Gamma$. 
\end{definition}
\noi For example, when $d=3$, the interior of a 2-d polygon is a smooth hypersurface, with $Y_i$ the edges and $\Sigma$ the set of corner points.
\begin{definition}[Curved]\label{def:curved}
A smooth hypersurface is \emph{curved} if there is a choice of normal so that the second fundamental form of the hypersurface is everywhere positive definite.
\end{definition}

Recall that the principal curvatures are the eigenvalues of the matrix of the second fundamental form in an orthonormal basis of the tangent space, and thus ``curved'' is equivalent to the principal curvatures being everywhere strictly positive (or everywhere strictly negative, depending on the choice of the normal).

\begin{definition}[Piecewise smooth]\label{def:piecewisesmooth}
A hypersurface $\Gamma$ is \emph{piecewise smooth} if $\Gamma=\cup_{i=1}^N \overline{\Gamma}_i$ where $\Gamma_i$ are smooth hypersurfaces 
and $\Gamma_i\cap \Gamma_j=\emptyset.$
\end{definition}

\begin{definition}[Piecewise curved]\label{def:piecewisecurved}
A piecewise-smooth hypersurface $\Gamma$ is \emph{piecewise curved} if $\Gamma$ is as in Definition \ref{def:piecewisesmooth} and each $\Gamma_j$ is curved.
\end{definition}

\section{Extension of the results to the exterior impedance problem}\label{sec:ext_imp}

\paragraph{The exterior impedance problem and associated boundary integral equations.}
Just as for the Neumann problem, we consider the plane-wave scattering problem for simplicity.
Given $u^I(x) = \exp(\ri k x \cdot \hat{a})$ for $\hat{a}\in \Rea^d$ with $|\hat{a}|_2=1$, $k>0$, and $\beta \in L^\infty(\Gamma)$ with $\Re \beta\geq 0$, let $u\in H^1_{\rm loc}(\Oe)$ be the solution of 
\begin{align}\label{eq:Helmholtz_imp}
\Delta u + k^2 u =0 \quad\tin\quad \Oe ,\qquad 
\partial_n^+ u + \ri \beta \gamma^+ u  = 0 \quad\ton\quad\Gamma,
\end{align}
and $u^S:= u- u^I$ satisfies the radiation condition \eqref{eq:src}; the solution of this problem is unique by, e.g., \cite[Corollary 2.9]{ChGrLaSp:12}.
Then
\begin{align}\label{eq:direct_imp}
\Bregimp \gamma^+ u:= 
\left[    \Breg
    + \ri \beta \left( \Reg \left(\half I+ \DL_k'\right) - \ri \eta S_k\right) \right] 
         \gamma^+  u  = \ri \eta \gamma^+ u^I  - \Reg \partial_{n}^+ u^I
\end{align}
(which reduces to \eqref{eq:direct} when $\beta=0$).
Indeed, expressing $u$ via Green's integral representation (see the first equality in \eqref{eq:Green}) and using the impedance boundary condition in \eqref{eq:Helmholtz_imp} yields.
\beq\label{eq:Green_imp}
u = u^I + (\cK_k + \ri \beta \cS_k)\gamma^+ u.
\eeq
Taking Dirichlet and Neumann traces of this expression and using the jump relations in \eqref{eq:jumprelations} yields two integral equations, which combined (after using the impedance boundary condition from \eqref{eq:Helmholtz_imp}) give \eqref{eq:direct_imp}.
Furthermore, if $\phi$ satisfies
\begin{align}\label{eq:indirect_imp}
\Bregimp' \phi:= 
\left[\Breg' + \ri \beta \left( \left(\frac{1}{2} I + \DL_k\right)\Reg - \ri \eta S_k \right)\right] 
\phi  = -\partial_n^+ u^I - \ri \beta \gamma^+ u^I,
\end{align}
(which reduces to \eqref{eq:indirect} when $\beta=0$) then, by the jump relations \eqref{eq:jumprelations}, $u=u^I+ (\cK_k  \Reg-\ri \eta \cS_k)\phi$ is a solution of \eqref{eq:Helmholtz_imp} and \eqref{eq:src}.

\paragraph{Bounds on the norms of $\Bregimp$ and $\Bregimp'$.}
Bounds analogous to those in Theorem \ref{thm:upper_bound_norm} on the norms of the operators on the left-hand sides of the BIEs \eqref{eq:direct_imp} and \eqref{eq:indirect_imp}
follow by arguing as in the proof of Theorem \ref{thm:upper_bound_norm} and using, in addition, the bounds on $\|S_k\|_{\LtGt}$ from 
\cite[Theorem 1.2]{GaSm:15}, \cite[Appendix A]{HaTa:15}, \cite[Chapter 4]{Ga:19} (with these bounds sharp up to a factor of $\log k$).

\paragraph{Bounds on the norms of $(\Bregimp)^{-1}$ and $(\Bregimp')^{-1}$.} Theorem \ref{thm:invert} holds with $\Breg$ and $\Breg'$ replaced by $\Bregimp$ and $\Bregimp'$, and the proof is essentially identical.
The upper bounds on the norms of $(\Bregimp)^{-1}$ and $(\Bregimp')^{-1}$ in Theorem \ref{thm:upper_bound_inverse} are based on Lemma \ref{lem:fav_formula_reg}. To state the analogue of Lemma \ref{lem:fav_formula_reg} for $\Bregimp$ and $\Bregimp'$, we 
let $\ItDext$ be the impedance-to-Dirichlet map for the exterior impedance problem 
\beqs
\Delta u + k^2 u =0 \quad\tin \Oi ,\qquad \partial_n^+ u + \ri \beta \gamma^+ u = g \quad\ton \Gamma,
\eeqs
and $u$ satisfies the radiation condition \eqref{eq:src}; i.e., $\ItDext$ is the map $g\mapsto \gamma^+u$.

\begin{lemma}\label{lem:fav_formula_reg_imp}
    \begin{align}\label{eq:fav_formula_reg_imp}
        (\Bregimp)^{-1}        
        = \ItDext \Reg^{-1} - \Big(I-\ItDext \big(\ri\eta \Reg^{-1} + \ri \beta\big)\Big)\ItDR
    \end{align}
    and
    \begin{align}\label{eq:fav_formula_reg2_imp}
(\Bregimp')^{-1}
        =  \Reg^{-1}\ItDext - \Reg^{-1}\ItDR \Big( \Reg - \big( \ri \eta I + \ri \beta \Reg) \ItDext\Big).
    \end{align}
    \end{lemma}

\bpf
By \eqref{eq:NtDdual}, if $u$ and $v$ are solutions of the Helmholtz equation in $\Oe$ satisfying the radiation condition \eqref{eq:src}, then
\beqs
\big\langle \partial_n^+ u + \ri \beta \gamma^+ u , \gamma^+ v\rangle_{\Gamma,\Rea} = \big\langle
 \partial_n^+ v + \ri \beta \gamma^+ v, \gamma^+ u \rangle_{\Gamma,\Rea}, 
\eeqs
and thus $\ItDext = (\ItDext)'$. The formula \eqref{eq:fav_formula_reg2_imp} follows from \eqref{eq:fav_formula_reg_imp} 
by taking the $^\prime$ of \eqref{eq:fav_formula_reg_imp} and arguing exactly as in the start of the proof of Lemma \ref{lem:fav_formula_reg}. 
To prove \eqref{eq:fav_formula_reg_imp}, 
given $g$ and $\varphi$ satisfying  $\Bregimp \varphi= g$,
let  \(u :=             (\cK_k + \ri \beta \cS_k)
         \varphi\); the motivation for this choice is that we are dealing with the direct BIE arising from Green's integral representation where, by \eqref{eq:Green_imp}, $u$ is the sum of $u^I$ and this particular combination of potentials. 
Exactly as in the proof of Lemma \ref{lem:fav_formula_reg}, the jump relations imply that $\Reg \partial^-_n u - \ri \eta \gamma^- u = \Bregimp\varphi =g$, and then  analogous arguments to those in the proof of Lemma \ref{lem:fav_formula_reg} give the result  \eqref{eq:fav_formula_reg_imp}.
\epf

Given the bounds on $\Reg$ and $\Reg^{-1}$ from Corollary \ref{cor:Rbound} and the bounds on $\ItDS$ from 
Corollary \ref{cor:ItDS} and Theorem \ref{thm:ItDS}, 
bounding $(\Bregimp)^{-1}$ and $(\Bregimp')^{-1}$ reduces to bounding $\ItDext$.
Indeed, arguing exactly as in the proof of Theorem \ref{thm:upper_bound_inverse}, we find that if $\eta\in \mathbb{R}\setminus\{0\}$ is independent of $k$ then
\begin{align}\label{eq:hot_imp}
\big\| (\Bregimp)^{-1}\big\|_{L^2(\Gamma)\rightarrow L^2(\Gamma)} \leq C \Big( 1+ \big((1+\lvert \eta \rvert)k + \beta\big) \big\| \ItDext\big\|_{H^{-1}_k(\Gamma)\rightarrow L^2(\Gamma)}\Big)
\end{align}
(compare to \eqref{eq:hot1}).

We now consider the most-common exterior impedance boundary condition where $\beta =k$.

\ble\label{lem:ItDext}
If $\Gamma$ is $C^\infty$ and $\beta=k$, then given $k_0>0$ there exists $C>0$ such that
\beq\label{eq:ItDext}
\big\| \ItDext\big\|_{H^{s-1/2}_k(\Gamma)\to H^{s+1/2}_k(\Gamma)} \leq C \quad\tfa k\geq k_0 \tand |s|\leq 1/2.
\eeq
\ele

When $\Oi$ is a ball and $u(r,\theta)= H_0^{(1)}(kr)$ (in $d=2$), standard Hankel-function asymptotics for large argument (see, e.g., \cite[\S10.17]{Di:22}) show that \eqref{eq:ItDext} is sharp in its $k$-dependence (see \cite[Lemma 5.5]{BaSpWu:16} for analogous arguments for the interior impedance to Dirichlet map).

\bpf[Proof of Lemma \ref{lem:ItDext}]
By repeating the proof of Lemma \ref{lem:NtD1} with 
 \eqref{eq:NtDdual2} replaced by \eqref{eq:ItDdual}, we see that 
the analogue of Lemma \ref{lem:NtD1} holds with $\NtD$ replaced by $\ItDext$. It is therefore sufficient to prove that 
\beq\label{eq:BSW1}
\big\| \ItDext\big\|_{L^2(\Gamma)\to H^{1}_k(\Gamma)} \leq C \quad\tfa k\geq k_0.
\eeq

The arguments in \cite[Proof of Lemma 3.2]{GaMaSp:21} can be used to show that given $R>0$ and $k_0>0$ there exists $C>0$ such that if $(\Delta +k^2)u=0$ in $\Oe$ and $u$ satisfies the Sommerfeld radiation condition, then
\beq\label{eq:GMS1}
\N{u}_{L^2(\Oe \cap B_R)} \leq C \N{ k^{-1}\partial_n^+ u + \ri \gamma^+ u }_{L^2(\Gamma)} \quad\tfa k\geq k_0
\eeq
(compare to \cite[Equation 3.23]{GaMaSp:21}). We first show how \eqref{eq:BSW1} follows from \eqref{eq:GMS1}; we then discuss how to prove \eqref{eq:GMS1}.

Once we have proved \eqref{eq:GMS1}, an argument using the real part of Green's identity shows that 
\beqs
k^{-1} \N{\nabla u}_{L^2(\Oe\cap B_R)} + \N{u}_{L^2(\Oe \cap B_R)} \leq C \N{ k^{-1}\partial_n^+ u + \ri \gamma^+ u }_{L^2(\Gamma)} \quad\tfa k\geq k_0;
\eeqs
see, e.g., \cite[Equation 4.8]{GaLaSp:21}, \cite[Lemma 2.2(b)]{Sp:14}. Then, an argument using the imaginary part of Green's identity shows that 
\beqs
k^{-1} \N{\partial_n^+ u }_{L^2(\Gamma)} + \N{\gamma^+ u }_{L^2(\Gamma)} \leq 2 \N{k^{-1} \partial^+_n u + \ri \gamma^+ u}_{L^2(\Gamma)};
\eeqs
see, e.g., \cite[Lemma 4.14]{Sp:14}. 
Finally, an argument using a Rellich-type identity shows that 
\beq\label{eq:GMS4}
k^{-1}\N{\nabla_\Gamma (\gamma^+ u) }_{L^2(\Gamma)} \leq C \Big( k^{-1} \N{\partial_n^+ u }_{L^2(\Gamma)} + \N{\gamma^+ u }_{L^2(\Gamma)} 
+k^{-1} \N{\nabla u}_{L^2(\Oe\cap B_R)} + \N{u}_{L^2(\Oe \cap B_R)} \Big),
\eeq
see \cite[\S5.1.2, 5.2.1]{Ne:67}, \cite[Lemma 3.5(ii)]{Sp:14}, where recall from \S\ref{sec:weighted_norms} that $\nabla_\Gamma$ is the surface gradient.
Combining \eqref{eq:GMS1}-\eqref{eq:GMS4}, we obtain \eqref{eq:BSW1}.

We now describe how to modify the proof of \cite[Equation 3.23]{GaMaSp:21} (part of \cite[Lemma 3.2]{GaMaSp:21}) to prove \eqref{eq:GMS1}. 
\cite[Lemma 3.2]{GaMaSp:21} considers the more-general impedance-type boundary condition $-k^{-1}Q_b \partial_n^+ u + \ri \gamma^+u =0$, whereas now we take $Q_b=-1$ (note that in \cite{GaMaSp:21} the sign convention with the normal is different to here; see \cite[Page 6724]{GaMaSp:21}).
\cite[Lemma 3.2]{GaMaSp:21} establishes estimates for $k$ in a strip in the complex plane, whereas here we are only interested in real $k$, and therefore we take $C=0$ at the start of \cite[Proof of Lemma 3.2]{GaMaSp:21}. Next, the bound \cite[Equation 3.23]{GaMaSp:21} is similar to \eqref{eq:GMS1} except with the $L^2(\Gamma)$ norm on the right-hand side replaced by the $H^{3/2}(\Gamma)$ norm; this is because \cite[Lemma 3.2]{GaMaSp:21} ultimately wants to bound the $H^2$ norm of $u$ in the domain. Since we are not interested in this $H^2$ norm, we can replace the $H^{3/2}(\Gamma)$ norm by the $L^2(\Gamma)$ norm and the arguments of \cite[Proof of Lemma 3.2]{GaMaSp:21} go through unchanged. Finally, the arguments in \cite[Proof of Lemma 3.2]{GaMaSp:21} concern the complex-scaled operator $\Delta_\theta +k^2$, whereas we want to now use them with the original Helmholtz operator $\Delta +k^2$. These arguments go through with now 
(i) the fact that the defect measure of the incoming set is zero (\cite[Proposition 3.5]{Bu:02}, \cite[Lemma 3.4]{GaSpWu:20}, \cite[Lemma 3.6]{GaLaSp:21}) replacing that $\mu( \{ {\bf r}\geq 2r_1\})=0$ at the bottom of \cite[Page 6751]{GaMaSp:21}, and (ii) propagation now possible in both directions along the flow (as opposed to only forwards for the complex-scaled operator -- see the bottom of \cite[Page 6751]{GaMaSp:21}).
\epf

Combining \eqref{eq:hot_imp} and \eqref{eq:ItDext}, we obtain the following bound on $(\Bregimp)^{-1}$.

\begin{corollary}
Suppose that $\Gamma$ is $C^\infty$, $\beta=k$, and $\eta\in \mathbb{R}\setminus \{0\}$ is independent of $k$. Then given $k_0>0$ there exists $C>0$ such that
\begin{align*}
\big\| (\Bregimp)^{-1}\big\|_{L^2(\Gamma)\rightarrow L^2(\Gamma)} \leq C ( 1+ k) \quad \tfa k\geq k_0.
\end{align*}
\end{corollary}

\section*{Acknowledgements}
EAS thanks Zydrunas Gimbutas (NIST) and Leslie Greengard (New York University and Flatiron Institute) for useful discussions 
about the operators $\Breg$ and $\Bregp$ during a visit to New York University in November 2012.
The authors thank the anonymous referee for their careful reading of the paper and constructive comments.
PM and EAS were supported by EPSRC grant EP/R005591/1  and JG by EPSRC grant EP/V001760/1.
The authors have no competing interests to declare that are relevant to the content of this article.

\footnotesize{
\bibliographystyle{acm}
\bibliography{biblio_combined_sncwadditions.bib}
}

\end{document}